\documentclass[11pt,leqno]{amsart}
\usepackage[backend=bibtex, citestyle=alphabetic, bibstyle=alphabetic, eprint=false, url=false, maxbibnames=5, maxcitenames=5]{biblatex}
\allowdisplaybreaks
\usepackage{mathrsfs}
\usepackage{amsmath, amssymb, mathtools, xfrac}
\usepackage{csquotes}
\usepackage{abstract}
\usepackage[hidelinks]{hyperref}
\usepackage{marginnote}
\usepackage{xcolor}
\usepackage{graphicx}%To scale align blocks in appendix
\usepackage{orcidlink}

\usepackage{etoolbox}
\makeatletter
\patchcmd{\@maketitle}{\newpage}{}{}{} 
\makeatother
\numberwithin{equation}{section}

\theoremstyle{definition}
\newtheorem{definition}{Definition}[section]
\newtheorem{example}[definition]{Example}
\newtheorem{remark}[definition]{Remark}
\theoremstyle{plain}
\newtheorem{theorem}[definition]{Theorem}
\newtheorem{lemma}[definition]{Lemma}
\newtheorem{corollary}[definition]{Corollary}
\newtheorem{prop}[definition]{Proposition}

\newtheorem{ass}[definition]{Assumption}

\newcommand{\C}{\mathbb{C}}

\newcommand{\g}{\overline{g}}

\renewcommand{\j}{\jmath}

\newcommand{\M}{\overline{M}}

\newcommand{\N}{\mathbb{N}}

\newcommand{\R}{\mathbb{R}}

\newcommand{\T}{\mathbb{T}}

\newcommand{\Z}{\mathbb{Z}}
\renewcommand{\epsilon}{\varepsilon}

\newcommand{\del}{\partial}
\newcommand{\Lap}{\Delta}
\renewcommand{\div}{\text{\normalfont{div}}}
\newcommand{\curl}{\text{\normalfont{curl}}}

\newcommand{\numberthis}{\addtocounter{equation}{1}\tag{\theequation}}

\renewcommand{\theequation}{\arabic{section}.\arabic{equation}}

\usepackage{bm}

\title{Landau damping on expanding backgrounds}
\author{David Fajman, Liam Urban}
\usepackage[foot]{amsaddr}
\address{
\begin{tabular}[h]{l@{\extracolsep{8em}}l} 
David Fajman  & Liam Urban \\
Faculty of Physics & Faculty of Physics\\ 
University of Vienna & University of Vienna \\
Boltzmanngasse 5 & Boltzmanngasse 5\\
1090 Vienna, Austria & 1090 Vienna, Austria\\
david.fajman@ univie.ac.at & liam.urban@ univie.ac.at \\
\orcidlink{0000-0003-3034-6232}\ 0000-0003-3034-6232 & \orcidlink{0000-0001-9185-9627}\ 0000-0001-9185-9627
\end{tabular}
}
\begin{document}

\maketitle

\begin{abstract}
We analyse the effect of expansion in Newtonian cosmology on the asymptotic behaviour of charged self-interacting plasmas close to Poisson equilibria. To this end, we study the Vlasov-Poisson system on the phase space of a $3$-torus which is expanding with respect to the scale factor $a(t)$. We show that, for $a(t)=t^q$ with $q\in(0,\frac12)$, solutions to this system exhibit nonlinear Landau damping for initial data that is small with respect to a suitably strong Gevrey class, i.e., the charge density contrast of the plasma decays superpolynomially. For larger choices of $q$ within this range, the initial data requirements become stricter while the decay weakens. To our knowledge, this is the first result showing Landau damping in a cosmological setting. 
\end{abstract}

\section{Introduction}\label{sec:intro}

We consider the Vlasov-Poisson system on $\T^3\times\R^3$ with repulsive particle interaction and near a nontrivial charged expanding Poisson equilibrium $f_0:(0,\infty)\times\R^3\longrightarrow(0,\infty)$ with
\begin{subequations}
\begin{equation}\label{eq:distr-shape}
f_0(t,v)=\mu(a(t)\,v)=\tilde{\mu}(a(t)\,\lvert v\rvert)\,.
\end{equation}
Here, $a:(0,\infty)\rightarrow(0,\infty)$ is a smooth, strictly increasing function that describes the growth of length scales in time, and the Poisson background is given by
\begin{equation}
\tilde{\mu}(\lvert v\rvert)=\frac1{\pi^2}\,\frac{\theta_0}{(\theta_0^2+\lvert v\rvert^2)^{2}}
\end{equation}
for some $\theta_0>0$. The equilibrium is normalised such that $\int_{\R^3}\mu(v)dv=1$, and thus the total charge density $\underline{\rho}_0$ of the matter distribution $f_0$ satisfies
\begin{equation}
\underline{\rho}_0(t,x)=\int_{\R^3} f_0(t,x,v)dv=a(t)^{-3}\,.
\end{equation}
\end{subequations}
Let $g:(0,\infty)\times\T^3\times\R^3\longrightarrow\R$ denote the deviation of the total particle distribution $f=f_0+g\geq 0$ from the Poisson background. Then, the Vlasov-Poisson equations for $g$ take the following form.
\begin{subequations}\label{eq:perturbation-evol}
\begin{gather}
\del_tg+\frac{v^i}{a}\del_{x^i}g-\frac{\dot{a}}a\,v^i\del_{v^i}g-\,a^{-1}(\del_xW)\cdot(\del_v)g-a^{-1}(\del_xW)\cdot (\del_vf_0)=0\,,\label{eq:perturbation-evol-landau}\\
 \Lap W=-4\pi\,a^2\,\underline{\rho},\quad \underline{\rho}=\int_{\R^3} g(\cdot,\cdot,v)dv\,.\label{eq:perturbation-evol-poiss}
\end{gather}
\end{subequations}
%\begin{subequations}
%\begin{equation}\label{eq:distr-shape}
%f_0(t,v)=\tilde{\mu}(a(t)^2\,\lvert v\rvert^2)\,,
%\end{equation}
%where $\tilde{\mu}$ is a positive analytic function such that, for
%\begin{equation}
%\mu(v)=\tilde{\mu}(\lvert v\rvert^2)=f_0(t,a(t)^{-1}v)\,
%\end{equation}
%one has $\int_{\R^3}\mu(v)dv=1$. We then write
%\begin{equation}
%\rho_0(t,x)=\int_{\R^3} f_0(t,x,v)dv=a(t)^{-3}
%\end{equation}
%\end{subequations}

In addition, we assume that the perturbation is charge-neutral initially and thus remains so throughout, i.e., for an initial distribution $\mathring{g}(x,v)=g(t_0,x,v)$ at the initial time $t_0>0$, we assume

\begin{equation}\label{eq:perturbation-norm}
\int_{\T^3}\int_{\R^3} \mathring{g}(x,v)dxdv=\int_{\T^3}\underline{\rho}(t_0,x)dx=0\,.
\end{equation}

By this convention, $\sfrac{\underline{\rho}}{\underline{\rho}_0}$ describes the density contrast of the plasma distributed along $f$.\\

The system \eqref{eq:perturbation-evol} describes the behaviour of the collisionless charged gas $f$ and its deviation $g$ from the Poisson equilibrium $f_0$ in coordinates $(t,x,v)$ adapted to a comoving observer subject to expansion.\footnote{For a comparison of this system with relativistic Vlasov(-Maxwell) equations on a fixed Friedman-Lema\^\i tre-Robertson-Walker (FLRW) background, we refer to Remark \ref{rem:nonrel-limits}.} Similar systems for attractive particle interaction arise in Newtonian cosmology by rewriting the classical Vlasov-Poisson system in comoving coordinates, see for example \cite[Chapter II, Section 7ff.]{Pee80}. In these settings, one can consider the case where the rate of expansion is given by $\underline{\rho}_0$, which naturally restricts the expansion rates that can arise. To isolate the effects of introducing expansion to the Vlasov-Poisson system near a fixed background, we decouple the expansion rate from Vlasov matter via the standard approch of fixing the rate of expansion with an external, unperturbed source. The derivation of such a system from a classic Vlasov-Poisson system with an external field is contained in Appendix \ref{sec:newt-cosmo}. To obtain \eqref{eq:perturbation-evol}, one could, for example, assume that there is an additional background Poisson equilibrium for particles of opposite charge to cancel the effects of the plasma itself, as well as a perfect fluid with suitable linear equation of state to determine the rate of expansion. We elaborate on this in Remark \ref{rem:scale-factor-appendix}.\\

Our main result for small data solutions to \eqref{eq:perturbation-evol} can be summarised as follows.

\begin{theorem}[Rough statement of the main result]\label{thm:main-intro}
Let $a(t)=t^q$ for $q\in(0,\frac12)$ and let $\mathring{g}$ be small with respect to a sufficiently strong Gevrey norm. Then, the density contrast $\sfrac{\underline{\rho}}{\underline{\rho}_0}$ and the relative force $\nabla W$ induced by $g$ decay at a superpolynomial rate depending on the Gevrey class, and always stronger than $\exp\left(-c\,t^{\alpha}\right)$ for some $c>0$ and $\alpha=\sfrac{1-2q}3$.\\
Within the admissible range, the larger the rate of expansion is, the stronger the Gevrey class must be chosen, and the weaker our upper bound on the decay rate is.
\end{theorem}

A precise statement can be found in Theorem \ref{thm:main}. For $q=0$, \eqref{eq:perturbation-evol} is the classical Vlasov-Poisson system, in which case the phenomenon described in Theorem \ref{thm:main} is known as Landau damping, as discussed in Subsection \ref{subsec:lit}. To our knowledge, this work is the first result that establishes Landau damping in an expanding regime. 

%In the following, we consider the Poisson background $\mu(v)=\tilde{\mu}(\lvert v\rvert)$ given by
%\[\tilde{\mu}(r)=\frac1{\pi^2}\cdot \frac{\theta_0}{(\theta_0^2+\lvert r\rvert^2)^{2}}\,.\]
%Note that this is normalised such that the background has total mass 1, and that the Fourier transform of this satisfies
%\[\hat{\mu}(\xi)=2\,\exp(-\theta_0\,\lvert \xi\rvert)\]
%in our Fourier transform convention.\\
%We show that, if the expansion rate is sufficiently weak, the classical phenomenon of Landau damping extends to the expanding regime near Poisson backgrounds in the following sense.

\subsection{Connection to previous work}\label{subsec:lit}

\subsubsection{The classical Vlasov-Poisson system}

Solutions $f:(0,\infty)\times\T^d\times\R^d\longrightarrow\R$  to the free transport equation 
\[\del_t f(t,x,v)+v^i\del_{x^i}f(t,x,v)=0\]
are given by
\[f(t,x,v)=\mathring{f}(x-tv,v)\]
for a given initial distribution $\mathring{f}$ at $t=0$. Thus, the Fourier transform of the associated density $\underline{\rho}$ satisfies
\[\hat{\underline{\rho}}_k(t)=\hat{f}(t,k,0)=\hat{\mathring{f}}(t,k,kt).\]
For $k\neq 0$, it follows that $\hat{\underline{\rho}}_k$ decays polynomially if $\mathring{f}$ is Sobolev-regular in momentum space, and exponentially if it is analytic in momentum space. More precisely, the distribution weakly converges to its spatial average in $L^2(\T^d)$, i.e., 
\[f(t,\cdot,v)\rightharpoonup (2\pi)^{-d}\int_{\T^d} \mathring{f}(x,v)\,dx\quad \text{as}\quad t\to\infty,\]
and in this sense, the solution homogenizes. This effect is a special case of the more general phenomenon of \enquote{phase mixing}.\\

Landau damping describes a phenomenon for collisionless gases, first truly demonstrated by Landau in \cite{Lan46sow,Lan46en}, improving upon initial observations by Vlasov, see \cite{Vlasov45}. In short, it refers to strong decay of the electromagnetic field generated by the perturbation of a charged plasma near an equilibrium, or equivalently of non-zero modes of the charge density of a plasma close to an equilibrium. From the modern mathematical viewpoint discussed below, it is best understood as a specific phase mixing effect for the Vlasov-Poisson system, and not necessarily restricted to plasmas, see \cite{MV10,BMM16}. The most general condition for Landau damping to occur is if the background $\mu$ satisfies the Penrose stability condition, originating from \cite{Pen60}, namely
\begin{equation}\label{eq:penrose-classic}
\inf_{k\in\Z^d\backslash\{0\}}\inf_{\lambda\in\C,\Re(\lambda)\geq 0}\left\lvert 1-4\pi\epsilon_F\int_0^\infty e^{-\lambda\,s}s\,\hat{\mu}(ks)\,ds\right\rvert\geq \kappa>0\,.
\end{equation}
Here, $\hat{\mu}$ denotes the Fourier transform of $\mu$ and $\epsilon_F=\pm 1$ is positive for attractive/gravitational particle interaction and negative for repulsive/electrostatic interaction. If this condition holds, then non-zero modes of the particle density decay exponentially for large times, and do so uniformly in $k$. In the electrostatic setting for $d\geq 3$, the Penrose condition holds for any positive radial background, see for example \cite[p.127]{LP81} and \cite[p.38f.]{Vil10}. For a more in-depth introduction to the relationship between phase mixing for the free Vlasov equation and Landau damping, we refer to \cite{Vil10, Bed22rev}.\\
%This well-known homogenization effect is due to phase mixing, see also \cite[Section 3.0]{Vil10}. In this vein, Landau damping refers to the specific phase mixing effect for solutions to the Vlasov-Poisson system near equilibria.

The first full proof of nonlinear Landau damping is due to Mouhot and Villani in their breakthrough work \cite{MV10} for analytic perturbations, assuming the Penrose condition \eqref{eq:penrose-classic} to hold. In particular, in \cite[Proposition 2.1]{MV10}, it is shown that the Penrose condition unifies various other stability conditions previously present in the literature. That this result extends to $(\gamma^{-1})$-Gevrey regular perturbations with $\gamma>\frac13$, as conjectured in \cite[Section 13]{MV10} via a heuristic nonlinear analysis, has since been rigorously established over the course of multiple works that also greatly simplify the original proof, such as \cite{BMM16,GNR21}. In addition, the effect has recently even been demonstrated to hold in the Gevrey-critical regime $\gamma=\frac13$ in \cite{IPWW24}. This still rather strong regularity requirement seems necessary as Sobolev regularity is generally insufficient to control nonlinear resonances, often referred to as plasma echoes, see \cite{Bed21}. However, these resonances can be controlled on fixed time scales or for certain specific solution families, see \cite{GNR22,Zil25}. Finally, the explicit structure of the Poisson equilibrium has proven to be especially useful to study more difficult settings, as it was used in \cite{IPWW24R3} to show Landau damping on $\R^3\times\R^3$. Therein, it provided a direct method to analyse the effects of dispersion for low frequencies where an analogue of the Penrose condition would necessarily be violated.

\subsubsection{Newtonian gravity and cosmology}\label{subsec:newt}

To the extent that phase mixing effects have been studied in Newtonian gravity, this has largely been in the context of isolated bodies, where the Vlasov-Poisson system is coupled to an additional external field that simulates the gravitational effect of a large body, e.g., a Keplerian potential modeling an isolated star in \cite{CL24,HS25} or describing the potential well of a spherically symmetric black hole in \cite{RS20}. We also refer to the textbook \cite[Chapters 4f.]{BT08} for a broader overview on the Vlasov-Poisson system in stellar dynamics.\\ %Note to self - only the first edition discusses the Plasma case of Landau damping

Models arising from Newtonian cosmology often consider the special case of \eqref{eq:perturbation-evol} where the scale factor is determined by the rate of expansion, with a linearized stability analysis similar but not identical to that of its macroscopic analogue, the expanding Euler-Poisson system; see for example \cite[Chapter II]{Pee80}. Both the Vlasov- and Euler-Poisson systems admit power law solutions with $a(t)=C\,t^\frac23$ which matches the behaviour of the scale factor in a dust-filled FLRW universe, and one expects the so-called Jeans instability to lead to gravitational collapse. For the Euler-Poisson system, rigorous results appear to match this expectation: In the recent work \cite{FMOOW25} including the first author, the expanding Euler-Poisson system was studied, where the fluid satisfies a linear equation of state and the scale factor is of the form $a(t)=t^q$. Here, it was observed that homogeneous solutions are stable toward the future for $q\in(\frac23,1)$, while for $q\leq\frac23$, numerical studies indicate that shocks form in finite time for arbitrarily small perturbations. 

Less is known for the expanding Vlasov-Poisson system. Again in the case of gravitational particle interaction where the scale factor is entirely determined by the matter density itself, global existence for compactly supported data and backgrounds with compact momentum support was shown by Rein and Rendall in \cite{RR94}. The scale factors admissible therein either cause recollapse in finite time, exhibit dust-like power law expansion or grow linearly as $t\to\infty$. In the latter case, these equilibria were also shown to be nonlinearly stable in \cite{Rei97}, in the sense that dilution ensures that the density contrast remains small if it is small initially, while the power-law solution needs to be excluded explicitly. These results have been extended to the Boltzmann-Poisson system, in particular with positive cosmological constant modeled by an external field, in \cite{Lee10,Lee13,Lee16}. In all of these cases, this extension essentially works since admissible scale factors behave asymptotically similarly with and without collisions, and in particular remain strong enough to make collisions asymptotically negligible toward the future.\\

If one wants to precisely study the effect of expansion on the damping rate as we intend, one needs to decouple it from the matter conent itself and has to consider the Vlasov-Poisson system subject to an external field driving expansion, as discussed in Appendix \ref{sec:newt-cosmo}. In particular, we note that such an external field is time-dependent, as opposed to the time-independent field caused by a single static stellar body discussed at the start of this section.\\
Having chosen such a field, Theorem \ref{thm:main} shows that repulsive particle interaction leads to nonlinear stability, and actively homogenizes due to the occurance of Landau damping, if the expansion rate is sufficiently weak rather than sufficiently strong as in \cite{Rei97,Lee16,FMOOW25}.\\
One has to expect this to break down for self-gravitating matter and in three spatial dimensions, even for slow expansion rates and the Poisson equilibrium, as discussed in Appendix \ref{subsec:gravit}. Curiously, for $d\geq 4$, Landau damping actually extends to the self-interacting regime, which we formulate and prove more precisely in Section \ref{subsec:4+}. %In Appendix \ref{sec:4+}, we outline how this is, in essence, a direct extension of the proof of Landau damping for classical Vlasov-Poisson systems, in particular \cite{GNR21}, under a slightly adapted version of the Penrose condition. This is not restricted to the Poisson background and applies to both attractive and repulsive particle interaction. This is in stark contrast to classical Landau damping on $\T^d$, where \eqref{eq:penrose-classic} is a robust stability condition irrespective of dimension and the type of particle interaction. In this context, one should note that, when studying Landau damping for the Poisson equilibrium on $\R^3$ in \cite{IPWW24R3}, the authors observe that the necessary dispersive estimates only yield sufficient decay for $d\geq 3$, see Remark 1.2 therein. 

\subsubsection{Relativistic equations}\label{subsec:intro-rel}

While phase mixing results in cosmologically motivated settings are still rather sparse, we note that such effects have been observed for relativistic kinetic equations near black hole exteriors in a similar vein to the Newtonian phenomena discussed at the start of the previous section, see \cite{RS18,RS24}.\\

For the relativistic Vlasov-Poisson system with repulsive particle interaction, which arises as a special case of relativistic Vlasov-Maxwell system, Landau damping was shown for the linearised equations on $\T^3$ and $\R^3$ by Young in \cite{You15}, and then partially extended to the nonlinear setting in \cite{You16}. For both of these results, the fact that information cannot travel faster than the speed of light imposes limitations compared to the classical setting: \cite[Theorem 2.1]{You15} states that, in most settings, exponential decay of the density modes is prohibited, and thus any damping effect one can establish must be weaker than in the classical setting. Indeed, in \cite[Theorem 2.2]{You15}, the modes decay as $\exp(-ct^{\gamma})$ for $\gamma<1$. In \cite{You16}, the background is taken to be $(\gamma^{-1})$-Gevrey regular with $\gamma<1$, and nevertheless, one can only obtain damping in $({\gamma^\prime})^{-1}$-Gevrey classes with ${\gamma^\prime}<\gamma$.\\

In their recent work \cite{TVR25}, Taylor and Velozo Ruiz studied phase mixing for the relativistic massive Vlasov equation on a fixed FLRW background with power law expansion and without particle interaction. Therein, they were able to show that phase mixing occurs if the background expands with $t^q$ for $q\in(0,\frac12)$ via a vector field method. For Sobolev-regular initial data, the density contrast decays polynomially, and for analytic initial data, it decays superpolynomially. At the critical expansion rate $q=\frac12$, the density contrast exhibits logarithmic and polynomial decay respectively. \\
When comparing this work to our present results, it is useful to view \eqref{eq:perturbation-evol} as the formal nonrelativistic limit of the relativistic Vlasov-Maxwell system on an FLRW background after transforming the system into coordinates comoving with respect to expansion. We discuss this more detail in Remark \ref{rem:nonrel-limits}. From this viewpoint, Theorem \ref{thm:main} confirms that the critical threshold for phase mixing of the free transport equation is also the critical threshold of Landau damping when introducing repulsive Newtonian particle interaction. In fact, we even obtain stronger decay rates for the density contrast in Theorem \ref{thm:main} than in \cite{TVR25}, as discussed after Theorem \ref{thm:main}. It is not clear whether the decay bounds obtained in \cite{TVR25} are sharp, especially since Gevrey regularity is not its main focus. Nevertheless, this difference in decay rates may be due to differences between the relativistic system and the Newtonian model as in the works by Young on the relativistic Vlasov-Poisson system, but also due to the presence of a nontrivial background aiding decay of the perturbation, leading to Landau damping that is truly stronger than phase mixing for the free transport equation alone.\\

For the Einstein-Vlasov system, while homogenization of the matter content itself is not well understood, expansion does cause many cosmological solutions to dilute toward the future when the expansion rate is sufficiently strong, and thus the geometry on spatial slices in a suitable gauge converges to a spatially homogeneous limit. This includes accelerated de-Sitter-like expansion in \cite{Rin13,AndRin16} and linear expansion in \cite{AndFaj20,BaFaj20} as well as the $2+1$-dimensional results \cite{Faj17,Faj18}. To obtain results for weaker expansion rates, including the FLRW solution on the torus with zero cosmological constant, one likely would need to invoke stronger effects than dilution, such as phase mixing in compact spatial topology or dispersion for open topologies. Regarding the latter, Taylor showed in \cite{Tay24} that spherically symmetric perturbations of the FLRW solution with spatial topology $\R^3$ within the Einstein-massless Vlasov system, which exhibit power law expansion with $q=\frac12$, are stable due to dispersion.\\

Finally, we note that vector-field methods originally developed to show stability of Minkowski spacetime within the Einstein-Vlasov system, such as in \cite{LiTay20,FJS21,BFJSTh21}\footnote{Compare also to the alternative approach in \cite{Tay17}.}, can also be useful in studying the Vlasov-Poisson and Vlasov-Maxwell systems, see for example \cite{Smu16,Big25}. %which studies small data asymptotics to the Vlasov-Poisson system with vector field methods originating from the former.

\subsection{Organisation}

In Section \ref{sec:prelim+over}, we provide the necessary tools to precisely state the main result in Theorem \ref{thm:main} and to outline the remainder of the proof. This splits into a linear analysis of charge density modes in Section \ref{sec:lin}, and a bootstrap argument in Section \ref{sec:nonlin} to deal with nonlinear terms. We also provide three appendices: In Appendix \ref{sec:newt-cosmo}, we derive the expanding Vlasov-Poisson system from a classical one for general dimensions and including gravitational particle interaction, and elaborate on the relationship between \eqref{eq:perturbation-evol} and charged relativistic equations. In Appendix \ref{sec:4+}, we discuss Landau damping in high spatial dimensions and some heuristics on why this fails for $d=3$. Finally, we collect tools regarding the Liouville-Green approximation in Appendix \ref{sec:ode}, which we use in Section \ref{sec:lin}.

\subsection*{Acknowledgements} This research was funded in part by the Austrian Science Fund (FWF) projects \enquote{Matter-dominated cosmology} (10.55776/PAT7614324) and \enquote{Dynamics of matter in the decelerated epoch} (10.55776/PAT1953025). For open access purposes, the authors have applied a CC BY public copyright license to any author accepted manuscript version arising from this submission.

\section{Preliminaries, main theorem and proof overview}\label{sec:prelim+over}

In this section, after setting up notation and showing that the system \eqref{eq:perturbation-evol} is locally well-posed at the level of Gevrey regularity, we renormalise \eqref{eq:perturbation-evol} to, in particular, derive the family of Volterra equations \eqref{eq:volterra-outline} for non-zero modes of the rescaled density, which form the centerpiece of our damping analysis. We do all of this first for a general positive analytic background $\mu$ before restricting to the Poisson background when introducing Theorem \ref{thm:main} and discussing its proof.

\subsection{Notation and Fundamentals}\label{subsec:prelim}

\subsubsection{Conventions} We denote the $d$-dimensional torus with side length $L$ by $\T^d_L:=\sfrac{[0,L]}{\sim}$ and choose the convention $\T^d:=\T^d_{2\pi}$. For any $d\geq 1$, we use $\lvert\,\cdot\,\rvert$ to denote the repsective Euclidean norm on $\R,\T^d$ and $\R^d$, and use the Japanese brackets
\[\langle y\rangle=\sqrt{1+\lvert y\rvert^2},\quad \langle y,y^\prime\rangle=\sqrt{1+\lvert y\rvert^2+\lvert y^\prime\rvert^2}\,.\]
We denote the standard scalar product between elements $y,y^\prime$ of $\Z^d$ or $\R^d$ by $y_1\cdot y_2$, and will commonly write $yz$ for the scalar multiplication of $z\in\C$ with $y\in\C^d$. We also use the Einstein summation convention throughout where appropriate, with lowercase Latin indices $(i,j,\dots)$ running from $1$ to $3$. The complex conjugation of $z\in\C$ is denoted by $\overline{z}$. \\

We use the Fourier transform on $\T^d\times\R^d$ with the non-unitary convention
\begin{align*}
\mathcal{F}_{x,v}[\psi]:\,\Z^d\times\R^d\longrightarrow\C\,,\quad
\mathcal{F}_{x,v}[\psi](k,\xi)=&\,\int_{\T^d\times\R^d}\psi(x,v)\,e^{-i(k\cdot x)}\,e^{-i(\xi\cdot v)}\,dx\,dv\,,
\end{align*}
and all other Fourier transforms use analogous conventions. To ease notation, we denote Fourier transforms in $\T^d$, $\R^d$ and $\T^d\times\R^d$ by hats, for example writing
\[\hat{\psi}(t,k,\xi)=\mathcal{F}_{x,v}[\psi(t,\cdot,\cdot)](k,\xi)\,.\]
In particular, the Poisson background satisfies
\begin{equation}\label{eq:poisson-fourier}
\hat{\mu}(\xi)=\exp(-\theta_0\,\lvert\xi\rvert)\,.
\end{equation}

For functions $y_1,y_2:D\longrightarrow[0,\infty)$ for some domain $D$, we use $y_1\lesssim y_2$ to denote that there exists some uniform constant $C>0$ such that $y_1(z)\leq Cy_2(z)$ holds for all $z\in D$. Here, $C$ may depend on any parameters in this exposition with exception of the initial data size $\epsilon>0$, see \eqref{eq:ass-init}. Where the implicit constant depends on certain other parameters (e.g., $\gamma$), we may choose to denote said dependency in the subscript (e.g., $\lesssim_\gamma$). If $y_1\lesssim y_2$ and $y_2\lesssim y_1$ hold, we write $y_1\simeq y_2$.
%To avoid any ambiguity in notation, the closure of a set $Z$ with respect to a given topology will be denoted by $\text{cl}[T]$.\\

\subsubsection{Function spaces}

\begin{definition}[Gevrey norms and and generator functions]\label{def:gevrey}
Let $d\geq 1,\,\gamma\in(0,1],\,\sigma\geq 0$ as well as $z\in[0,1]$. First, for $(k,\xi)\in\Z^d\times\R^d$, we introduce the multiplier
\[A(z)_{k,\xi}=e^{z\langle k,\xi\rangle^\gamma}\langle k,\xi\rangle^\sigma\,.\]
We define the following Gevrey norms for any $\psi_1:\T^d\rightarrow\R$ and $\psi_2:\T^d\times\R^d\rightarrow\R$.

\begin{align*}
\|\psi_1\|_{\mathcal{G}^{z,\sigma,\gamma}(\T^d)}^2=&\sum_{k\in\Z^d} e^{2z\langle k\rangle^{\gamma}}\,\langle k\rangle^{2\sigma}\,\lvert \hat{\psi}_1(k)\rvert^2\,.\\
\|\psi_2\|_{\mathcal{G}^{z,\sigma,\gamma}(\T^d\times\R^d)}^2=&\,\sum_{k\in\Z^d}\int_{\R^d} (A(z)_{k,\xi})^2\,\lvert \hat{\psi}_2(k,\xi)\rvert^2\,d\xi\,.
\end{align*}

The Gevrey spaces $\mathcal{G}^{z,\sigma,\gamma}(\T^d)$ and $\mathcal{G}^{z,\sigma,\gamma}(\T^d\times\R^d)$ are the spaces of smooth functions with finite Gevrey norm, and for the case of $\sigma=0$, we simply write $\mathcal{G}^{z,\gamma}$ instead of $\mathcal{G}^{z,0,\gamma}$ in all of these instances. Additionally, for $\psi_1: (0,\infty)\times\T^d\rightarrow\R$ and $\psi_2:(0,\infty)\times\T^d\times\R^d\rightarrow\R$, and parameters $\gamma$ and $\sigma$ as in Definition \ref{def:gevrey} as well as $\alpha\in(\frac13,1]$, we define the following generator functions.

\begin{align*}
F[\psi_1](\tau,z)=&\,\sup_{k\in\Z^d\backslash\{0\}}A(z)_{k,k\tau}\,\lvert k\rvert^{-\alpha}\,\widehat{\psi_1}(\tau,k)\,,\\%=a(t)^{d}\,\sup_{k\in\Z^3}A(z)_{k,k\tau(t)}\hat{\sigma}_k(t)\lvert k\rvert^{-\alpha}\\
G[\psi_2](\tau,z)=&\,\sum_{j=0}^{d}\sum_{k\in\Z^d}\int_{\R^d} (A(z)_{k,\xi})^2\,\lvert D^j_\xi\widehat{\psi_2}(\tau,k,\xi)\rvert^2\,d\xi\,.
\end{align*}
\end{definition}

Additionally, we will make use of sliding regularity norms for $\lambda_0>0$ and $\gamma\in(\frac13,1]$ as later specified in the initial data condition in Theorem \ref{thm:main}.

\begin{definition}[Sliding regularity]\label{def:slide}
For $\delta>0$ sufficiently small and $\lambda_1\leq\lambda_0/4, \lambda_0\leq 1$ as well as, if $\gamma=1$, $\lambda_1<\frac{\theta_0}2$, we define the sliding parameter
\begin{subequations}\label{eq:slide}
\begin{equation}
z(\tau)=\lambda_1\left(1+\langle\tau\rangle^{-\delta}\right)
\end{equation}
and the corresponding generator functions
\begin{equation}
\tilde{F}_k[\psi_1](\tau)=F_k[\psi_1](\tau,z(\tau)),\quad \tilde{F}[\psi_1](\tau)=F[\psi_1](\tau,z(\tau)),\quad \tilde{G}[\psi_2](\tau)=G[\psi_2](\tau,z(\tau))\,.\end{equation}
\end{subequations}
\end{definition}

\subsubsection{Elementary estimates}

For later convenience, we collect some basic estimates that we will frequently use throughout this work. Since their proofs are entirely elementary, we omit them here.

\begin{lemma}\label{lem:algebraic} For $\gamma\in(0,1]$, $k,k^\prime\in\Z^d$ and $\xi,\xi^\prime\in\R^d$, one has
\begin{gather*}
\langle k+k^\prime,\xi+\xi^\prime\rangle^\gamma\leq \langle k,\xi\rangle^\gamma+\langle k^\prime,\xi^\prime\rangle^\gamma\,,\\
\frac{\langle k,\xi\rangle}{\langle k^\prime,\xi^\prime\rangle}\leq 2\langle k+k^\prime,\xi+\xi^\prime\rangle\,.
\end{gather*}
In particular, this implies for $\sigma>0$ that
\begin{equation}\label{eq:algebra-prop}
A(z)_{k,\xi}\leq 2^\sigma A(z)_{k-k^\prime,\xi-\xi^\prime}\,A(z)_{k^\prime,\xi^\prime}\,.
\end{equation}
\end{lemma}

\begin{lemma}\label{lem:tool}
%For $b_1,b_2,c>0$, the map
%\[y\mapsto \lvert y\rvert^{b_1}\,\exp(-c\,\lvert y\rvert^{b_2})\]
%assumes its maxima at $\lvert y\rvert=\left(b_1\,c^{-1}\right)^\frac1{b_2}$ with the value $(b_1\,c^{-1})^\frac{b_1}{b_2}\,\exp(-b_1)$.
For $b_1,b_2,c>0$ and $y\in\R$, one has the following uniform bound.
\begin{equation}\label{eq:tool}
\lvert y\rvert^{b_1}\,\exp(-c\,\lvert y\rvert^{b_2})\leq (b_1\,c^{-1})^\frac{b_1}{b_2}\,\exp(-b_1)\,.
\end{equation}
\end{lemma}

\subsection{Local well-posedness}\label{subsec:lwp}

To have access to sufficient regularity and a continuation criterion for solutions to \eqref{eq:perturbation-evol}, we invoke the following standard well-posedness result.

\begin{lemma}[Local well-posedness and continuation criterion]\label{lem:lwp}
Let $\mathring{g}:\T^3\times\R^3\rightarrow\R$ be an analytic initial distribution for \eqref{eq:perturbation-evol} at $t=t_0$, in particular satisfying
\[\int_{\T^3\times\R^3}\mathring{g}(x,v)\,dxdv=0\,\]
such that
\[f_0(t_0,v)+\mathring{g}(x,v)\geq 0\quad \text{for all} (x,v)\in\T^3\times\R^3\,.\] 
Further, for $I\in\N^3$, writing
\[(v^I\mathring{g})(x,v)=v^{i_1}v^{i_2}v^{i_3}\,\mathring{g}(x,v)\]
and for $\tilde{\lambda}_0>0$ as well as $\sigma\geq 0$, assume that the following holds. 
\[\sum_{I\in\N^3,\,\lvert I\rvert\leq 2}\|v^I\,\mathring{g}\|_{\mathcal{G}^{\tilde{\lambda}_0,\sigma,1}}<\infty\,.\]
Let $\tilde{\lambda}$ be a continuous decreasing function with $\tilde{\lambda}(t_0)=\tilde{\lambda}_0$ and $\lim_{t\to\infty}\tilde{\lambda}(t)>0$. Then, there exists some maximal $t^\ast\in(t_0,\infty]$ such that there exists a unique analytic solution $g$ to \eqref{eq:perturbation-evol} and that, for all $t<t^\ast$, one has
\[\sup_{\tilde{t}\in[0,t]}\sum_{I\in\N^3,\lvert I\rvert\leq 2}\|v^Ig(t,\cdot,\cdot)\|_{\mathcal{G}^{\tilde{\lambda}(\tilde{t}),\sigma,1}(\T^3\times\R^3)}<\infty\,.\]
Further, there exist some $\aleph_0>0$ such that one either has that $t^\ast=\infty$ or
\[\lim_{t\to t^\ast} G[g](t,\lambda(t))\geq\aleph_0^2\,.\]
Finally, $(f_0+g)(t,\cdot,\cdot)$ is nonnegative for all $t\in [t_0,t^\ast)$. 
%and if $\tau^\ast<\infty$, then one has
%\begin{equation}\label{eq:cont-crit}
%\limsup_{\tau\up\tau^\ast}\|\sum_{I\in\N^3,\lvert I\rvert\leq 2}\|v^Ig(T(\tau),\cdot,\cdot)\|_{H^2(\T^3\times\R^3)}=\infty
%\end{equation}
\end{lemma}
\begin{proof}[Proof-Sketch.]
Local existence follows by applying an abstract Cauchy-Kovalevskaya theorem to \eqref{eq:perturbation-evol} as in \cite{Nir72}. The continuation criterion then follows by a standard parabolic regularisation argument as is sketched in \cite[Lemma 5.1]{IPWW24}, in particular since $\tilde{G}[g]$ is clearly bounded from above and below by suitable $(\gamma^{-1})$-Gevrey norms of $g$. Finally, nonnegativity of $f_0+g$ follows from solving the transport equation with the method of characteristics.
\end{proof}

We can assume our initial data to be analytic without loss of generality by an approximation argument. In particular, under the above result, $\hat{g}(t,\cdot,\cdot)$ twice continuously differentiable in space and momentum and all generator functions that will occur are well-defined and continuously differentiable in $t$. We will assume all of this tacitly from here onwards, and refer to initial data to \eqref{eq:perturbation-evol} as a sufficiently regular function $\mathring{g}:\T^3\times\R^3\rightarrow\R$ that has zero mean and such that $f(t_0,v)+\mathring{g}(t,x,v)$ is nonnegative. 

\subsection{Transforming the Vlasov-Poisson system}\label{subsec:transform}

To study \eqref{eq:perturbation-evol}, we first want to remove free transport effects from the analysis of the perturbation. The free transport component of \eqref{eq:perturbation-evol-landau} is given by
\begin{align}\label{eq:free-transport}
\del_tg_{free}+\frac{v^i}{a}\del_{x^i}g_{free}-\frac{\dot{a}}{a}v^i\del_{v^i}g_{free}=0\,.
\end{align}
The associated characteristic system is given by
\[\dot{X}^i(t)=\frac1{a(t)}\,V^i(t),\quad \dot{V}^i(t)=-\frac{\dot{a}(t)}{a(t)}\,V^i(t)\,.\]
Thus, free particles with initial data $(x,v)$ at $t_0>0$ follow the trajectories 
\[X(t;t_0,x,v)=x+\tau(t)\,a(t_0)\,v,\quad V(t;t_0,x,v)=\frac{a(t_0)}{a(t)}\,v\]
for
\[\tau(t)=\int_{t_0}^ta(s)^{-2}\,ds\]
and one has
\begin{align}\label{eq:free-transport-flrw}
&\,g_{free}(t,x+\tau(t)\,v,a(t)^{-1}\,v)=g_{free}(t_0,x,a(t_0)v)\\
\Longleftrightarrow&\,g_{free}(t,\tilde{x},a(t_0)^{-1}\tilde{v})=g_{free}(t_0,\tilde{x}-\tau(t)\,a(t)\,\tilde{v},a(t)\,\tilde{v})\,.
\end{align}
We define $T$ as the inverse of $t\mapsto \tau(t)$ on its range, i.e., the solution to the initial value problem
\[T^\prime(\tau)=a(T(\tau))\,\quad T(0)=t_0\,,\]
and define the renormalised distribution $h$ and density contrast $\rho$ as follows.
\begin{subequations}\label{eq:h-def}
\begin{gather}
h:\tau((0,\infty))\times\T^3\times\R^3\longrightarrow\R,\quad h(\tau,x,v):=g\left(T(\tau),x+\tau\,v,a(T(\tau))^{-1}\,v\right)\\
\rho:\tau((0,\infty))\times\T^3\longrightarrow\R,\quad \rho(\tau,x):=a(T(\tau))^{3}\underline{\rho}(T(\tau),x)\,.
\end{gather}
\end{subequations}
One computes, using $\del_\tau T=a(T(\tau))^2$, that \eqref{eq:perturbation-evol} becomes
\begin{align*}
&\,(\del_\tau h)(\tau,x,v)\\
=&\,a(T(\tau))^2\,(\del_xW)(T(\tau),x+\tau\,v)\cdot\left[(D\mu)(v)+\frac1{a(T(\tau))}(\del_vg)(T(\tau),x+\tau\,v,a(T(\tau))^{-1}\,v)\right]\\
=&\,a(T(\tau))^2\,(\del_xW)(T(\tau),x+\tau\,v)\cdot \left[(D\mu)(v)+(\del_vh)(\tau,x,v)-\tau\,\del_xh(\tau,x,v)\right]
\end{align*}
This transforms \eqref{eq:perturbation-evol} into the following system.
\begin{subequations}\label{eq:tau-sys}
\begin{align*}
\numberthis\label{eq:landau}&\left(\del_\tau h\right)(\tau,x,v)-a(T(\tau))^2\left(\del_xW\right)(T(\tau),x+\tau v)\,\cdot \,(D\mu)(v)\\
&\,+a(T(\tau))^2\,(\del_xW)(T(\tau),x+\tau v)\cdot\left(\left(\tau\del_x-\del_v\right)h\right)(\tau,x,v)=0
\end{align*}
\begin{equation}\label{eq:sigma-W-and-h}
(\Lap W)(T(\tau),x)=-4\pi\,a(T(\tau))^{-1}\,\rho(\tau,x),\quad \rho(\tau,x)=\int_{\R^3} h(\tau,x-\tau\,w,w)\,dw 
\end{equation}
\end{subequations}
Again invoking Lemma \ref{lem:lwp} tacitly from here on out, we apply the Fourier transform over $\T^3\times\R^3$ to this system. For $k\in\Z^3\backslash\{0\}$, using integration by parts and \eqref{eq:sigma-W-and-h} implies
 \begin{align*}
&a(T(\tau))^2\,\int_{\T^3\times\R^3} \left[(\del_x W)(T(\tau),x+\tau v)\,\cdot\,(D\mu)(v)\right]\,e^{-i(k\cdot x)}\,e^{-i(\xi\cdot v)}\,dx\,dv\\
=&\,a(T(\tau))^2\int_{\T^3\times\R^3}\left[(\del_{\tilde{x}}W)(T(\tau),\tilde{x})\,\cdot\,(D\mu)(v)\right]\,e^{-i(k\cdot (\tilde{x}-\tau{v}))}\,
e^{-i(\xi\cdot v)}\,d\tilde{x}\,dv\\
=&\,a(T(\tau))^2\,\frac{k}{i\lvert k\rvert^2}\cdot\int_{\T^3\times\R^3} (D\mu)(v)\,(\Lap W)(T(\tau),\tilde{x})\,e^{-i\,(k\cdot(\tilde{x}-\tau{v}))}\,e^{-i\,(\xi\cdot v)}\,d\tilde{x}\,dv\\
=&\,-4\pi\,a(T(\tau))\,\frac{k}{i\lvert k\rvert^2}\cdot\int_{\T^3\times\R^3} (D\mu)(v)\,\rho(\tau,\tilde{x})\,e^{-i(k\cdot\tilde{x})}e^{-i((\xi-k\tau)\cdot {v})}\,d\tilde{x}\,d{v}\\
=&\,-4\pi\,a(T(\tau))\hat{\rho}_k(\tau)\,\frac{k}{i\lvert k\rvert^2}\cdot \int_{\R^3} (D\mu)(v)\,e^{-i((\xi-k\tau)\cdot {v})}\,d{v}\\
=&\,-4\pi\,a(T(\tau))\frac{k\cdot (\xi -k\tau)}{\lvert k\rvert^2}\, \hat{\rho}_k(\tau)\,\hat{\mu}\left(\xi-k\tau\right)\,.
\end{align*}
Treating the second term similarly, we obtain that \eqref{eq:landau} is equivalent to
\begin{subequations}
\begin{align*}
\numberthis\label{eq:landau-fourier}\del_\tau\hat{h}(\tau,k,\xi)+4\pi a(T(\tau))\frac{k\cdot(\xi-k\tau)}{\lvert k\rvert^2}\,\hat{\rho}_k(\tau)\,\hat{\mu}(\xi-k\tau)&\,\\-4\pi a(T(\tau))\sum_{l\in\Z^3}\frac{l\cdot(\xi-k\tau)}{\lvert l\vert^2}\,\hat{\rho}_l(\tau)\,\hat{h}(\tau,k-l,\xi-l\tau)&\,=0\,.
\end{align*}
%The linearization of \eqref{eq:landau-fourier} thus reads
%\begin{gather}
%\del_\tau\hat{h}(\tau,k,\xi)-4\pi\,a(T(\tau))\,\frac{k\cdot (\xi -\tau\,k)}{\lvert k\rvert^2}\, \hat{\rho}_k(\tau)\,\hat{\mu}\left(\xi-\tau\,k\right)=0\,.
%\end{gather}
Commuting by $\del_{\xi_{i_1}}\dots\del_{\xi_{i_j}}$ for $j\in\{1,2,3\}$, which we formally denote as $D_{\xi}^j$ with $D_{\hat{j},j}$ as $\del_{\xi_{i_1}}\dots\del_{\xi_{i_j}}$ without $\del_{\xi_{\hat{j}}}$, this reads
\begin{align*}
\numberthis\label{eq:landau-fourier-commuted}0=&\,(\del_\tau D_\xi^j\hat{h})(\tau,k,\xi)+4\pi a(T(\tau))\frac{k\cdot(\xi-k\tau)}{\lvert k\rvert^2}\,\hat{\rho}_k(\tau)\,(D_\xi^j\hat{\mu})(\xi-k\tau)\\
&-4\pi a(T(\tau))\sum_{l\in\Z^3}\frac{l\cdot(\xi-k\tau)}{\lvert l\vert^2}\,\hat{\rho}_l(\tau)\,(D_\xi^j\hat{h})(\tau,k-l,\xi-l\tau)\\
&-4\pi a(T(\tau))\sum_{\hat{j}=1}^j\sum_{l\in\Z^3}\frac{l_{\hat{j}}}{\lvert l\vert^2}\,\hat{\rho}_l(\tau)\,(D_{\hat{j},j}\hat{h})(\tau,k-l,\xi-l\tau)\\
=:&\,(\del_\tau D_\xi^j\hat{h})(\tau,k,\xi)+(L)+(R)+(P)\,.
\end{align*}
\end{subequations}
Furthermore, one has
\begin{align*}
\hat{\rho}_k(\tau)=&\,\int_{\T^3\times\R^3} e^{-i(k\cdot x)}\,h(t,x-\tau w,w)\,dx\,dw\\
=&\,\int_{\T^3\times\R^3} e^{-i(k\cdot x)}\,e^{-i(k\tau\cdot w)}h(t,x,w)\,dx\,dw=\,\hat{h}(\tau,k,k\tau)
\end{align*}
Thus, after integrating in $\tau$ and inserting $\xi=k\tau$, \eqref{eq:landau-fourier} becomes
\begin{subequations}
\begin{align}\label{eq:volterra-outline}
\hat{\rho}_k(\tau)=&\,\hat{S}_k(\tau)-4\pi\int_0^\tau(\tau-\tilde{\tau})\,\hat{\mu}(k(\tau-\tilde{\tau}))\,\hat{\rho}_k(\tilde{\tau})\cdot a(T(\tilde{\tau}))\,d\tilde{\tau}
\end{align}
with $S(\tau,x)$ being defined by its spatial Fourier transform
%with, in the linearized setting,
%\[\hat{S}_k(\tau)=\hat{h}(0,k,\tau k)\]
%and, in the full setting
\begin{align}\label{eq:source-nl}
\hat{S}_k(\tau)=&\,\hat{h}(0,k,k\tau)+4\pi\int_0^{\tau}\sum_{l\in\Z^3}\frac{(l\cdot k)(\tau-\tilde{\tau})}{\lvert l\vert^2}a(T(\tilde{\tau}))\,\hat{\rho}_l(\tilde{\tau})\,\hat{h}(T(\tilde{\tau}),k-l,k\tau-l\tilde{\tau})\,d\tilde{\tau}
\end{align}
\end{subequations}
for $k\in\Z^3$. Note that, for $k=0$, the fact that $\mathring{g}$ is charge-neutral, and consequently also that $\hat{\rho}_0(\tau)\equiv 0$ holds throughout, implies $\hat{S}_0(\tau)=0$, so we focus on $k\neq 0$ from here on out.\\
We will \eqref{eq:volterra-outline} as
\begin{subequations}
\begin{equation}\label{eq:volterra}
\hat{\rho}_k(\tau)=\hat{S}_k(\tau)+\int_0^\tau K_k(\tau,\tilde{\tau})\hat{\rho}_k(\tilde{\tau})\,d\tilde{\tau}
\end{equation}
with $K_k:\R\times\R\longrightarrow\R$ given by
\begin{equation}\label{eq:def-k}
K_k(\tau,\tilde{\tau})=\begin{cases}
-4\pi\,\hat{M}_k(\tau-\tilde{\tau})\,a(T(\tilde{\tau})),\quad \hat{M}_k(s):=s\,\hat{\mu}(ks) & \tau\geq\tilde{\tau}\geq 0\\
0 & \text{else}\,.
\end{cases}
\end{equation}
\end{subequations}

Before moving on to the proof outline, it will also be convenient to consider the scale factor $a$ as a function of $\tau$. In fact, we will from now on mainly consider $a\circ T$, and assume that this function exhibits power law expansion rather than $a$ itself, along with some additional technical assumptions.

\begin{ass}[Scale factor]\label{ass:scale}
We assume $a$ to be a positive, smooth and increasing function on $[t_0,\infty)$ for some $t_0>0$, and define $T$ as the inverse of $t\mapsto \tau(t)$, i.e., via the implicit ODE
\[T^\prime(\tau)=a(T(\tau))^2,\quad T(0)=t_0\,.\]
Furthermore, we assume that there exists some $\beta\in(0,\infty)$ such that
\begin{equation}\label{eq:scale-ass}
\lvert (a\circ T)(\tau)\rvert \lesssim \langle\tau\rangle^\beta
\end{equation}
holds as well as
\begin{equation}\label{eq:lg-ass}
\int_0^{\infty} a(T(\tilde{\tau}))^{-\frac14}\,\left\lvert\left((a\circ T)^{-\frac14}\right)^{\prime\prime}(\tilde{\tau})\right\rvert\,d\tilde{\tau}<\infty\,.
\end{equation}
\end{ass}
This is slightly more general than assuming that $a(t)$ exhibits power law expansion, but fully includes this case up to the critical threshold of $q=\frac12$:

\begin{example}[The power law solution]\label{ex:power-law-trafo}
Let $a(t)=t^q$ for some $q\geq 0$. Then, one has
\begin{equation*}
\tau(t)=\begin{cases}
\frac{1}{1-2q}(t^{1-2q}-t_0^{1-2q}) & q<\frac12\\
\log(t)-\log(t_0) & q=\frac12\\
\frac{1}{2q-1}(t_0^{1-2q}-t^{1-2q}) & q>\frac12\,.
\end{cases}
\end{equation*}
as well as
\begin{equation*}
a(T(\tau))=\begin{cases}
\left((1-2q)\tau+t_0^{1-2q}\right)^{\frac{q}{1-2q}} & q<\frac12\\
\sqrt{t_0}\,\exp(\frac{\tau}2) & q=\frac12\\
\left(t_0^{1-2q}-(2q-1)\tau\right)^{\frac{q}{1-2q}} & q>\frac12,\ \tau\in\left[0,\frac{1}{2q-1}\right)\,.\\
\end{cases}
\end{equation*}
Thus, power law spacetimes clearly satisfy \eqref{eq:scale-ass} for $q\in(0,\frac12)$ with $\beta=\frac{q}{1-2q}>0$. Furthermore, one then has
\begin{align*}
a(T({\tau}))^{-\frac14}\,\left\lvert\left((a\circ T)^{-\frac14}\right)^{\prime\prime}({\tau})\right\rvert=&\,\frac{q}4\left(1-\frac74q\right)\,\left((1-2q)\tau+t_0^{1-2q}\right)^{\frac{q}{2(1-2q)}-2}
\end{align*}
which is integrable on $(0,\infty)$, implying \eqref{eq:lg-ass}. For $q=0$, all conditions are satisfied trivially for arbitrary $\beta>0$.\\

For $q>\frac12$, $\tau(t)$ has finite range. Thus, even if we were able to prove bounds as in Theorem \ref{thm:main} below, this would not imply Landau damping since $\exp(-\langle \tau(t)\rangle)$ does not converge to zero as $t\to\infty$. In other words, particles only travel a finite distance in comoving coordinates when studying the free transport equation, see \eqref{eq:free-transport-flrw}, and thus, full phase mixing cannot occur.
\end{example}

%\begin{example}\label{ex:log-trafo}
%Let $a(t)=\log(1+t)$. Then, one has
%\begin{equation*}
%\tau(t)=\int_{t_0}^t\log(1+s)^{-2}\,ds=\text{li}(t+1)-\frac{t+1}{\log(t+1)}+C_{t_0}\,,
%\end{equation*}
%where $\text{li}$ is the logarithmic integral. Note that
%\[\text{li}(t+1)-\frac{t+1}{\log(t+1)}=\O{\frac{t+1}{(\log(t+1)^2}}\quad\text{as}\quad t\to\infty\,.\]
%and thus $\frac{\tau(t)}{t}\rightarrow \frac12$ as $t\to\infty$. Consequently, for appropriate $c>0$, we heuristically have
%\[a(T(\tau))=\log(1+T(\tau))\approx\log(1+c\tau)\quad\text{as}\quad \tau\to\infty\]
%\end{example}

\subsection{Main theorem}\label{subsec:outline}

With all of the setup dealt with, we are now in a position to state the main theorem of this paper.

\begin{theorem}[Main result]\label{thm:main} 
%Let $\mu$ be real analytic and satisfy
%\begin{equation}\label{eq:distr-ass}
%\left\lvert \left(D^{\leq 2}_\xi\hat{\mu}\right)(\xi)\right\rvert\leq C_0 e^{-\theta_0\lvert \xi\rvert}\,,
%\end{equation}
%for some $C_0,\,\theta_0>0$\,. 
Assume that the scale factor satisfies Assumption \ref{ass:scale} for some $\beta>0$, let $\beta^\prime=\frac32\,\beta$, and let
%\[\gamma\in\left(1-\frac{3}{6+2\beta},1\right),\quad \sigma>\max\{4,2+\beta\}. \footnote{\todo{If we extend this to $\gamma=1$, then we would additionally need $\lambda_0<\theta_0$ there, but this eludes me currently.}}\]
\begin{equation}\label{eq:reg-main}
\gamma\in\left(1-\frac{2}{3+\beta+\beta^\prime},1\right],\quad \sigma>\max\{4,2+\beta+\beta^\prime\}\,.
\end{equation}
Let $\mathring{g}$ be initial data for the perturbed Vlasov-Poisson system with repulsive particle interaction \eqref{eq:perturbation-evol}, see Section \ref{subsec:lwp}, that has zero mean and satisfies
\begin{equation}\label{eq:ass-init}
\|\mathring{g}\|_{\mathcal{G}^{\lambda_0,\sigma,\gamma}(\T^3\times\R^3)}\leq\epsilon^2
\end{equation}
for some $\lambda_0>0$, with $\lambda_0<\theta_0$ if $\gamma=1$, as well as some sufficiently small $\epsilon>0$. Then, for $\lambda_1\in(0,0.25\lambda_0)$ and any $\lambda^\prime\in(0,\lambda_1)$, there exists some $h_\infty\in \mathcal{G}^{\lambda_1,\gamma}(\T^3\times\R^3)$ with zero mean such that one has
\begin{subequations}
\begin{align}
\|h(\tau(t),\cdot,\cdot)-h_\infty\|_{\mathcal{G}^{\lambda^\prime,\gamma}(\T^3\times\R^3)}\lesssim&\,\epsilon e^{-(\lambda_1-\lambda^\prime)\langle\tau(t)\rangle^\gamma}\label{eq:main-g}\\
\|\underline{\rho}(t,\cdot)\|_{\mathcal{G}^{\lambda^\prime,\gamma}(\T^3)}\lesssim&\,\epsilon a(t)^{-3}e^{-(\lambda_1-\lambda^\prime)\langle\tau(t)\rangle^\gamma}\label{eq:main-rho}
\end{align}
\end{subequations}
\end{theorem}
In particular, if $\tau(t)\to \infty$ as $t\to\infty$, this implies that the density gains additional decay. As discussed in Example \ref{ex:power-law-trafo}, if one has $a(t)= C\,t^q$ for $q\in(0,\frac12)$ and $C>0$, and one chooses
\begin{equation}\label{eq:reg-max-q}
\gamma\in\left(\frac13+\frac{10q}{18-21q},1\right],\quad \sigma>\max\left\{4,2+\frac{5q}{2-4q}\right\}\,,
\end{equation}
%\begin{equation}\label{eq:reg-max-q}
%\gamma\in\left(\frac12+\frac{q}{6-10q},1\right),\quad \sigma>\max\left\{4,2+\frac{q}{1-2q}\right\}\,,
%\end{equation}
this implies that non-zero modes of the charge density decay at least by
\[t^{-3q}\,\exp\left(-c_{\beta,\gamma}\,(\lambda_1-\lambda^\prime)\,t^{\gamma(1-2q)}\right)\]
for an appropriate constant $c_{\beta,\gamma}>0$ depending only on $\beta$ and $\gamma$. By contrast, for the relativistic Vlasov equation on a fixed background, Taylor and Velozo Ruiz obtained the weaker decay rate 
\[\exp\left(-c\,t^{\frac{1-2q}{9+\gamma^{-1}}}\right)\,\]
for some $c>0$ and $(\gamma^{-1})$-regular initial data, see \cite[Remark 1.7]{TVR25}. As discussed in Section \ref{subsec:intro-rel}, this difference may occur due to complications arising from the relativistic equation, but also due to the equilibrium aiding decay of the perturbation.\\

The Gevrey regularity requirements \eqref{eq:reg-main} are necessary to control plasma echoes for large times, as in the classical setting, since the Fourier transform of $h$ sufficiently strongly to ensure summability. Since the critical nonlinearity \eqref{eq:source-nl} gains an additional factor of $a(T(\tau))$ in the expanding setting compared to the classical one, one must expect that the regularity requirements become more sharp as the rate of expansion increases.  We make no claim that our bounds are optimal in this regard -- in particular, the presence of $\beta^\prime$ comes from a loss of scaling in the linear estimates that might have room for improvement for this background, or maybe be smaller for other backgrounds. Nevertheless, \eqref{eq:reg-max-q} demonstrates that, for very slow expansion rates, the lower Gevrey regularity threshold is close to the critical value $\gamma=\frac13$ in the classical setting, while for $q$ close to the critical value $\frac12$, $\gamma$ must be chosen very close to $1$ and $\sigma$ must be chosen to be very large.  
%However, it is not clear whether this regularity is optimal, since we at minimum require $2$-Gevrey regularity for any rate of expansion. Indeed, when estimating the nonlinearities pointwise, one can loosen the requirements on $\gamma$ to
%\[\gamma\in\left(\frac13+\frac{2q}{9-15q},1\right)\,,\]
%and the stricter requirements are only needed since we need additional decay for certain coefficient terms, as outlined in Section \ref{subsec:nonlin}. Meanwhile, the analytic regime can likely be covered by taking more care in the proof of linear damping, but we refrain from doing so for the sake of simplicity, especially in light of results in the relativistic Vlasov-Poisson setting discussed below.

%In other words, for Landau damping to occur for faster expansion rates, one needs to impose increasingly strict regularity requirements on chosen initial data to still ensure damping, and said damping becomes weaker.%\footnote{[if extendible: higher order polynomial decay for $a(t)\simeq t^\frac12$, this would have to involve some parameter loss in the Gevrey-$1$-class depending on the background]}

\subsection{Proof overview and discussion}

The proof can be split into two main components. 
%First, we show that Volterra integral equations as in \eqref{eq:volterra-outline} admit a linear damping estimate for an abstract source for the Poisson background. Then, we run a bootstrap argument to show plasma echoes do not arise from the nonlinear terms in the source \eqref{eq:source-nl} if the initial satisfies the initial data assumption \eqref{eq:ass-init} with the regularity requirement \eqref{eq:reg-main}. 

\subsubsection{The linear analysis}\label{subsubsec:overview}

In Section \ref{sec:lin}, we view the source terms $(\hat{S}_k)_{k\in\Z^3\backslash\{0\}}$ in \eqref{eq:volterra-outline} as a black box and study the family of functions $(\phi_k)_{k\in\Z^3\backslash\{0\}}$ that solve the Volterra integral equations
\begin{subequations}\label{eq:volterra-abstract}
\begin{align*}
\numberthis\phi_k(\tau)=&\,s_k(\tau)+\int_0^\tau K_k(\tau,\tilde{\tau})\,\phi_k(\tilde{\tau})
\end{align*}
\end{subequations}
with $K_k$ defined as in \eqref{eq:def-k}. Solutions to this family of equations are given by
\begin{equation*}
\phi_k(\tau)=s_k(\tau)+\int_0^\tau R_k(\tau,\tilde{\tau})\,s_k(\tilde{\tau})\,d\tilde{\tau}
\end{equation*}
for some continuous resolvent function $R_k:\R\times\R\longrightarrow \C$. The main goal of this section is Proposition \ref{prop:res-bd}, which establishes the pointwise bound \eqref{eq:res-bd} on $R_k$, bounding its growth polynomially along with exponential decay in $(\tau-\tilde{\tau})$. This specific bound, when applied to \eqref{eq:volterra-outline}, then allows us to estimate $\tilde{F}[\rho]$ pointwise in terms of $\tilde{F}[S]$ by straightforward computations, see Corollary \ref{cor:lin-damping-gevrey-concrete}.\\

In absence of the scale factor $a(T(\tilde{\tau}))$, the integration kernel in \eqref{eq:volterra-abstract} would be of convolution type. Then, one could explicitly solve these equations using the Fourier-Laplace transform, which turns \eqref{eq:volterra-abstract} into an algebraic relation in the Fourier-Laplace domain, as is usually done when studying classical Landau damping, see for example \cite{MV10,BMM16,GNR21}. The presence of expansion causes significant issues in this standard approach since it breaks the convolution structure. Instead, we use that $R_k(\tau,\tilde{\tau})$ satisfies an integral equation itself that, for the Poisson background, can be reduced to a second order ordinary differential equation whose solutions we can estimate via the Liouville-Green approximation. This crucially uses that, for the Poisson equilibrium, $K_k$ can be written as a product of functions of $\tau$ and $\tilde{\tau}$, which is the main simplification it provides. While this method does not extend to general equilibria, it is possible that there are different ways of establishing resolvent bounds as in \eqref{eq:res-bd} for general, radially symmetric equilibria in light of the classic Penrose condition.

\subsubsection{The nonlinear analysis}\label{subsec:nonlin}

In Section \ref{sec:nonlin}, our goal is to use the linear estimate in Corollary \ref{cor:lin-damping-gevrey-concrete} to obtain Theorem \ref{thm:main} by estimating the nonlinearities in the source $S$. To this end, we make the following bootstrap assumption.

\begin{ass}[Bootstrap assumption]\label{ass:boot} For $\tau\in[0,\tau_{Boot})$, we assume
\begin{equation}\label{eq:boot}
\tilde{G}[h](\tau)\leq 2(1+C_0)\epsilon\,.
\end{equation}
Here, $C_0>0$ is an appropriate constant that we can specify after the fact.
\end{ass}

The local well-posedness result in Section \ref{subsec:lwp} ensures that this assumption is satisfied for some $\tau_{Boot}>0$. With Corollary \ref{cor:lin-damping-gevrey-concrete}, we pursue a bootstrap improvement strategy inspired by \cite[Section 4]{GNR21}, leading to the following result.

\begin{corollary}[Differential bound in sliding regularity]\label{cor:F-to-G-intro}
Assume that, for any $\tau\in[0,\tau_{Boot})$, one has
\begin{equation}\label{eq:F-cond-intro}
a(T(\tau))\tilde{F}[\rho](\tau)\lesssim \sqrt{\epsilon} \,\langle\tau\rangle^{-2-\delta}
\end{equation}
for some $\delta>0$. Then, for sufficiently small $\epsilon>0$ and any $\tau\in[0,\tau_{Boot})$,
\begin{subequations}
\begin {equation}\label{eq:delt-G-imp-intro}
\del_\tau\sqrt{\tilde{G}[h](\tau)}\lesssim \epsilon\langle \tau\rangle^{-2}
\end{equation}
holds as well as
\begin{equation}\label{eq:G-imp-intro}
\tilde{G}[h](\tau)\leq (1+C_0)\epsilon\,.
\end{equation}
\end{subequations}
\end{corollary}
Thus, if \eqref{eq:F-cond-intro} is satisfied, then this is a strict improvement of the bootstrap assumption, and by Lemma \ref{lem:lwp}, we know that there is some $\tau^\prime>\tau_{Boot}$ such that the bootstrap assumption \eqref{eq:boot} extends to $[0,\tau^\prime)$. This closes the bootstrap, and Theorem \ref{thm:main} then follows directly. We provide more details on this at the end of Section \ref{sec:nonlin}.

While adaptations are needed throughout the argument to account for the presence of expansion, the two main points causing qualitatively different behaviour compared to, in particular, \cite{GNR21} are due to changes in the relationship between $\tilde{F}[\rho]$ and $\tilde{F}[S]$ compared to the classical setting. First, the linear estimate in Corollary \ref{cor:lin-damping-gevrey-concrete} implies that there is a loss of decay by $\langle \tau\rangle^{\beta^\prime}$ when estimating the density by its source. Conversely, the nonlinearity in \eqref{eq:source-nl} behaves the most poorly for $\lvert k\rvert\approx \lvert l\rvert$, while the integral is concentrated at $k\tau\approx l\tilde{\tau}$. Thus, for the terms that impose sharp Gevrey requirements on the initial data in the classical setting, the scale factor acts approximately as an amplification by $\langle \tau\rangle^\beta$ when trying to relate $\tilde{F}[S]$ back to $\tilde{F}[\rho]$. Nevertheless, by imposing \eqref{eq:reg-main}, we can account for these amplifications in the proof of Lemma \ref{lem:ckl-pw} along with that arising from the resolvent bound, and thus to control the source on the bootstrap interval.\\

In particular, we note that this section is largely independent of the specific shape of the background equilibrium, or even whether the particle interaction is attractive or repulsive, once an estimate as in Corollary \ref{cor:lin-damping-gevrey-concrete} is obtained. While there are certainly limitations on obtaining such an estimate for attractive particle interactions in $d=3$ even for the Poisson equilibrium, as we discuss in Appendix \ref{sec:4+}, it is possible that this section can be applied to other equilibria in the repulsive setting if a similar linear estimate is obtained by other means. In this spirit, the nonlinear estimates in Section \ref{sec:nonlin} are performed without using the specific form of the background aside from when the linear damping results need to be invoked, instead only relying on the weaker regularity assumption \eqref{eq:distr-ass}.

\section{Linear damping}\label{sec:lin}

In this section, we study the Volterra equation \eqref{eq:volterra-abstract} and apply what we obtain to the special case \eqref{eq:volterra-outline} from the Vlasov-Poisson system. To this end, we use some standard statements on the Liouville-Green approximation that we have collected in Appendix \ref{sec:ode} for the sake of convenience.

\begin{prop}[Pointwise resolvent bound for the Poisson equilibrium]\label{prop:res-bd}
Let $\phi_k$ be the solution to the Volterra equation \eqref{eq:volterra}. Then, the solution to the Volterra equation can be written as
\begin{subequations}
\begin{equation}\label{eq:res-expr}
\phi_k(\tau)=s_k(\tau)+\int_0^\tau R_k(\tau,\tilde{\tau})\,s_k(\tilde{\tau})\,d\tilde{\tau}
\end{equation}
and one has
\begin{equation}\label{eq:res-bd}
\lvert R_k(\tau,\tilde{\tau})\rvert\lesssim (\tau-\tilde{\tau})^2\,\exp(-\theta_0\lvert k\rvert(\tau-\tilde{\tau}))\,a(T(\tilde{\tau}))\,a(T({\tau}))^\frac12.
\end{equation}
\end{subequations}
\end{prop}
%\begin{remark}[\todo{Conjecture -- improved bound}]
%\todo{I strongly suspect one can remove $a(T({\tau}))^\frac12$ or a similar factor of this type -- this arises from a slightly brutal estimation of an oscillatory integral. However, this bound still yields a damping result, so I'll keep it for now.}
%\end{remark}
Before moving on to the proof, we will need the following technical lemma to apply the Liouville-Green approximation below.
\begin{lemma}\label{lem:liouville-green-error-bound}
Under Assumption \ref{ass:scale}, define $F_a:[0,\infty)\longrightarrow\R$ by
\[F_a(\tau)=\int_0^\tau a(T(\tilde{\tau}))^{-\frac14}\,\left((a\circ T)^{-\frac14}\right)^{\prime\prime}(\tilde{\tau})\,d\tilde{\tau}\,.\]
Then $F_a$ has finite variation on $(0,\infty)$ and, for any $b\in[0,\infty)$ its variation $\mathcal{V}_{b,\tau}[F]$ on $(b,\tau)$ converges to zero as $\tau\downarrow b$.
%\todo{Furthermore, one has
%\begin{equation}\label{eq:LG-quantitative-error-term}
%\int_b^\tau \left[\exp(\mathcal{V}_{b,\tau})-1\right]\,a(T(\tau))^\frac12\,d\tau\lesssim 1
%\end{equation}
%}
\end{lemma}
\begin{proof}
Note that 
\begin{align*}
\mathcal{V}_{b,\tau}[F_a]=\int_b^\tau \lvert F_a^\prime(\tilde{\tau})\rvert\,d\tilde{\tau}=\int_b^{\tau} a(T(\tilde{\tau}))^{-\frac14}\,\left\lvert\left((a\circ T)^{-\frac14}\right)^{\prime\prime}(\tilde{\tau})\right\rvert\,d\tilde{\tau}\,.
\end{align*}
Thus, Assumption \ref{ass:scale} immediately implies that $F_a$ has finite variation, and since $F_a$ is clearly continuously differentiable,  $\lim_{\tau\downarrow b}\mathcal{V}_{b,\tau}[F_a]=0$ follows immediately.
%We compute
%\begin{align*}
%\del_\tau\langle\tau\rangle^{-\frac{\beta}4}=&\,-\frac{\beta}4\,\langle \tau\rangle^{-\frac{\beta}4-1}\,\frac{\tau}{\langle \tau\rangle}\\
%\del_\tau^2\langle\tau\rangle^{-\frac{\beta}4}=&\,-\frac{\beta}4\,\langle \tau\rangle^{-2}\left(\langle \tau\rangle^{-\frac{\beta}4}-((1+\frac{\beta}4)+1)\tau\,\langle \tau\rangle^{-\frac{\beta}4-1}\,\frac{\tau}{\langle \tau\rangle}\right)\\
%=&\,-\frac{\beta}4\,\langle \tau\rangle^{-2-\frac{\beta}4}\,\left(1-(2+\frac{\beta}4)\frac{\tau^2}{\langle\tau\rangle^2}\right)\\
%=&\,\frac{\beta}4\,\langle \tau\rangle^{-2-\frac{\beta}4}\left(1+\frac{\beta}4-\frac{2+\frac{\beta}4}{\langle \tau\rangle^2}\right)
%\end{align*}
%and thus
%\begin{align*}
%\mathcal{V}_{b,\tau}[F_a]\leq&\,\frac{\beta}4\left(1+\frac{\beta}4\right)\int_b^\tau \langle\tilde{\tau}\rangle^{-2-\frac{\beta}2}\,d\tilde{\tau}\lesssim \frac{\beta}4\,(1+\frac{\beta}4)\min\left\{\tau-b,1+\tau^{-1}\right\}\lesssim \frac{\beta}2(1+\frac{\beta}4).
%\end{align*}
\end{proof}
\begin{proof}[Proof of Proposition \ref{prop:res-bd}]
As in \cite[(1.2.9f.)]{Bru17}, the resolvent kernel $R_k:\R\times\R\longrightarrow\C$ is a continuous function satisfying the integral equation
\begin{align}
\label{eq:resolvent-volterra}R_k(\tau,\tilde{\tau})=&\,K_k(\tau,\tilde{\tau})+\int_{\tilde{\tau}}^\tau K_k({\tau},s)\,R_k(s,\tilde{\tau})\,ds\,,%\\
%=&\,K_k(\tau,\tilde{\tau})+\int_{\tilde{\tau}}^\tau R_k({\tau},v)\,K_k(v,\tilde{\tau})\,d\tilde{\tau}
\end{align}
recalling $K_k$ from \eqref{eq:def-k}. Furthermore, we restrict ourselves to $\tau\geq\tilde{\tau}\geq 0$ since $R_k$ vanishes otherwise due to the structure of $K_k$. Treating $\tilde{\tau}$ as a parameter, let us consider the auxiliary function $u_ {\tilde{\tau}}:[\tau,\infty)\longrightarrow\C$ given by
\begin{align*}
\numberthis\label{eq:def-u}u_{\tilde{\tau}}(\tau):=&\,\exp(\theta_0\lvert k\rvert(\tau-\tilde{\tau}))\left[K_k(\tau,\tilde{\tau})-R_k(\tau,\tilde{\tau})\right]=-\exp(\theta_0\lvert k\rvert(\tau-\tilde{\tau}))\int_{\tilde{\tau}}^\tau K_k({\tau},s)\,R_k(s,\tilde{\tau})\,ds\\
=&\,4\pi\int_{\tilde{\tau}}^\tau (\tau-s)\exp(\theta_0\lvert k\rvert(s-\tilde{\tau}))R_k(s,\tilde{\tau})\,a(T(s))\,ds\,.
\end{align*}
One computes that
\begin{align*}
u_{\tilde{\tau}}^\prime(\tau)=&\,\del_\tau\left[4\pi\int_{\tilde{\tau}}^\tau (\tau-s)\exp(\theta_0\lvert k\rvert(s-\tilde{\tau}))\,a(T(s))\,R_k(s,\tilde{\tau})\,ds\right]\\
=&\,4\pi\int_{\tilde{\tau}}^\tau \exp(\theta_0\lvert k\rvert(s-\tilde{\tau}))\,a(T(s))\,R_k(s,\tilde{\tau})\,ds
\end{align*}
holds and consequently, using \eqref{eq:def-u},
\begin{align*}
u_{\tilde{\tau}}^{\prime\prime}(\tau)=&\,4\pi\,\exp(\theta_0\lvert k\rvert(\tau-\tilde{\tau}))\,a(T(\tau))\,R_k(\tau,\tilde{\tau})\\
=&\,4\pi\,a(T(\tau))\,[\exp(\theta_0\lvert k\rvert(\tau-\tilde{\tau}))\,K_k(\tau,\tilde{\tau})-u_{\tilde{\tau}}(\tau)].
\end{align*}
Inserting \eqref{eq:def-k} and rearranging, this yields
\begin{align*}
u_{\tilde{\tau}}^{\prime\prime}(\tau)+4\pi\,a(T(\tau))\,u_{\tilde{\tau}}(\tau)
=&\,16\pi^2 a(T(\tau))\,a(T(\tilde{\tau}))(\tau-\tilde{\tau})\,.
\end{align*}
We can apply the Liouville-Green-approximation to this parameter-dependent second order ODE, see Corollary \ref{cor:LG-inhom}. It follows that $u_{\tilde{\tau}}(\tau)$ is of the form
\begin{align*}
\numberthis\label{eq:u-explicit} u_{\tilde{\tau}}(\tau)=&\,C^1(\tilde{\tau})\,u_{\tilde{\tau}}^{1}(\tau)+C^2(\tilde{\tau})\,u_{\tilde{\tau}}^{2}(\tau)\\
&\,+16\pi^2\int_{\tilde{\tau}}^\tau \left[u_{\tilde{\tau}}^1(s)u_{\tilde{\tau}}^2({\tau})-u_{\tilde{\tau}}^2(s)u_{\tilde{\tau}}^1({\tau})\right]\,a(T(s))\,a(T(\tilde{\tau}))(s-\tilde{\tau})\,ds
\end{align*}
where $C(\tilde{\tau})$ is a complex-valued function depending on $\tilde{\tau}$, we write
\[\xi_{\tilde{\tau}}(\tau)=\sqrt{4\pi}\,\int_{\tilde{\tau}}^\tau \,a(T(s))^\frac12\,ds\,,\]
and the fundamental solutions $u^{1,2}_{\tilde{\tau}}$ are of the form
\begin{align*}
u_{\tilde{\tau}}^1(\tau)=(4\pi)^{-\frac14}\,a(T(\tau))^{-\frac14}\left[\sin\left(\xi_{\tilde{\tau}}(\tau)\right)+\epsilon(\xi_{\tilde{\tau}}(\tau))\right]\,,\\
u_{\tilde{\tau}}^2(\tau)=(4\pi)^{-\frac14}\,a(T(\tau))^{-\frac14}\left[\cos\left(\xi_{\tilde{\tau}}(\tau)\right)+\epsilon(\xi_{\tilde{\tau}}(\tau))\right]
\end{align*}
for a given error function $\epsilon:(0,\infty)\longrightarrow\C$ as in Lemma \ref{lem:LG-hom}. In particular, by Lemma \ref{lem:liouville-green-error-bound} , $\epsilon\circ\xi_{\tilde{\tau}}$ and $(\epsilon\circ\xi_{\tilde{\tau}})^\prime$ are uniformly bounded, $(\epsilon\circ\xi_{\tilde{\tau}})(\tau)$ and $(\epsilon\circ\xi_{\tilde{\tau}})^\prime(\tau)$ converge to $0$ as $\tau\downarrow\tilde{\tau}$, and thus also 
\[u^1_{\tilde{\tau}}(\tilde{\tau})=0,\quad u^2_{\tilde{\tau}}(\tilde{\tau})=\left(4\pi a(T(\tilde{\tau}))\right)^{-\frac14}\,.\]
%Altogether, \eqref{eq:u-explicit} implies the following formula for the resolvent for $\tau\geq\tilde{\tau}\geq0$.
%\begin{align*}
%R_k(\tau,\tilde{\tau})=&\,K_k(\tau,\tilde{\tau})-\exp(-\theta_0\lvert k\rvert(\tau-\tilde{\tau}))u_{\tilde{\tau}}(\tau)\\
%=&\,-\exp(-\theta_0\lvert k\rvert\,(\tau-\tilde{\tau}))\left\{4\pi\,(\tau-\tilde{\tau})\,a(T(\tilde{\tau}))+C^1(\tilde{\tau})\,u_{\tilde{\tau}}^{1}(\tau)+C^2(\tilde{\tau})\,u_{\tilde{\tau}}^{2}(\tau)\right\}\\
%&\,+16\pi^2\,\exp(-\theta_0\lvert k\rvert(\tau-\tilde{\tau}))\int_{\tilde{\tau}}^\tau\left[u_{\tilde{\tau}}^1(s)u_{\tilde{\tau}}^2({\tau})-u_{\tilde{\tau}}^2(s)u_{\tilde{\tau}}^1({\tau})\right]\,a(T(s))\,a(T(\tilde{\tau}))(s-\tilde{\tau})\,ds\,.\numberthis\label{eq:R-explicit}
%\end{align*}
It remains to determine the functions $C^1$ and $C^2$. Note that one clearly has $u_{\tilde{\tau}}(\tilde{\tau})=0$ by definition, which immediately implies $C^2\equiv 0$. To determine $C^1$, we compute that, for any $\tau\geq\tilde{\tau}$, one has
\begin{align*}
(u_{\tilde{\tau}}^1)^\prime(\tau)=&\,-\frac14 a(T(\tau))^{-1}\,(a\circ T)^\prime(\tau)\,(u_{\tilde{\tau}}^1)(\tau)\\
&\,+(4\pi)^{-\frac14}\,a(T(\tau))^{-\frac14}\left[\cos\left(\sqrt{4\pi}\,\int_{\tilde{\tau}}^\tau a(T(s))^\frac12\,ds\right)\,\sqrt{4\pi}\,a(T(\tau))^\frac12+(\epsilon\circ \xi_{\tilde{\tau}})^\prime(\tau)\right]\,.
\end{align*}
Thus, since $u^1_{\tilde{\tau}}({\tau})$ and $(\epsilon\circ \xi_{\tilde{\tau}})^\prime({\tau})$ vanish as $\tau$ approaches $\tilde{\tau}$, we obtain
\begin{align*}
(u_{\tilde{\tau}}^1)^\prime(\tilde{\tau})=&\,(4\pi)^{\frac14}\,a(T(\tilde{\tau}))^{\frac14}.
\end{align*}
On the other hand, since $K_k(\tau,\tau)=0$, differentiating \eqref{eq:resolvent-volterra} implies
\begin{align*}
\left.\left(\del_\tau R_k(\cdot,\tilde{\tau})\right)\right|_{\tau=\tilde{\tau}}=\left.\left(\del_\tau K_k(\cdot,\tilde{\tau})\right)\right|_{\tau=\tilde{\tau}}\,,
\end{align*}
and thus also $u_{\tilde{\tau}}^\prime(\tilde{\tau})=0$. This yields
%Consequently, using $u_{\tilde{\tau}}(\tilde{\tau})=0$,
\begin{align*}
%0=&\,-\del_\tau R_k(\tilde{\tau},\tilde{\tau})+\del_\tau K_k(\tilde{\tau},\tilde{\tau})\\
%=&\del_\tau\left[\exp(-\theta_0\lvert k\rvert(\tau-\tilde{\tau}))\,u_{\tilde{\tau}}(\tau)\right]\vert_{\tau=\tilde{\tau}}\\
0=&\,u^\prime_{\tilde{\tau}}(\tilde{\tau})\\
=&\,(4\pi)^{\frac14}\,C^1(\tilde{\tau})\,a(T(\tilde{\tau}))^{\frac14}\,\\
&\,+\left[16\,\pi^2\,\exp(-\theta_0\lvert k\rvert(\tau-\tilde{\tau}))\big[u_{\tilde{\tau}}^1(\tau)u_{\tilde{\tau}}^2({\tau})-u_{\tilde{\tau}}^2(\tau)u_{\tilde{\tau}}^1({\tau})\right]\,a(T(\tau))\,a(T(\tilde{\tau}))(\tau-\tilde{\tau})\big]\big\vert_{\tau=\tilde{\tau}}\\
=&\,(4\pi)^{\frac14}\,C^1(\tilde{\tau})\,\,a(T(\tilde{\tau}))^{\frac14}
\end{align*}
and hence $C^1\equiv 0$. Finally, it remains to estimate the surviving integral expression in \eqref{eq:u-explicit}. Note that one has
\begin{align*}
&\,\left\lvert\sqrt{4\pi}\,\left[u^1_{\tilde{\tau}}(s)u^2_{\tilde{\tau}}(\tau)-u^1_{\tilde{\tau}}(\tau)u^2_{\tilde{\tau}}(s)\right]\,a(T(s))^\frac14\,a(T(\tau))^\frac14\right\rvert\\
=&\,\Big\lvert\sin(\xi_{\tilde{\tau}}(s))\cos(\xi_{\tilde{\tau}}(\tau))-\sin(\xi_{\tilde{\tau}}(\tau))\cos(\xi_{\tilde{\tau}}(s))+\epsilon(\xi_{\tilde{\tau}}(s))\left[\cos(\xi_{\tilde{\tau}}(\tau))-\sin(\xi_{\tilde{\tau}}(\tau))\right]\\
&\,-\epsilon(\xi_{\tilde{\tau}}(\tau))\left[\cos(\xi_{\tilde{\tau}}(s))-\sin(\xi_{\tilde{\tau}}(s))\right]\Big\rvert\\
%=&\,-\sin\left(\sqrt{4\pi}\int_s^\tau a(T(s^\prime))^\frac12\,ds^\prime\right)+\sqrt{2}\,\epsilon(\xi_{\tilde{\tau}}(s))\sin\left(\frac\pi4-\xi_{\tilde{\tau}}(\tau)\right)-\sqrt{2}\,\epsilon(\xi_{\tilde{\tau}}(\tau))\sin\left(\frac\pi4-\xi_{\tilde{\tau}}(s)\right)
\lesssim&\,1+\|\epsilon\circ\xi_{\tilde{\tau}}\|_{L^\infty[\tilde{\tau},\infty)}
\end{align*}
In addition, Lemma \ref{lem:liouville-green-error-bound} and Corollary \ref{cor:LG-inhom} imply
\[\|\epsilon\circ \xi_{\tilde{\tau}}\|_{L^\infty[\tilde{\tau},\infty)}\leq \sup_{\tau\geq\tilde{\tau}}\left[ \exp(\mathcal{V}_{\tilde{\tau},\tau}[F])-1\right]\leq \exp(\mathcal{V}_{0,\infty}[F])-1\]
which is uniform in $\tilde{\tau}$. Thus, one gets
\begin{align*}
&\,\left\lvert\int_{\tilde{\tau}}^\tau \left[u_{\tilde{\tau}}^1(s)u_{\tilde{\tau}}^2({\tau})-u_{\tilde{\tau}}^2(s)u_{\tilde{\tau}}^1({\tau})\right]\,a(T(s))\,a(T(\tilde{\tau}))(s-\tilde{\tau})\,ds\right\rvert\\
\lesssim&\, a(T(\tilde{\tau})) \int_{\tilde{\tau}}^\tau \left(\frac{a(T(s))}{a(T(\tau))}\right)^\frac14\,(s-\tilde{\tau})\,a(T(s))^\frac12\,ds\\
%\lesssim&\,a(T(\tilde{\tau}))a(T(\tau))^{-\frac14}\,\left\lvert\int_{\tilde{\tau}}^\tau \left[\sin\left(\sqrt{4\pi}\,\int_s^\tau a(T(s^\prime))^\frac12\,ds^\prime\right)+\sqrt{2}\,\epsilon(\xi_{\tilde{\tau}}(s))\sin\left(\frac\pi4-\xi_{\tilde{\tau}}(\tau)\right)\right.\right.\\
%&\quad\quad\quad\left.\left.+\epsilon(\xi_{\tilde{\tau}}(\tau))\sin\left(\frac\pi4-\xi_{\tilde{\tau}}(r)\right)\right]\,a(T(s))^\frac34\,(s-\tilde{\tau})\,ds\right\rvert\\
%\lesssim&\,a(T(\tilde{\tau}))(\tau-\tilde{\tau})\left[\int_{0}^{\xi_{\tilde{\tau}}(\tau)}\,\lvert\sin(\zeta)\rvert+\lvert \epsilon(\zeta)\rvert+\left\lvert \sin\left(\frac\pi4-\zeta\right)\right\rvert d\zeta\right]\\
\lesssim&\,a(T(\tilde{\tau}))(\tau-\tilde{\tau})\,\xi_{\tilde{\tau}}(\tau)\,.
\end{align*}
Altogether, recalling the definition of $u_{\tilde{\tau}}$, this implies the following estimate for any $(\tau,\tilde{\tau})$ with $\tau\geq\tilde{\tau}\geq 0$.
\begin{align*}
\lvert R_k(\tau,\tilde{\tau})\rvert\leq&\, \lvert K_k(\tau,\tilde{\tau})\rvert+\exp(-\theta_0\lvert k\rvert(\tau-\tilde{\tau}))\,u_{\tilde{\tau}}(\tau)\\
\lesssim&\,\exp(-\theta_0\lvert k\rvert(\tau-\tilde{\tau}))\left[(\tau-\tilde{\tau})\,a(T(\tilde{\tau}))+(\tau-\tilde{\tau})\,a(T(\tilde{\tau}))\,\xi_{\tilde{\tau}}(\tau)\right]\\
\lesssim&\,(\tau-\tilde{\tau})^2\,\exp(-\theta_0\lvert k\rvert(\tau-\tilde{\tau}))\,a(T(\tilde{\tau}))\,a(T(\tau))^\frac12\,.
\end{align*}

\end{proof}

Applying this resolvent bound to the original mode equations in \eqref{eq:volterra} implies the following generator function estimate.

\begin{corollary}[Linear Gevrey estimate]\label{cor:lin-damping-gevrey-concrete}
Assume the conditions of Proposition \ref{prop:res-bd} and Assumption \ref{ass:scale}, and take $\beta^\prime=\frac32\beta$. Then, there exists some parameter $\theta_1\in(0,0.25\theta_0]$ and some constant $C>0$ such that the following bounds hold for any $\tau\geq 0$, $\eta>0$ and $z\in[0,1]$ for $\gamma<1$, respectively $z\in[0,0.5\theta_0)$ for $\gamma=1$.
\begin{subequations}
\begin{align}
\label{eq:generator-bd}\langle \tau\rangle^\eta\,F[\rho](\tau,z)\leq&\,\langle \tau\rangle^\eta\,F[S](\tau,z)+C\int_0^\tau \exp(-\theta_1(\tau-\tilde{\tau}))\,\langle{\tilde{\tau}}\rangle^{\beta^\prime}\,\langle \tilde{\tau}\rangle^\eta\,F[S](\tilde{\tau},z)\,d\tilde{\tau}\,,\\
\label{eq:generator-bd-slide}\langle{\tau}\rangle^{\eta}\tilde{F}[\rho](\tau)\leq&\,\langle \tau\rangle^\eta\,\tilde{F}[S](\tau)+C\int_0^\tau \exp(-\theta_1(\tau-\tilde{\tau}))\,\langle{\tilde{\tau}}\rangle^{\beta^\prime}\,\langle \tilde{\tau}\rangle^\eta\,\tilde{F}[S](\tilde{\tau})\,d\tilde{\tau}\,.
\end{align}
\end{subequations}

\end{corollary}
\begin{proof}
Rescaling \eqref{eq:res-expr} by $\langle\tau\rangle^\eta\,\lvert k\rvert^{-\alpha}$ for $k\in\Z^3\backslash\{0\}$, one obtains by Definition \ref{def:gevrey} that
\begin{align*}
\langle \tau\rangle^\eta\,\lvert k\rvert^{-\alpha}\,A_{k,k\tau}(z)\,\lvert\hat{\rho}_k(\tau)\rvert\leq&\,\langle \tau\rangle^\eta\,\lvert k\rvert^{-\alpha}\,A_{k,k\tau}(z)\left(\lvert\hat{S}_k(\tau)\rvert+\int_0^\tau \left\lvert R_{k}(\tau,\tilde{\tau})\right\rvert\,\left\lvert\hat{S}_k(\tilde{\tau})\right\rvert\,d\tilde{\tau}\right)\\
\leq&\,\langle \tau\rangle^\eta\,F[S](\tau,z)+\int_0^\tau\frac{\langle \tau\rangle^\eta\,A_{k,k\tau}(z)}{\langle \tilde{\tau}\rangle^\eta\,A_{k,k\tilde{\tau}}(z)}\,\left\lvert R_{k}(\tau,\tilde{\tau})\right\rvert\,\langle \tilde{\tau}\rangle^\eta\,F[S](\tilde{\tau},z)\,d\tilde{\tau}\,.
\end{align*}
Recall that, by Lemma \ref{lem:algebraic}, one has
\begin{align*}
\left\lvert\langle k,k\tau\rangle^\gamma-\langle k,k\tilde{\tau}\rangle^\gamma\right\rvert\lesssim&\,\langle k(\tau-\tilde{\tau})\rangle^\gamma\,,\\
\langle k,k\tau\rangle^\sigma\lesssim&\,\langle k,k\tilde{\tau}\rangle^\sigma\,\langle k(\tau-\tilde{\tau})\rangle^\sigma\,,\\
\langle \tau\rangle^{\eta+\frac{\beta}2}\lesssim&\,\langle \tau-\tilde{\tau}\rangle^{{\eta+\frac{\beta}2}}\langle\tilde{\tau}\rangle^{\eta+\frac{\beta}2}\,.
\end{align*}
Further using \eqref{eq:res-bd} for the resolvent term as well as Assumption \ref{ass:scale}, we compute for $\tau\geq\tilde{\tau}$:
\begin{align*}
\frac{A_{k,k\tau}(z)\,\langle \tau\rangle^\eta}{A_{k,k\tilde{\tau}}(z)\,\langle \tilde{\tau}\rangle^\eta}\,\left\lvert R_{k}(\tau,\tilde{\tau})\right\rvert\lesssim&\,\exp\left(z\left(\langle k,k\tau\rangle^\gamma-\langle k,k\tilde{\tau}\rangle^\gamma\right)-\theta_0\lvert k\rvert(\tau-\tilde{\tau})\right)\,\langle k(\tau-\tilde{\tau})\rangle^{\sigma+\eta+\frac{\beta}2+2}\,\langle\tilde{\tau}\rangle^{\frac32\beta}\\
\leq&\,\exp\left(z\langle k(\tau-\tilde{\tau})\rangle^\gamma-\theta_0\lvert k\rvert(\tau-\tilde{\tau})\right)\,\langle k(\tau-\tilde{\tau})\rangle^{\sigma+\eta+\frac{\beta}2+2}\,\langle\tilde{\tau}\rangle^{\frac32\beta}\,.
\end{align*}
Furthermore, for $\gamma<1$, the exponential term is bounded by $C_{\theta_0} \exp(-0.5\theta_0\lvert k\rvert(\tau-\tilde{\tau}))$ for some $C_{\theta_0}>0$ independent of $\lvert k\rvert$ since $y\mapsto z\langle y\rangle^\gamma-0.5\theta_0\,\lvert y\rvert$ is bounded from above on the real line. For $\gamma=1$, this also holds since $z\in[0,0.5\theta_0)$. Either way, we get
\begin{align*}
\left\lvert \frac{A_{k,k\tau}(z)\,\langle \tau\rangle^\eta}{A_{k,k\tilde{\tau}}(z)\,\langle \tilde{\tau}\rangle^\eta}\,R_{k}(\tau,\tilde{\tau})\right\rvert\lesssim&\,\exp\left(-0.5\theta_0\lvert k\rvert\,(\tau-\tilde{\tau})\right)\,\langle k(\tau-\tilde{\tau})\rangle^{\sigma+\eta+\frac{\beta}2+2}\,\langle\tilde{\tau}\rangle^{\frac32\beta}\,.
\end{align*}
Since $x\mapsto \langle x\rangle^{\sigma+\eta+\frac{\beta}2+2} \exp\left(-0.25\theta_0\,x\right)$ is uniformly bounded on the positive real line, \eqref{eq:generator-bd} now follows for $\theta_1\in(0,0.25\theta_0]$.\\
The sliding generator function bound \eqref{eq:generator-bd-slide} follows similarly, where $\tau\mapsto z(\tau)$ being decreasing implies
\begin{align*}
\exp\left(z(\tau)\langle k,k\tau\rangle^\gamma-z(\tilde{\tau})\langle k,k\tilde{\tau}\rangle^\gamma\right)\leq&\,\exp\left(z(\tilde{\tau})\left(\langle k,k\tau\rangle^\gamma-\langle k,k\tilde{\tau}\rangle^\gamma\right)\right)\\
\leq&\,\exp\left(2\lambda_1\left(\langle k,k\tau\rangle^\gamma-\langle k,k\tilde{\tau}\rangle^\gamma\right)\right)\,.
\end{align*}
In particular, note that in the analytic case, one has $z(\tau)\leq 2\lambda_1<0.5\lambda_0<0.5\theta_0$, see Definition \ref{def:slide}.
\end{proof}

\section{Nonlinear damping}\label{sec:nonlin}

In this section, we focus on providing the necessary pointwise estimates for the generator functions introduced in Definition \ref{def:slide} to improve the bootstrap assumption, see Assumption \ref{ass:boot}. While the proof follows a similar outline to \cite[Section 4]{GNR21}, we opt to provide a fully self-contained exposition here, especially to make the role of expansion as clear as possible and to track differences arising from the proof of \cite{GNR21} being written for $d=1$. To elucidate the role of regularity and expansion parameters, we also state in each step for which parameter choices each step holds and, in particular, when Assumption \ref{ass:scale} on the scale factor is used. Furthermore, we do not use the explicit form of the Poisson distribution in this section except when we need to rely on results from linear analysis in Section \ref{sec:lin}. Instead, we will only use that the background $\mu$ is smooth and that there exists some $\theta_0>0$ such that, for $j=0,\dots,3$,
\begin{equation}\label{eq:distr-ass}
\left\lvert \left(D^{j}_\xi\hat{\mu}\right)(\xi)\right\rvert\lesssim e^{-\theta_0\lvert \xi\rvert}\,.
\end{equation}
Clearly, this holds for the Poisson background itself, see \eqref{eq:poisson-fourier}\,.

\subsection{Generator function hierarchy}\label{subsec:hierarchy}

Before dealing with the nonlinear source in \eqref{eq:source-nl}, we first provide the basic differential estimate underlying our bootstrap argument.

\begin{lemma}[Differential bound on ${G[h]}$]\label{lem:delt-G} For $\gamma>0,\sigma>2$ and $\alpha\leq 1$, the following holds for any $\tau>0$ and $z\in[0,1]$.
\begin{equation}\label{eq:G-evol-ineq}
\del_\tau G[h](\tau,z)\lesssim a(T(\tau))\,F[\rho](\tau,z)\,\sqrt{G[h](\tau,z)}+a(T(\tau))\,\langle\tau\rangle\,F[\rho](\tau,z)\,\del_zG[h](\tau,z)\,.
\end{equation}
\end{lemma}
\begin{proof}[Proof-Sketch following {\cite[Proposition 4.1]{GNR21}}] 
Throughout this proof, we will suppress the dependency of the Fourier multiplier $A(z)$ on $z$ in notation. Furthermore, recall that since $g$ is charge-neutral, we have $\hat{\rho}_0\equiv 0$ which we will use to drop any terms with zero density modes without further comment.\\

Notice that one has
\begin{equation}\label{eq:delt-G}
\del_\tau G[h](\tau,\cdot)\leq 2\,\Re\left\{\sum_{j=0}^{3}\sum_{k\in\Z^3\backslash\{0\}}\int_{\R^3}\left(A_{k,\xi}\,(D_\xi^j\del_\tau\hat{h})(\tau,k,\xi)\right)\left(\overline{A_{k,\xi}\,(D_\xi^j\hat{h})(\tau,k,\xi)}\right)\,d\xi\right\},
\end{equation}
so we insert \eqref{eq:landau-fourier-commuted} for $j=0,\dots,3$ and treat contributions from $(L)$, $(R)$ and $(P)$ therein separately.\\

\noindent Using \eqref{eq:distr-ass} as well as, by Lemma \ref{lem:algebraic},
\[A_{k,\xi}\leq 2^\sigma\,A_{k,k\tau}\,\langle \xi-k\tau\rangle^\sigma\,,\]
one gets that contributions from $(L)$ are bounded by
\begin{align*}
\lesssim&\,a(T(\tau))\sum_{j=0}^3\sum_{k\in\Z^3\backslash\{0\}}\int_{\R^3}\lvert k\rvert^{-\alpha}\,\left\lvert A_{k,k\tau}\,\hat{\rho}_k(\tau)\right\rvert\,\left\lvert A_{k,\xi}\,(D_\xi^j\hat{h})(\tau,k,\xi))\right\rvert\,\cdot\\
&\,\qquad\qquad\qquad\qquad\qquad\cdot\left(\frac{\lvert\xi-k\tau\rvert}2^\sigma\,e^{-\theta_0\lvert \xi-k\tau\rvert}\,\langle \xi-k\tau\rangle^\sigma\,\lvert k\rvert^{-1+\alpha}\right)\,d\xi\\
\lesssim&\,a(T(\tau))\,F[\rho](\tau,\cdot)\,\sqrt{G[h](\tau,\cdot)}\,.
\end{align*}
Considering the contribution of $(R)$ to \eqref{eq:delt-G}, this takes the form
\begin{align*}
\numberthis\label{eq:contr-R-deltG}&\,8\pi \lvert k\rvert^{-\alpha}\,\Re\left\{\sum_{j=0}^3\sum_{k,l\in\Z^3\backslash\{0\}}\int_{\R^3} a(T(\tau))\,\frac{l\cdot(\xi-k\tau)}{\lvert l\rvert^2}\,\hat{\rho_l}(\tau)\,A_{k,\xi}\,(D_\xi^j\hat{h})(\tau,k-l,\xi-l\tau)\right.\\
&\,\qquad\qquad\qquad\cdot\left(\overline{A_{k,\xi}\,(D_\xi^j\hat{h})(\tau,k,\xi)}\right)\,d\xi\Biggr\}\\
=&\,8\pi\lvert k\rvert^{-\alpha}\,\Re\left\{\sum_{j=0}^3\sum_{k,l\in\Z^3\backslash\{0\}}\int_{\R^3}a(T(\tau))\,\left(\frac{l\cdot(\xi-k\tau)}{\lvert l\rvert^2}\,\hat{\rho_l}(\tau)\right)\cdot\left(A_{k-l,\xi-l\tau}\,(D_\xi^j\hat{h})(\tau,k-l,\xi-l\tau)\right)\cdot\right. \\
&\,\qquad\qquad\qquad\cdot\left(\overline{A_{k,\xi}\,(D_\xi^j\hat{h})(\tau,k,\xi)}\right)\,d\xi\Biggr\}\\
&\,+8\pi\lvert k\rvert^{-\alpha}\,\Re\left\{\sum_{j=0}^3\sum_{k,l\in\Z^3\backslash\{0\}}\int_{\R^3} a(T(\tau))\,A_{l,l\tau}\hat{\rho_l}(\tau)\,\frac{k\cdot(\xi-k\tau)}{\lvert k\rvert^2}\,\frac{A_{k,\xi}-A_{k-l,\xi-l\tau}}{A_{l,l\tau}\,A_{k-l,\xi-l\tau}}\,\cdot\right.\\
&\,\qquad\qquad\qquad\qquad\qquad \cdot\left(A_{k-l,\xi-l\tau}(D_\xi^j\hat{h})(\tau,k-l,\xi-l\tau)\right)\left(\overline{A_{k,\xi}\left(D_\xi^j\hat{h}\right)(\tau,k,\xi)}\right)\,d\xi\Biggr\}\,.
\end{align*}
The first term on the right hand side vanishes since, relabeling 
\[\tilde{k}=k-l,\ \tilde{l}=-l,\ \tilde{\xi}=\xi-l\tau\]
and using $\hat{\rho}_{{l}}=\hat{\rho}_{-\tilde{l}}=\overline{\hat{\rho}_{\tilde{l}}}$, one sees that the series expression is equal to
\begin{align*}
-\sum_{j=0}^3\sum_{\tilde{k},\tilde{l}\in\Z^3\backslash\{0\}}\int_{\R^3}&\, a(T(\tau))\,\left(\frac{\tilde{l}\cdot (\tilde{\xi}-\tilde{k}\tau)}{\lvert \tilde{l}\rvert^2}\,\overline{\hat{\rho}_{\tilde{l}}}\right)\left(A_{\tilde{k},\tilde{\xi}}\,(D_\xi^j\hat{h})(\tau,\tilde{k},\tilde{\xi})\right)\cdot\\
&\,\cdot\left(\overline{A_{\tilde{k}-\tilde{l},\tilde{\xi}-\tilde{l}\tau}\,(D^j_\xi\hat{h})(\tau,\tilde{k}-\tilde{l},\tilde{\xi}-\tilde{l}\tau)}\right)\,d\tilde{\xi}
\end{align*}
and thus purely imaginary. To control the remaining term, we first estimate the weight
\[\lvert \xi-k\tau\rvert\frac{\lvert A_{k,\xi}-A_{k-l,\xi-l\tau}\rvert}{A_{l,l\tau}\,A_{k-l,\xi-l\tau}}\,.\]
First, observe that one has, for any $l\in\Z^3$,
\begin{equation}\label{eq:first-case-distinction}
\langle \xi-k\tau\rangle\leq \langle\xi-l\tau\rangle+\langle k-l\rangle\,\langle\tau\rangle\lesssim \langle k-l,\xi-l\tau\rangle\,\langle\tau\rangle\,.
\end{equation}
On the multiplier term, since
\[\langle k,\xi\rangle\leq \langle l,l\tau\rangle+\langle k-l,\xi-l\tau\rangle\]
holds, we distinguish two cases:\\

\noindent\textbf{Case 1}: $\langle l,l\tau\rangle\geq\frac12\langle k,\xi\rangle$. 
Then, one has
\begin{align*}
\frac{\lvert A_{k,\xi}-A_{k-l,\xi-l\tau}\rvert}{A_{l,l\tau}\,A_{k-l,\xi-l\tau}}
\lesssim&\,A^{-1}_{k-l,\xi-l\tau}+A^{-1}_{k,\xi}\lesssim \langle k-l,\xi-l\tau\rangle^{-\sigma}+\langle k,\xi\rangle^{-\sigma}
\end{align*}
and consequently, by \eqref{eq:first-case-distinction},
\begin{align*}
\lvert \xi-k\tau\rvert\frac{\lvert A_{k,\xi}-A_{k-l,\xi-l\tau}\rvert}{A_{l,l\tau}\,A_{k-l,\xi-l\tau}}\lesssim&\,\langle\tau\rangle\left(\langle k-l,\xi-l\tau\rangle^{-\sigma+1}+\langle k,\xi\rangle^{-\sigma+1}\right)\lesssim \langle \tau\rangle\,\left(\langle k-l\rangle^{-\sigma+1}+\langle l\rangle^{-\sigma+1}\right)\,.
\end{align*}
\textbf{Case 2}: $\langle k-l,\xi-l\tau\rangle\geq\frac12\langle k,\xi\rangle$. By Taylor's theorem, one obtains the following inequality for any $y,y^\prime\in \R^{6}$.
\begin{align*}
\lvert e^{z\langle y\rangle^\gamma}\langle y\rangle^\sigma-e^{z\langle y^\prime\rangle^\gamma}\langle y^\prime\rangle^\sigma\rvert\lesssim&\,e^{z\langle y\rangle^\gamma}\lvert\langle y\rangle^\gamma-\langle y^\prime\rangle^\gamma\rvert\,\,\left[e^{z\langle y\rangle^\gamma}\langle y\rangle^\sigma+e^{z\langle y^\prime\rangle^\gamma}\langle y^\prime\rangle^\sigma\right]\\
\lesssim&\, \frac{\langle y-y^\prime\rangle}{\langle y\rangle^{1-\gamma}+\langle y^\prime\rangle^{1-\gamma}}\,\left[e^{z\langle y\rangle^\gamma}\langle y\rangle^\sigma+e^{z\langle y^\prime\rangle^\gamma}\langle y^\prime\rangle^\sigma\right]\,.
\end{align*}
This implies, using \eqref{eq:algebra-prop} in the second step and \eqref{eq:first-case-distinction} in the third,
\begin{align*}
\lvert \xi-k\tau\rvert&\,\frac{\lvert A_{k,\xi}-A_{k-l,\xi-l\tau}\rvert}{A_{l,l\tau}\,A_{k-l,\xi-l\tau}}\lesssim\langle \xi-k\tau\rangle\frac{\langle l,l\tau\rangle}{\langle k,\xi\rangle^{1-\gamma}+\langle k-l,\xi-l\tau\rangle^{1-\gamma}}\frac{A_{k,\xi}+A_{k-l,\xi-l\tau}}{A_{l,l\tau}\,A_{k-l,\xi-l\tau}}\\
\lesssim &\,\frac{\langle \xi-k\tau\rangle\langle l,l\tau\rangle^{-\sigma+1}}{\langle k,\xi\rangle^{1-\gamma}+\langle k-l,\xi-l\tau\rangle^{1-\gamma}}\\
\lesssim&\,\sqrt{\frac{\langle k,\xi\rangle\langle\tau\rangle}{{\langle k,\xi\rangle^{1-\gamma}+\langle k-l,\xi-l\tau\rangle^{1-\gamma}}}}\sqrt{\frac{\langle k-l,\xi-l\tau\rangle\langle\tau\rangle}{\langle k,\xi\rangle^{1-\gamma}+\langle k-l,\xi-l\tau\rangle^{1-\gamma}}}\langle l,l\tau\rangle^{-\sigma+1}\\
\lesssim&\,\langle k,\xi\rangle^\frac{\gamma}2\langle k-l,\xi-l\tau\rangle^{\frac{\gamma}2}\langle l\rangle^{-\sigma+1}\langle\tau\rangle\,.
\end{align*}
Combining both cases, we altogether obtain
\[\lvert \xi-k\tau\rvert\frac{\lvert A_{k,\xi}-A_{k-l,\xi-l\tau}\rvert}{A_{l,l\tau}\,A_{k-l,\xi-l\tau}}\lesssim\langle\tau\rangle\left(\langle k-l\rangle^{1-\sigma}+\langle l\rangle^{1-\sigma}\right)\,\langle k,\xi\rangle^{\frac{\gamma}2}\,\langle k-l,\xi-l\tau\rangle^{\frac{\gamma}2}\]
and thus, using the Young equality and $\alpha\leq 1$ in the second step, that we can bound the contribution from $(R)$ that remains on the right hand side of \eqref{eq:contr-R-deltG} by
\begin{align*}
\lesssim&\,a(T(\tau))\langle\tau\rangle\,\sum_{j=0}^3\sum_{k,l\in\Z^3\backslash\{0\}}\left(\langle k-l\rangle^{1-\sigma}+\langle l\rangle^{1-\sigma}\right)\lvert l\rvert^{\alpha-1}\cdot \left(\lvert l\rvert^{-\alpha}\,A_{l,l\tau}\,\lvert\hat{\rho}_l(\tau)\rvert\right)\cdot\\
&\,\quad\cdot\int_{\R^3}\left(\langle k,\xi\rangle^{\frac{\gamma}2}\,A_{k,\xi}\,\lvert(D_\xi^j\hat{h})(\tau,k,\xi)\rvert\right)\,\left(\langle k-l,\xi-l\tau \rangle^{\frac{\gamma}2}\,A_{k-l,\xi-l\tau}\,\lvert(D_\xi^j\hat{h})(\tau,k-l,\xi-l\tau)\rvert\right)\,d\xi\\
\lesssim&\,a(T(\tau))\langle\tau\rangle\,F[\rho](\tau,\cdot)\,\sum_{j=0}\sum_{k,l\in\Z^3\backslash\{0\}}\left(\langle k-l\rangle^{1-\sigma}+\langle l\rangle^{1-\sigma}\right)\lvert l\rvert^{\alpha-1}\cdot\\
&\,\cdot\left(\int_{\R^3} \langle k-l,\xi-l\tau\rangle^\gamma A_{k-l,\xi-l\tau}^2\,\left\lvert (D^{j}_\xi\hat{h})(\tau,k-l,\xi-l\tau)\right\rvert^2\,d\xi+\int_{\R^3} \langle k,\xi\rangle^\gamma A^2_{k,\xi}\,\left\lvert(D^{j}_\xi\hat{h})(\tau,k,\xi)\right\rvert^2\,d\xi\right)\\
\lesssim&\,a(T(\tau))\langle\tau\rangle\,F[\rho](\tau,\cdot)\,\del_zG[h](\tau,\cdot)\,.
\end{align*}
Regarding the remaining term arising from $(P)$, we argue exactly as above to obtain, one has, using $\alpha\leq 1$ in the last step,
\begin{align*}
\lesssim&\,a(T(\tau))\sum_{j=0}^3\sum_{k,l\in\Z^3\backslash\{0\}}\int_{\R^3}\frac{A_{k,\xi}}{A_{l,l\tau}A_{k-l,\xi-l\tau }}\,\frac{\langle k\rangle^{\frac{\gamma}2}\langle k-l\rangle^{\frac{\gamma}2}}{\langle l\rangle^{\frac{\gamma}2}\,\lvert l\rvert^{1-\alpha}}\,\left(\lvert l\rvert^{-\alpha}\,A_{l,l\tau}\lvert\hat{\rho}_l(\tau)\rvert\right)\cdot\\
&\qquad\qquad\qquad\qquad\qquad\cdot\left\lvert A_{k-l,\xi-l\tau}\,(D^{j-1}_\xi\hat{h})(\tau,k-l,\xi-l\tau)\right\rvert\,\left\lvert A_{k,\xi}\,(D^{j}\hat{h})(\tau,{k,\xi})\right\rvert\,d\xi\\
\lesssim&\,a(T(\tau))F[\rho](\tau,z)\sum_{k,l\in\Z^3\backslash\{0\}}\left(\langle k-l\rangle^{-\sigma}+\langle l\rangle^{-\sigma}\right)\lvert l\rvert^{\alpha-1-\frac{\gamma}2}\cdot\\
&\,\qquad\cdot\left(\int_{\R^3} \langle k-l\rangle^\gamma A_{k-l,\xi-l\tau}^2\,\left\lvert(D^{j-1}_\xi\hat{h})(\tau,k-l,\xi-l\tau)\right\rvert^2\,d\xi+\int_{\R^3} \langle k\rangle^\gamma A_{k,\xi}^2\,\left\lvert(D_\xi^{j}\hat{h})(\tau,k,\xi)\right\rvert^2\,d\xi\right)\\
\lesssim&\,a(T(\tau))F[\rho](\tau,\cdot)\,\del_zG[h](\tau,\cdot)\,.
\end{align*}
\end{proof}
This implies the following in terms of sliding regularity.
\begin{corollary}[Differential bound in sliding regularity]\label{cor:F-to-G}
Under the conditions of Lemma \ref{lem:delt-G}, as long as
\begin{equation}\label{eq:F-cond}
a(T(\tau))\tilde{F}[\rho](\tau)\lesssim \sqrt{\epsilon} \,\langle\tau\rangle^{-2-\delta}
\end{equation}
holds for a sufficiently small $\epsilon>0$ and $\tau\in[0,\tau_{Boot})$, one has
\[\left(\del_\tau\sqrt{\tilde{G}[h]}\right)(\tau)\lesssim a(T(\tau))\tilde{F}[\rho](\tau)\,.\]
\end{corollary}
\begin{proof} Lemma \ref{lem:delt-G} implies
\[\del_\tau\tilde{G}[h](\tau)-\left[z^\prime(\tau)+C_0\langle\tau\rangle a(T(\tau))\tilde{F}[\rho](\tau)\right]\del_z\tilde{G}[h](\tau,z(\tau))\lesssim a(T(\tau))\tilde{F}[\rho](\tau,z)\sqrt{\tilde{G}[h](\tau,z)}\,,\]
and the statement follows after inserting \eqref{eq:F-cond} and observing
\[z^\prime(\tau)=-\delta\lambda_1\langle\tau\rangle^{-1-\delta}\,.\]
\end{proof}

Thus, we need to establish \eqref{eq:F-cond} as a consequence of \eqref{eq:boot} to control $G[h]$ and close the bootstrap argument. To this end, we will need to use the following scaled Sobolev embedding.

\begin{lemma}[Preliminary bound on the density]\label{lem:prelim} For any $z\in[0,1],\,\tau\geq 0$, and regularity parameters $\sigma>1$, $\gamma\in(0,1],$ the following estimate holds for any $(k,\xi)\in\Z^3\times\R^3$.
\begin{equation}\label{eq:sob-emb}
A(z)_{k,\xi}\lvert\hat{h}(\tau,k,\xi)\rvert\lesssim \sqrt{G[h](\tau,z)}\,.
\end{equation}
In particular, for $\alpha\geq 1$, one has
\begin{equation}\label{eq:sob-emb-FG}
F[\rho](\tau,z)\lesssim\sqrt{G[h](\tau,z)}\,.
\end{equation}
\end{lemma}
\begin{proof}
Since the statements are trivial otherwise, assume $G[h](\tau,z)$ is finite, and thus that $\hat{h}(\tau,k,\xi)$ decays to zero as $\lvert \xi\rvert\to\infty$ for any $k\in\Z^3,\,\tau\geq 0$. Thus, the fundamental theorem of calculus implies
\[\left\lvert\hat{h}(\tau,k,\xi)\right\rvert^2\leq \int_{\lvert\xi_1^\prime\rvert\geq\lvert\xi_1\rvert}\left(\del_{\xi_1}\left\lvert\hat{h}\left(\tau,k,\left(\cdot,\xi_2,\xi_3\right)\right)\right\rvert^2\right)(\xi_1^\prime)\,d\xi_1^\prime\,.\]
Iterating this over all indices, after again observing that first and second derivatives of $\hat{h}$ in $\xi$ decay to 0 as $\lvert \xi\rvert\to\infty$, one gets
\[\left\lvert\hat{h}(\tau,{k,\xi})\right\rvert^2\leq \int_{\lvert\xi_i^\prime\rvert\geq\lvert\xi_i\rvert}\left(\del_{\xi_1}\del_{\xi_2}\del_{\xi_3}\left\lvert\hat{h}\left(\tau,k,\cdot\right)\right\rvert^2\right)(\xi^\prime)\,d\xi^\prime\]
Applying the product rule and applying the Young inequality, this implies
\begin{align*}
A(z)_{k,\xi}^2\,\left\lvert\hat{h}(\tau,k,\xi)\right\rvert^2\lesssim&\, \sum_{j=0}^{3}\int_{\lvert\xi^\prime\rvert\geq\lvert\xi\rvert}A(z)_{k,\xi^\prime}^2\left\lvert (D_{\xi}^j\hat{h})(\tau,k,\xi^\prime)\right\rvert^2\,d\xi^\prime\leq G[h](\tau,z)
\end{align*}
and hence \eqref{eq:sob-emb}, while \eqref{eq:sob-emb-FG} follows by inserting $\xi=k\tau$ into \eqref{eq:sob-emb} using $\lvert k\rvert^{-\alpha}\leq 1$ for $k\neq 0$.
\end{proof}

\subsection{Bounds on the source}\label{subsec:source-nl-bd}
Our goal, now, is to obtain \eqref{eq:F-cond} by leveraging Corollary \ref{cor:lin-damping-gevrey-concrete}, for which we require a pointwise bound on $\tilde{F}_k[S]$ in terms of $\tilde{F}_k[\rho]$. To this end, we will use the initial data assumption to control linear contributions in the source, and the bootstrap assumption to control its nonlinear terms. \\

First, we have the following lemma:

%Our goal is to leverage Corollary \ref{cor:l2-pw-volterra-intro} to derive a pointwise bound on $\tilde{F}_k[\rho]$, and in particular \eqref{eq:F-cond}. This requires, first, that we bound $\tilde{F}_k[\rho]$ in $L^2([0,\tau_{Boot}])$. To this end, we require $L^2$- and pointwise bounds on $\tilde{F}_k[S]$ in terms of $\tilde{F}_k[\rho]$. For both of these, we will use the initial data assumption to control initial data contributions in the source, and the bootstrap assumption to control its nonlinear terms. \\

\begin{lemma}\label{lem:ckl-pw}[Coefficient estimate for pointwise bound]
Assume the following parameter choices for some $\beta,\beta^\prime>0$.
\[\gamma\in\left(1-\frac{2}{3+\beta+\beta^\prime},1\right]\,,\,\quad\alpha\in\left(\frac13,1\right),\quad \sigma>\max\{4,2+\beta+\beta^\prime\}\,.\]
Furthermore, define the coefficient function $C_{k,l}:\{(t,s)\in\R^2\ \vert\ t\geq s\geq 0\}\longrightarrow\R$ for $k,l\in\Z^3\backslash\{0\}$ as
\begin{align*}\numberthis\label{eq:def-Ckl}
C_{k,l}(t,s)=&\,(t-s)\frac{A(z(t))_{k,kt}}{A(z(s))_{l,ls}\,A(z(s))_{k-l,kt-ls}}\left(\frac{\lvert k\rvert}{\lvert l\rvert}\right)^{1-\alpha}\,.
\end{align*}
Then, the following estimate holds for $\eta>0$.
\begin{equation}\label{eq:I1}
\sup_{k\in\Z^3\backslash\{0\}}\sum_{l\in\Z^3\backslash\{0\}}\int_0^\infty \langle t\rangle^{\beta^\prime}\,C_{k,l}(t,s)\langle s\rangle^{\beta-\eta}\,ds\lesssim_{\alpha,\beta,\beta^\prime,\gamma,\eta,\sigma}\langle t\rangle^{-\eta}\,.
\end{equation}
\end{lemma}

Assuming for now that this lemma holds, one obtains the following:

\begin{lemma}[Pointwise Gevrey estimate for the source term]\label{lem:pw-est-source} Let the coefficient choices be as in Lemma \ref{lem:ckl-pw}, let the scale factor satisfy Assumption \ref{ass:scale} with respect to $\beta>0$ and take $\beta^\prime=\frac32\beta$. Then, for $\eta>0$, the following pointwise estimate holds for the source term \eqref{eq:source-nl} for any $\tau\in[0,\tau_{Boot}]$.
\begin{equation}\label{eq:pw-est-source}
\langle\tau\rangle^{\eta+\beta^\prime} \tilde{F}[S](\tau)\lesssim \epsilon\,e^{-\frac{\lambda_1}3\langle\tau\rangle^\gamma}+\sqrt{C_0}\sqrt{\epsilon}\sup_{0\leq \tilde{\tau}\leq\tau}\left(\langle \tilde{\tau}\rangle^{\eta}\tilde{F}[\rho](\tilde{\tau})\right)\,.
\end{equation}
\end{lemma}
%Until applying Lemma \ref{lem:ckl-pw}, this proof can be performed as in \cite[Lemma 4.4]{GNR21} as well as applied to higher dimensional settings, so we write it in all generality \todo{for now}.
\begin{proof} First recall that, since $g$ is charge-neutral initiallly, we have $\hat{\rho}_0\equiv 0$, thus we drop any such terms below without further comment.\\
To establish \eqref{eq:pw-est-source}, it is necessary to estimate, uniformly in $k\in\Z^3\backslash\{0\}$,
\begin{align*}
&\,\langle \tau\rangle^{\eta+\beta^\prime}\,A(z(\tau))_{k,k\tau} \lvert k\rvert^{-\alpha}\,\lvert\hat{S}_k(\tau)\rvert\\
\leq&\,\langle \tau\rangle^{\eta+\beta^\prime}\,A(z(\tau))_{k,k\tau}\, \lvert k\rvert^{-\alpha}\,\lvert \hat{h}(0,k,k\tau)\rvert\\
&\,+\langle \tau\rangle^{\eta+\beta^\prime}\sum_{l\in \Z^3\backslash\{0\}}\int_{0}^{\tau}a(T(\tilde{\tau}))(\tau-\tilde{\tau})\,A(z(\tau))_{k,k\tau}\lvert k\rvert^{1-\alpha}\lvert l\rvert^{-1}\,\lvert\hat{\rho}_l(\tilde{\tau})\rvert\,\lvert\hat{h}(\tilde{\tau},k-l,(k-l)\tilde{\tau})\rvert\,d\tilde{\tau}
\end{align*}
Regarding the initial data term first, one has
\[\langle \tau\rangle^{\eta+\beta^\prime}\,A(z(\tau))_{k,k\tau}\,\left\lvert \hat{h}(0,k,k\tau)\right\rvert\,\lvert k\rvert^{-\alpha}\leq \left(A(\lambda_0)_{k,k\tau}\,\hat{h}(\tau,k,k\tau)\right)\cdot \langle \tau\rangle^{\eta+\beta^\prime}\,\frac{A(z(\tau))_{k,k\tau}}{A(\lambda_0)_{k,k\tau}}\,,\]
where the weight can be bounded, using $\lvert k\rvert\geq 1$ and $\lambda_0-z(\tau)\geq \lambda_0-2\lambda_1>\frac{\lambda_0}2$, by
\[\lesssim \langle \tau\rangle^{\eta+\beta^\prime}\,\exp\left(-\frac{\lambda_0}2\langle \tau\rangle^{\gamma}\right)\lesssim_{\lambda_0,\eta,\beta^\prime} \exp\left(-\frac{\lambda_0}3\langle\tau\rangle^{\gamma}\right)\,.\]
Using \eqref{eq:sob-emb} and the initial data assumption \eqref{eq:ass-init}, one obtains
\[A(z(\tau))_{k,k\tau}\,\lvert \hat{h}(0,k,k\tau)\rvert\,\lvert k\rvert^{-\alpha}\lesssim \exp\left(-\frac{\lambda_1}3\langle\tau\rangle^{\gamma}\right)\sqrt{{G}[h](0,\lambda_0)}\lesssim \epsilon\exp\left(-\frac{\lambda_1}3\langle\tau\rangle^{\gamma}\right)\,.\]
Regarding the nonlinear term, one estimates using Assumption \ref{ass:scale} to estimate the scale factor along with \eqref{eq:sob-emb}, one gets
\begin{align*}
&\,\langle \tau\rangle^{\eta+\beta^\prime}\,\sum_{l\in \Z^3\backslash\{0\}}\int_{0}^{\tau}a(T(\tilde{\tau}))(\tau-\tilde{\tau})\,A_{k,k\tau}(z(\tau))\lvert k\rvert^{1-\alpha}\lvert l\rvert^{-1}\,\lvert\hat{\rho}_l(\tilde{\tau})\rvert\,\lvert\hat{h}(\tilde{\tau},k-l,\tilde{\tau}(k-l))\rvert\,d\tilde{\tau}\\
\lesssim&\,\langle \tau\rangle^{\eta}\,\sum_{l\in\Z^3\backslash\{0\}}\int_{0}^{\tau}\left(\langle \tau\rangle^{\beta^\prime}\,(\tau-\tilde{\tau})\frac{A(z(\tau))_{k,k\tau}}{A(z(\tilde{\tau}))_{l,l\tilde{\tau}}\,A(z(\tilde{\tau}))_{k-l,k\tau-l\tilde{\tau}}}\lvert k\rvert^{1-\alpha}\lvert l\rvert^{-1+\alpha}\,\langle\tilde{\tau}\rangle^{\beta-\eta}\right)\,\left(\langle\tilde{\tau}\rangle^\eta\,\tilde{F}[\rho](\tilde{\tau})\right)\,\cdot\\
&\,\qquad\qquad\qquad\cdot\left\lvert A(z(\tilde{\tau}))_{k-l,k\tau-l\tilde{\tau}}\,\hat{h}(\tau,k-l,k\tau-l\tilde{\tau})\right\rvert\,d\tilde{\tau}\\
\lesssim&\,\sup_{0\leq \tilde{\tau}\leq \tau}\left(\langle \tilde{\tau}\rangle^{\eta}\tilde{F}[\rho](\tilde{\tau})\right)\sqrt{\tilde{G}[h](\tau)}\cdot \langle \tau\rangle^\eta\sum_{l\in\Z^3\backslash\{0\}}\left(\int_{0}^{\tau} \langle \tau\rangle^{\beta^\prime} C_{k,l}(\tau,\tilde{\tau})\langle\tilde{\tau}\rangle^{\beta-\eta}\,d\tilde{\tau}\right)
\end{align*}
Applying Lemma \ref{lem:ckl-pw}, Assumption \ref{ass:scale} and the bootstrap assumption \eqref{eq:boot} then leads to the claim.
\end{proof}

\begin{proof}[Proof of Lemma \ref{lem:ckl-pw}]
First, by Lemma \ref{lem:algebraic}, we have
\[\langle k-l,kt-ls\rangle^\gamma+\langle l,ls\rangle^\gamma\geq \left(\langle k-l,kt-ls\rangle^\gamma+\langle l,ls\rangle^\gamma\right)^\gamma\geq\langle k,kt\rangle^\gamma\]
and consequently obtain the bound
\begin{equation}\label{eq:Ckl-first-est}
C_{k,l}(t,s)\leq (t-s)\left(\frac{\lvert k\rvert}{\lvert l\rvert}\right)^{(1-\alpha)}\,\left(\frac{\langle k,kt\rangle}{\langle l,ls\rangle\,\langle k-l,kt-ls\rangle}\right)^{\sigma}\exp\left[\left(z(t)-z(s)\right)\,\langle k,kt\rangle^\gamma\right]\,.
\end{equation}

Furthermore, the above implies that $\langle k-l,kt-ls\rangle\geq\frac12\langle k,kt\rangle$ or $\langle l,ls\rangle\geq\frac12\langle k,kt\rangle$ must hold, and thus we distinguish two cases.\\

\noindent\textbf{Case 1}: Assume $\langle k-l,kt-ls\rangle\geq\frac12\langle k,kt\rangle$. \\

Then, \eqref{eq:Ckl-first-est} simplifies to
\[C_{k,l}(t,s)\leq 2^\sigma\,\left(\frac{\lvert k\rvert}{\lvert l\rvert}\right)^{1-\alpha}(t-s)\,\exp\left((z(t)-z(s))\langle k,kt\rangle^\gamma\right)\langle l,ls\rangle^{-\sigma}\,.\]

\underline{Case 1a)}: Additionally assume $s\in[0,\frac{t}2]$. This implies

 \[z(t)-z(s)\leq \lambda_1\left(\langle t\rangle^{-\delta}-\langle 0.5\,t\rangle^{-\delta})\right)\leq (2^\delta-1)\lambda_1\,\langle t\rangle^{-\delta}=:\theta_\delta\langle t\rangle^{-\delta}\]
along with
\begin{align*}
\lvert k\rvert^{1-\alpha}(t-s)\,\exp\left(-\theta_\delta\,\langle t\rangle^{-\delta}\langle k,kt\rangle^\gamma\right)&\,\lesssim_{\delta}\exp\left(-\frac{\theta_\delta}2\,\langle t\rangle^{-\delta}\right)\cdot\,\left(\langle k\rangle\langle t\rangle\exp\left(-\theta_\delta^\prime(\langle k\rangle\langle t\rangle)^{\gamma-\delta}\right)\right)\\
&\,\lesssim_{\delta,\gamma}\exp\left(-\frac{\theta_\delta}2\langle t\rangle^{-\delta}\right)
\end{align*}
for $\theta_\delta^\prime=2^{-\frac{\gamma-\delta}2-1}\theta_\delta>0$, where we used in the first step that, since $\lvert k\rvert\geq 1$,
\[\langle k\rangle\langle t\rangle=\sqrt{1+\lvert k\rvert^2+\lvert t\rvert^2+\lvert kt\rvert^2}\leq\sqrt{1+\lvert k\rvert^2+2\lvert kt\rvert^2}\leq \sqrt{2}\langle k,kt\rangle\,,\]
as well as \eqref{eq:tool} in the second. Thus, also using $\langle l\rangle\langle s\rangle\lesssim \langle l,ls\rangle$, we get
\[C_{k,l}(t,s)\lesssim_{\sigma,\gamma,\delta} \langle l\rangle^{-1-\sigma+\alpha}\,\exp\left(-0.5\theta_\delta\langle t\rangle^{\gamma-\delta}\right)\,.\]
Bearing in mind that $\sigma+1-\alpha>5-\alpha>3$ ensures summability in $l$ throughout and using the estimate \eqref{eq:tool}, the contribution of this case to \eqref{eq:I1} can be bounded by
\begin{align*}
&\,\lesssim\left(\sum_{l\in\Z^3\backslash\{0\}} \langle l\rangle^{-1-\sigma+\alpha}\right)\exp\left(-\frac{\theta_\delta}2\langle t\rangle^{\gamma-\delta}\right)\,\langle t\rangle^{\beta^\prime}\int_0^{\frac{t}2}\langle s\rangle^{-\sigma+\beta-\eta}\,ds\\
&\,\lesssim \left(\sum_{l\in\Z^3\backslash\{0\}} \langle l\rangle^{-1-\sigma+\alpha}\right) \langle t\rangle^{\beta+\beta^\prime+1}\exp\left(-\frac{\theta_\delta}2\langle t\rangle^{\gamma-\delta}\right)\\
&\,\lesssim_{\alpha,\beta,\beta^\prime,\gamma,\delta,\sigma}\exp\left(-\frac{\theta_\delta}4\langle t\rangle^{\gamma-\delta}\right)\,.
\end{align*}
Clearly, this is stronger than the claimed bound after potentially updating the implicit constant.\\

\underline{Case 1b)}: For $s\in[\frac{t}2,t]$, we use
\[z(s)-z(t)\geq \lambda_1\frac{t-s}{\langle t\rangle^{1+\delta}}\]
along with $\langle l,ls\rangle\simeq \langle l,lt\rangle$ to get
\begin{align*}
C_{k,l}(t,s)\lesssim&\,(t-s)\lvert l\vert^{\alpha-1}\,\langle l,lt\rangle^{-\sigma}\,\left(\lvert k\rvert^{1-\alpha}\,\exp\left(-\lambda_1\,(t-s){\langle t\rangle^{\gamma-(1+\delta)}}\cdot\lvert k\rvert^\gamma\right)\right).
\end{align*}
Applying \eqref{eq:tool} to the bracketed expression on the right to bound it uniformly in $\lvert k\rvert$, this becomes
\begin{align}\label{eq:coeff-1b}
C_{k,l}(t,s)\lesssim&\,\langle t\rangle^\frac{(1-\gamma+\delta)(1-\alpha)}{\gamma}\lvert l\rvert^{\alpha-1}\langle l,lt\rangle^{-\sigma}\lvert t-s\rvert^{1-\frac{1-\alpha}{\gamma}}\,.
\end{align}
Subsequently, the contribution to \eqref{eq:I1} is bounded by
\begin{align*}
\lesssim&\,\left(\sum_{l\in\Z^3\backslash\{0\}}\lvert l\rvert^{\alpha-1}\langle l\rangle^{-\sigma}\right)\langle t\rangle^{-\sigma+\frac{(1-\gamma+\delta)(1-\alpha)}{\gamma}+\beta^\prime}\int_0^{\frac{t}2}\lvert s^\prime\rvert^{1-\frac{1-\alpha}{\gamma}}\langle t-s^\prime\rangle^{\beta-\eta}\,ds^\prime\,.
\end{align*}
The integral is finite for
\[1-\frac{1-\alpha}{\gamma}>-1\,\Longleftrightarrow \alpha>1-2\gamma\,,\]
which is ensured by $\alpha>\frac13$ and $\gamma>\frac13$, and the series is summable as in 1a). Altogether, the contribution of this regime is bounded by
\[\lesssim \langle t\rangle^{-\sigma+\frac{(1-\gamma+\delta)(1-\alpha)}{\gamma}+\beta^\prime}\,\langle t\rangle^{2-\frac{1-\alpha}\gamma+\beta-\eta}=\langle t\rangle^{-\eta}\,\langle t\rangle^{-\sigma+\alpha+1+\beta+\beta^\prime+\delta\,\frac{1-\alpha}{\gamma}}\,.\]
This is consistent with the upper bound in \eqref{eq:I1} if $\delta$ is chosen to be appropriately small in a manner dependent on the regularity and expansion parameters since 
\[\sigma>2+\beta+\beta^\prime>\beta+\beta^\prime+\alpha+1\,.\]

\noindent\textbf{Case 2}: Otherwise, one has $\langle l,ls\rangle\geq \frac12\langle k,kt\rangle$, and thus by \eqref{eq:Ckl-first-est}
\begin{equation}\label{eq:Ckl-first-est-c2}
C_{k,l}(t,s)\lesssim_{\sigma} C_{k,l}(t,s)\leq (t-s)\left(\frac{\lvert k\rvert}{\lvert l\rvert}\right)^{(1-\alpha)}\,\langle k-l,kt-ls\rangle^{-\sigma}\,\exp\left[\left(z(t)-z(s)\right)\,\langle k,kt\rangle^\gamma\right]\,.
\end{equation}

\indent\underline{Case 2a)}: If, in addition, $\langle kt-ls\rangle\geq \frac{t}4$ holds, then one can estimate 
\[\langle k-l,kt-ls\rangle^{-\sigma}\lesssim \langle k-l\rangle^{-\sigma}\,\langle k-l,kt-ls\rangle^{-\sigma}\lesssim \langle k-l\rangle^{-\sigma}\,\langle t\rangle^{-\sigma}\,.\]
\begin{enumerate} 
\item If $s\in[0,\frac{t}2]$, we argue from \eqref{eq:Ckl-first-est-c2} as in Case 1a) that one has
\[C_{k,l}(t,s)\lesssim \langle k-l\rangle^{-\sigma}\,\exp\left(-\theta_\delta\langle t\rangle^{\gamma-\delta}\right)\,,\]
and we similarly obtain that the contribution to \eqref{eq:I1} is bounded by
\[\lesssim \left(\sum_{l\in\Z^3}\langle k-l\rangle^{-\sigma}\right) \exp\left(-\frac{\theta_\delta}2\langle t\rangle^{\gamma-\delta}\right)\,\int_0^{\frac{t}2}\langle s\rangle^{-\sigma+\beta-\eta}\,ds\lesssim_{\alpha,\beta,\beta^\prime,\gamma,\delta,\sigma} \exp\left(-\frac{\theta_\delta}{4}\langle t\rangle^{\gamma-\delta}\right)\]
using $\sigma>3$ and \eqref{eq:tool}.\\

\item If $s\in[\frac{t}2,t]$, we argue as in Case 1b) to get
\[C_{k,l}(t,s)\lesssim \langle t\rangle^{\frac{(1-\gamma+\delta)(1-\alpha)}{\gamma}-\sigma}\lvert l\rvert^{\alpha-1}\langle k-l\rangle^{-\sigma}\lvert t-s\rvert^{1-\frac{1-\alpha}\gamma}\,.\]
Thus, since $\alpha<1$ and $\alpha>1-2\gamma$ hold, the respective contribution to \eqref{eq:I1} can be bounded by
\[\lesssim \left(\sum_{{l}\in\Z^3}\left\langle k-l\right\rangle^{-\sigma}\right)\,\langle t\rangle^{-\eta}\,\langle t\rangle^{-\sigma+\alpha+1+\beta+\beta^\prime+\delta\,\frac{1-\alpha}{\gamma}}\]
with $\sigma>3$ ensuring summability of the first term and the final factor being uniformly bounded for sufficiently small $\delta>0$ as in 1b).\\
\end{enumerate}

\underline{Case 2b)}: Otherwise, $\langle kt-ls\rangle\leq\frac{t}4$ holds and we distinguish between the following regimes.\\

\begin{enumerate}
\item For $l=k\neq 0$, we get the following contribution to \eqref{eq:I1}.
\begin{align*}
\int_{\frac{t}2}^t \langle t\rangle^{\beta^\prime}C_{k,k}(t,s)\langle s\rangle^{\beta-\eta}\,ds\lesssim&\,\int_{\frac{t}2}^t \langle t\rangle^{\beta^\prime} \langle t-s\rangle^{-\sigma+1}\langle s\rangle^{\beta-\eta}\,ds\\
\lesssim&\,\langle t\rangle^{-\sigma+2+\beta+\beta^\prime-\eta}\,\int_0^\frac{t}2 \langle s\rangle^{-\sigma+1}\,ds\lesssim \langle t\rangle^{-\sigma+2+\beta+\beta^\prime-\eta}\,.
\end{align*}
This is consistent with the claimed bound since we assumed $\sigma>2+\beta+\beta^\prime$.

\item For $l\neq k$ and $s\in[0,\frac{t}2]$, one argues as in Case 1a) to get
\[C_{k,l}(t,s)\leq \exp\left(-\frac{\theta_\delta}2\langle t\rangle^{\gamma-\delta}\right)\langle k-l\rangle^{-\sigma}\]
and thus the contribution is bounded as claimed after exchanging summation index and otherwise arguing as in Case 1a).\\
\item For $l\neq k$ and $s\in [\frac{t}2,t]$, we have
\begin{equation}\label{eq:2b3-aux}
\lvert k(t-s)\rvert\geq s\lvert k-l\vert -\lvert kt-ls\rvert\geq \frac{t}2\lvert k-l\rvert-\frac{t}4\lvert k-l\rvert=\frac{t}4\lvert k-l\rvert
\end{equation}
and thus for some $\tilde{\theta}_\delta>0$ and again using $\langle k\rangle\,\langle t\rangle\lesssim\langle k,kt\rangle$,
\begin{align*}
\exp((z(t)-z(s))\langle k,kt\rangle^\gamma)\leq&\,\exp\left(-\theta_0\langle t\rangle^{-1-\delta}\,\lvert k(t-s)\rvert \cdot \lvert k\rvert^{-1}\langle k,kt\rangle^{\gamma}\right)\\
\leq&\,\exp\left(-\theta_{\delta}\langle t\rangle^{\gamma-\delta}\frac{\lvert l-k\rvert}{\lvert k\rvert^{1-\gamma}}\right)\,.
\end{align*}

We make one final case distinction.\\

\begin{enumerate}
\item If $\lvert l-k\rvert\geq\frac12\lvert k\rvert$, then, returning to \eqref{eq:Ckl-first-est-c2}, there exists some appropriate $\tilde{\theta}_\delta>0$ such that $C_{k,l}$ is bounded by
\[C_{k,l}(t,s)\lesssim\,\left(\lvert k\rvert\langle t\rangle \exp(-2\tilde{\theta}_\delta\,\langle t\rangle^{\gamma-\delta}\lvert k\rvert^\gamma)\right)\,\cdot\langle k-l\rangle^{-\sigma}\lesssim_{\gamma,\delta} \exp\left(-\tilde{\theta}_\delta \langle t\rangle^{\gamma-\delta}\right)\,\langle k-l\rangle^{-\sigma}\]
using \eqref{eq:tool}, and this bound can again be treated as in Case 2a), (1).\\

\item This leaves the regime where $\lvert l-k\rvert\leq\frac12\lvert k\rvert$, and consequently 
\begin{align*}
\lvert l\vert\geq&\,\lvert k\rvert-\lvert l-k\rvert\,\geq \frac12\lvert k\rvert\geq 3\lvert l-k\rvert\,,\\
\lvert l\rvert\leq&\,\lvert k\rvert+\lvert l-k\rvert\leq \frac32\lvert k\rvert\,.
\end{align*}
In particular, $\left(\frac{\lvert k\rvert}{\lvert l\rvert}\right)^{1-\alpha}\lesssim_\alpha 1$ follows. Thus, applying \eqref{eq:2b3-aux} to \eqref{eq:Ckl-first-est-c2}, 
\begin{align*}
&\,\langle t\rangle^{\beta^\prime}\,C_{k,l}(t,s)\,\langle s\rangle^{\beta-\eta}\\
\lesssim&\,\langle t\rangle^{\beta+\beta^\prime}\,(t-s)\langle k-l,kt-ls\rangle^{-\sigma}\,\exp\left(-\tilde{\theta}_{\delta}\langle t\rangle^{\gamma-\delta}\frac{\lvert k-l\rvert}{\lvert k\rvert^{1-\gamma}}\right)\,\langle t\rangle^{-\eta}\\
\lesssim&\,\frac{\lvert k-l\rvert}{\lvert k\rvert} \langle k-l,kt-ls\rangle^{-\sigma}\,\langle t\rangle^{-\eta}\,\left[\langle t\rangle^{1+\beta+\beta^\prime}\exp\left(-\tilde{\theta}_\delta \lvert l-k\rvert\lvert k\rvert^{\gamma-1}\langle t\rangle^{\gamma-\delta}\right)\right]\,.
\end{align*}
Again using \eqref{eq:tool} to bound the final factor uniformly in $\langle t\rangle$, we obtain
\begin{align*}
&\,\quad\langle t\rangle^{\beta^\prime}\,C_{k,l}(t,s)\,\langle s\rangle^{-\eta+\beta}\\
&\,\lesssim_{\beta,\beta^\prime,\gamma,\delta} \frac{\lvert k-l\rvert}{\lvert k\rvert}\,\langle k-l,kt-ls\rangle^{-\sigma}\,\lvert k-l\rvert^{-\frac{1+\beta+\beta^\prime}{\gamma-\delta}}\,\lvert k\rvert^{-\frac{(\gamma-1)\left(1+\beta+\beta^\prime\right)}{\gamma-\delta}}\,\langle t\rangle^{-\eta}\\
&\,\lesssim_{\delta,\gamma}\left(\langle k-l,kt-ls\rangle^{-\sigma}\lvert l\rvert\right)\left(\lvert k-l\rvert^{1-\frac{1+\beta+\beta^\prime}{\gamma-\delta}}\lvert l\rvert^{-1}\lvert k\rvert^{\frac{(1-\gamma)\left(1+\beta+\beta^\prime\right)}{\gamma-\delta}-1}\right)\langle t\rangle^{-\eta}\numberthis\label{eq:interim-worst-case}\,.
\end{align*}
Regarding the first factor, we now bound its integral over $s\in[\frac{t}2,t]$: Writing
\begin{align*}
\Delta_{k,l}(t):=\sqrt{1+\lvert k-l\rvert^2+\frac{\lvert k\times l\rvert^2}{\lvert l\rvert^2}t^2}\geq \langle k-l\rangle\geq \sqrt{2}\,,
\end{align*}
one obtains
\begin{align*}
&\int_{\frac{t}2}^t \langle k-l,kt-ls\rangle^{-\sigma}\,\lvert l\rvert\,ds\\
=&\,\int_{\frac{t}2}^t \left[1+\lvert k-l\rvert^2+\lvert kt\rvert^2-2 (k\cdot l)ts+\lvert ls\rvert^2\right]^{-\frac{\sigma}2}\,\lvert l\rvert\,ds\\
=&\,\int_{\frac{t}2}^t\,\left[\left(\lvert l\rvert s-\frac{k\cdot l}{\lvert k\rvert}t\right)^2+1+\lvert k-l\rvert^2+t^2\left(\lvert k\rvert^2-\frac{(k\cdot l)^2}{\lvert l\rvert^2}\right)\right]^{-\frac{\sigma}2}\,\lvert l\rvert\,ds\\
=&\,\int_{\left(\frac{\lvert l\rvert}2-\frac{k\cdot l}{\lvert l\rvert}\right)t}^{(\lvert l\rvert-\frac{k\cdot l}{\lvert l\rvert})t}\left[\tilde{s}^2+{\Delta_{k,l}(t)}^2\right]^{-\frac{\sigma}2}\,d\tilde{s}\\
=&\,\left(\Delta_{k,l}(t)\right)^{-\sigma+1}\int_{\left(\frac{\lvert l\rvert}2-\frac{k\cdot l}{\lvert l\rvert}\right)\frac{t}{\Delta_{k,l}(t)}}^{(\lvert l\rvert-\frac{k\cdot l}{\lvert l\rvert})\frac{t}{\Delta_{k,l}(t)}} \left[1+(s^\prime)^2\right]^{-\frac{\sigma}2}\,ds^\prime\\
\leq&\,\left(\Delta_{k,l}(t)\right)^{-\sigma+1}\,\int_{-\infty}^\infty \langle s^\prime\rangle^{-\sigma}\,ds^\prime\,.
\end{align*}
As $\sigma>4$, the integral on the right hand side is finite, while $\lvert k-l\rvert\geq1$ implies
\begin{align*}
\Delta_{k,l}(t)^{-\sigma+1}\lesssim_{\sigma} \langle k-l\rangle^{-\sigma+1}\,.
\end{align*}
For the second bracketed expression in \eqref{eq:interim-worst-case}, recall that $\lvert l\rvert\simeq \lvert k\rvert$ holds in this regime, and consequently
\[\lvert k\rvert^{\frac{(1-\gamma)(1+\beta+\beta^\prime)}{\gamma-\delta}-1}\,\lvert l\rvert^{-1}\lesssim \lvert k\rvert^{-2+\frac{(1-\gamma)(1+\beta+\beta^\prime)}{\gamma-\delta}}\,.\]
Altogether, we can bound the contribution of this case to \eqref{eq:I1} by
\[\lesssim \langle t\rangle^{-\eta}\,\left(\sup_{k\in \Z^3\backslash\{0\}}\lvert k\rvert^{-2+\frac{(1-\gamma)(1+\beta+\beta^\prime)}{\gamma-\delta}}\right)\,\left(\sum_{l\in\Z^3\backslash\{0\}}\langle k-l\rangle^{-\sigma+2-\frac{1+\beta+\beta^\prime}{\gamma-\delta}}\right)\,.\]
The final factor is finite for appropriately small $\delta>0$, relative to terms depending only on parameters of regularity and expansion, since we have
\[-\sigma+2-\frac{1+\beta+\beta^\prime}{\gamma}\leq -\sigma+1<3\,.\]
Meanwhile, the supremum is finite since
\[\gamma>1-\frac{2}{3+\beta+\beta^\prime}=\frac{1+\beta+\beta^\prime}{3+\beta+\beta^\prime}\]
implies
\[{(1-\gamma)(1+\beta+\beta^\prime)}<2\gamma\,,\]
and thus, again for sufficiently small $\delta>0$,
\[\frac{(1-\gamma)(1+\beta+\beta^\prime)}{\gamma-\delta}<2\,.\]
\end{enumerate}
\end{enumerate}
\end{proof}
Altogether, we note that the contributions from the final case and from Case 1b) (with its analogues) are the main obstacles forcing the Gevrey-regularity restrictions above, and subsequently in Theorem \ref{thm:main}.

\subsection{Closing the bootstrap argument}\label{subsec:close-bs}

Finally, we show how to use these tools to obtain Corollary \ref{cor:F-to-G-intro}, which closes the bootstrap argument, as well as how this implies Theorem \ref{thm:main}.

\begin{corollary}[Improved pointwise Gevrey bound]\label{cor:gevrey-imp} Let 
\[\gamma\in\left(1-\frac{2}{3+\beta+\beta^\prime},1\right],\quad \alpha\in\left(\frac13,1\right],\quad \sigma>\max\{4,2+\beta+\beta^\prime\}\,.\]
Furthermore, let $\beta^\prime=\frac32\beta$ and $\eta=\frac52+\beta$, and choose $\delta\in(0,\frac12)$ to be sufficiently small depending only on $\alpha,\beta,\gamma$ and $\sigma$. 
%\[\gamma\in\left(1-\frac{3}{6+2\beta},1\right),\quad \alpha\in\left(\frac13,1\right],\quad \sigma>\max\{4,2+\beta\}\]
Then, for any $\tau\in[0,\tau_{Boot}]$ and sufficiently small $\epsilon>0$, one has
\begin{equation}\label{eq:gevrey-imp}
\tilde{F}[\rho](\tau)\lesssim \epsilon\langle\tau\rangle^{-\eta}\,.
\end{equation}
\end{corollary}
\begin{proof}
By Corollary \ref{cor:lin-damping-gevrey-concrete}, there exists a constant $\theta_1^\prime>0$ such that
\begin{align*}
\langle\tau\rangle^\eta\,\tilde{F}[\rho](\tau)\lesssim&\, \langle\tau\rangle^\eta\,\tilde{F}[S](\tau)+\int_0^\tau \exp\left(-\theta_1^\prime\,\lvert k\rvert\,(\tau-\tilde{\tau})\right)\,\langle\tilde{\tau}\rangle^{\eta+\beta^\prime}\,\tilde{F}[S](\tilde{\tau})\,d\tilde{\tau}\,.
\end{align*}
%Under these parameter choices, we can first insert Corollary \ref{cor:L2-decay-F} into Corollary \ref{cor:l2-pw-volterra-intro} and obtain
%\begin{equation*}
%\langle \tau\rangle^{\eta}\tilde{F}[\rho](\tau)\leq\langle\tau\rangle^{\eta}\tilde{F}[S](\tau)+C\epsilon\,.
%\end{equation*}
Estimating the source terms on the right hand side with \eqref{eq:pw-est-source}, this becomes
\begin{align*}
\langle\tau\rangle^\eta\,\tilde{F}[\rho](\tau)\lesssim&\,\epsilon\,e^{-\frac{\lambda_1}3\langle \tau\rangle^\gamma}+\int_0^\tau \epsilon\,e^{-\frac{\lambda_1}3\langle \tilde{\tau}\rangle^\gamma-\theta_1^\prime\lvert k\rvert(\tau-\tilde{\tau})}\,d\tilde{\tau}\\
&\,+\sqrt{C_0}\,\sqrt{\epsilon}\,\left(\sup_{0\leq s\leq\tau}\langle s\rangle^\eta\,\tilde{F}[\rho](s)\right)+\int_0^\tau e^{-\theta_1^\prime\lvert k\rvert(\tau-\tilde{\tau})}\,\sup_{0\leq s\leq\tilde{\tau}}\left(\langle s\rangle^\eta\,\tilde{F}[\rho](s)\right)\,d\tilde{\tau}\\
\lesssim&\,\left(\epsilon+\sqrt{C_0}\sqrt{\epsilon}\,\sup_{0\leq s\leq\tau}\left(\langle s\rangle^\eta\,\tilde{F}[\rho](s)\right)\right)\left(1+\int_0^\tau e^{-\theta_1^\prime\lvert k\rvert s}\,ds\right)\\
\lesssim&\,\epsilon+\sqrt{C_0}\sqrt{\epsilon}\,\sup_{0\leq s\leq\tau}\left(\langle s\rangle^\eta\,\tilde{F}[\rho](s)\right)\,.
\end{align*}
%\begin{align*}
%\langle\tau\rangle^{\eta}\tilde{F}[\rho](\tau)\lesssim&\,\epsilon+\epsilon\langle\tau\rangle^{\eta}\exp\left(-\frac{\lambda_0}3\langle\tau\rangle^\gamma\right)+\sqrt{C_0}\sqrt{\epsilon}\left(\max_{s\in[0,\tau_{Boot}]}\langle \tau\rangle^{\eta}\tilde{F}[\rho](s)\right)\,.
%\end{align*}
In other words, there exists some $C>0$ such that
\begin{align*}
(1-C\sqrt{C_0}\sqrt{\epsilon})\left(\max_{s\in[0,\tau_{Boot}]}\langle s\rangle^{\eta}\tilde{F}[\rho](s)\right)\leq C\epsilon
\end{align*}
holds. This implies the statement for sufficiently small $\epsilon>0$.
\end{proof}
%\begin{corollary}[Bootstrap improvement] Under the conditions of Corollary \ref{cor:gevrey-imp}, for any $\tau\in[0,\tau_{Boot}]$, one has
%\begin{equation}\label{eq:imp-G-decay}
%\del_\tau\sqrt{\tilde{G}[h](\tau)}\lesssim{\epsilon}\langle\tau\rangle^{-2-\delta}\,.
%\end{equation}
%In particular, if $\epsilon>0$ is chosen small enough, this implies
%\begin{equation}\label{eq:bs-imp-G}
%\sqrt{\tilde{G}[h](\tau)}\leq \sqrt{(1+C_0)\epsilon}
%\end{equation}
%\end{corollary}
\begin{proof}[Proof of Corollary \ref{cor:F-to-G-intro}]
The estimate \eqref{eq:delt-G-imp-intro} follows by applying Corollary \ref{cor:gevrey-imp} for $\alpha\in(\frac13,1]$ to Corollary \ref{cor:F-to-G}, where the choice of $\eta$ ensures that \eqref{eq:F-cond} is satisfied. Integrating \eqref{eq:delt-G-imp-intro} and inserting the initial data condition \eqref{eq:ass-init} yields that, for some $C>0$, one has
\[\sqrt{\tilde{G}[h](\tau)}\leq\sqrt{{G}[h](0,\lambda_0)}+C\epsilon\int_0^\infty\langle s\rangle^{-2-\delta}\,ds\leq\sqrt{\epsilon}\left[\sqrt{\epsilon}+C\sqrt{\epsilon}\int_0^\infty\langle s\rangle^{-2}\,ds\right]\,,\]
where the bracketed expression is smaller than $\sqrt{1+C_0}$ for $\epsilon>0$ small enough compared to $\delta$ and $C_0$.
\end{proof}
\begin{proof}[Proof of Theorem \ref{thm:main}]
Note that \eqref{eq:G-imp-intro} strictly improves upon the bootstrap assumption \eqref{eq:boot}. Invoking the continuation criterion in Lemma \ref{lem:lwp} for sufficiently small $\epsilon>0$, the bootstrap assumption hence must hold on $[0,\tau^\prime]$ for some sufficiently close $\tau^\prime>\tau_{Boot}$ by continuity. Thus, \eqref{eq:G-imp-intro} holds globally, and so does \eqref{eq:gevrey-imp}.\\

Regarding \eqref{eq:main-rho}, where the Fourier multiplier $A(z(\tau))_{k,k\tau}$ and sliding generator function $\tilde{F}[\rho]$ are to be understood with parameter choices for $\gamma,\,\sigma$ and $\alpha$ as in Corollary \ref{cor:gevrey-imp}, one has for any $\lambda^\prime\in(0,\lambda_1)$ that
\begin{align*}
\|\rho(\tau,\cdot)\|^2_{\mathcal{G}^{\lambda^\prime,\gamma}(\T^3)}=&\,\sum_{k\in\Z^3\backslash\{0\}}\exp(2\lambda^\prime\langle k\rangle^\gamma)\,\lvert \hat{\rho}_k(\tau)\rvert^2\\
\lesssim&\,\sum_{k\in\Z^3\backslash\{0\}} \left(A(\lambda^\prime)_{k,k\tau}^2\,\lvert \hat{\rho}_k(\tau)\rvert^2\right)\,\langle k\rangle^{-2\sigma}\\
\lesssim&\,\exp(-2(\lambda_1-\lambda^\prime)\langle\tau\rangle^\gamma)\sum_{k\in\Z^3\backslash\{0\}} \Big[\lvert k\rvert^{-2\alpha}A(z(\tau))_{k,k\tau}^2\,\lvert \hat{\rho}_k(\tau)\rvert^2\Big]\,\langle k\rangle^{-2\sigma+2\alpha}\\
\lesssim&\,\exp(-2(\lambda_1-\lambda^\prime)\langle\tau\rangle^\gamma)\,\tilde{F}[\rho](\tau)^2
%\lesssim&\,\epsilon^2\,\langle\tau\rangle^{-4}\exp(-2(\lambda_1-\lambda^\prime)\langle\tau\rangle^\gamma)\,,
\end{align*}
with the final step since we have $2\sigma-2\alpha>8-2=6>3$. The claim now follows by applying the improved bound \eqref{eq:gevrey-imp}.\\

\noindent Regarding \eqref{eq:main-g}, for $(x,v)\in\T^3\times\R^3$, we formally wish to define
\begin{equation*}
h_\infty(x,v)=\lim_{\tau\to\infty}\left[h(0,x,v)+\int_0^\tau (\del_\tau h)(\tilde{\tau},x,v)\,d\tilde{\tau}\,\right]\,.
\end{equation*}
The bracketed expression is clearly well-defined for any $\tau>0$, and to show that $h_\infty$ is well-defined and an element of $\mathcal{G}^{\lambda_1,\gamma}(\T^3\times\R^3)$, it is sufficient to show
\begin{equation}\label{eq:main-g-step}
\sum_{k\in\Z^3\backslash\{0\}}\int_{\R^3}\exp\left(2\lambda_1\langle k,\xi\rangle^\gamma\right)\,\lvert (\del_\tau\hat{h})(\tau,k,\xi)\rvert^2\,d\xi\lesssim \epsilon^2\,\langle \tau\rangle^{-2}\,.
\end{equation}
To this end, we consider \eqref{eq:landau-fourier} and start by dealing with contributions to the left hand side of \eqref{eq:main-g-step} from the term in the second line of \eqref{eq:landau-fourier}. These can be estimated as follows, with parameters, Fourier multipliers and generator functions to be understood as above.
\begin{align*}
\lesssim&\,a(T(\tau))^2\sum_{k\in\Z^3\backslash\{0\}}\int_{\R^3}\left[\sum_{l\in\Z^3}\exp(2\lambda_1\langle k,\xi\rangle^\gamma)\,\left(\frac{\lvert\xi-l\tau\rvert}{\lvert l\rvert}\right)^2\lvert \hat{\rho}_l(\tau)\rvert^2\,\left\lvert \hat{h}(\tau,k-l,\xi-l\tau)\right\rvert^2\right]\,d\xi\\
\lesssim&\,a(T(\tau))^2\sum_{k,l\in\Z^3}\left[A(\lambda_1)_{l,l\tau}^2\langle l,l\tau\rangle^{-2\sigma}\lvert\hat{\rho}_l(\tau)\rvert^2\right]\,\cdot\\
&\,\quad\cdot\int_{\R^3}\left[A(\lambda_1)_{k-l,\xi-l\tau}^2\lvert\hat{h}(\tau,k-l,\xi-l\tau)\rvert^2\frac{\lvert \xi-l\tau\rvert^2+\langle\tau\rangle^2\lvert k-l\rvert^2}{\langle k-l,\xi-l\tau\rangle^{2\sigma}}\right]\,\,d\xi\\
\lesssim&\,a(T(\tau))^2\langle\tau\rangle^{-2\sigma}{F}[\rho](\tau,\lambda_1)^2\,\sum_{k,l\in\Z^3\backslash\{0\}}\langle l\rangle^{-2(\sigma-\alpha)}\int_{\R^3}  \langle\tau\rangle^2\,A(\lambda_1)_{k-l,\tilde{\xi}}^2\,\left\lvert \hat{h}\left(\tau,k-l,\tilde{\xi}\right)\right\rvert^2\,d\tilde{\xi}\\
\lesssim&\,\langle\tau\rangle^{-2(\sigma+1)}\,\left(a(T(\tau))\tilde{F}[\rho)(\tau)\right)^2\left(\sum_{{l}\in\Z^3}\langle {l}\rangle^{-2(\sigma-\alpha)}\right)\,\left(\sum_{\tilde{k}\in\Z^3} \int_{\R^3}A(\lambda_1)_{\tilde{k},\tilde{\xi}}^2\,\left\lvert \hat{h}\left(\tau,\tilde{k},\tilde{\xi}\right)\right\rvert^2\,d\xi\right)\\
\lesssim&\,\langle\tau\rangle^{-2(\sigma+1)}\,\left(a(T(\tau))\,\tilde{F}[\rho](\tau)\right)^2\,\tilde{G}[h](\tau)\,.\\
%\lesssim&\,\epsilon^3\langle\tau\rangle^{-2\sigma-2}
\end{align*}
In particular, by \eqref{eq:G-imp-intro} and \eqref{eq:gevrey-imp}, recalling $\eta>2+\beta$ and Assumption \ref{ass:scale}, the contribution is bounded by
\[\lesssim\epsilon^3\langle\tau\rangle^{-2\sigma-2}\,.\]
By an analogous argument, the other contribution from \eqref{eq:landau-fourier} can be bounded by
\begin{align*}
\lesssim&\,\langle\tau\rangle^{-2\sigma}\left(a(T(\tau))\tilde{F}[\rho](\tau)\right)^2\|\mu\|_{\mathcal{G}^{\lambda_1,\gamma}(\R^3)}^2\lesssim\epsilon^2\langle\tau\rangle^{-2\sigma-2}\,.
\end{align*}
In combination, this implies \eqref{eq:main-g-step} and thus $h_\infty\in\mathcal{G}^{\lambda_1,\gamma}(\T^3\times\R^3)$. Furthermore, by estimating the density modes as in the proof of \eqref{eq:main-rho} and otherwise proceeding as above, one gets
\[\|(\del_\tau\hat{h})(\tau,\cdot,\cdot)\|_{\mathcal{G}^{\lambda^\prime,\gamma}(\T^3\times\R^3)}\lesssim\epsilon\,\langle \tau\rangle^{-\sigma-1}\,e^{-(\lambda_1-\lambda^\prime)\langle\tau\rangle^{\gamma}}\]
Finally, this implies
\begin{align*}
\|h(\tau,\cdot,\cdot)-h_\infty\|_{\mathcal{G}^{\lambda^\prime,\gamma}(\T^3\times\R^3)}\lesssim&\,\int_\tau^\infty \|(\del_\tau h)(\tilde{\tau},\cdot,\cdot)\|_{\mathcal{G}^{\lambda^\prime,\gamma}(\T^3\times\R^3)}\,d\tilde{\tau}\\
\lesssim&\,\epsilon\,e^{-(\lambda_1-\lambda^\prime)\langle\tau\rangle^\gamma}\int_\tau^\infty \langle\tilde{\tau}\rangle^{-\sigma-1}\,d\tilde{\tau}\\
\lesssim&\,\epsilon\,e^{-(\lambda_1-\lambda^\prime)\langle\tau\rangle^\gamma}
\end{align*}
which is \eqref{eq:main-g}.
\end{proof}

\appendix
\renewcommand*{\theequation}{\Alph{section}.\arabic{equation}}

\section{Deriving cosmological Vlasov-Poisson equations from a classical system}\label{sec:newt-cosmo}

In this section, we outline how Newtonian cosmological models for Vlasov matter arise in general dimensions. This follows the standard approach as in, for example, \cite[Chapter II, Sections 7ff.]{Pee80} and \cite[Section 2]{RR94}, additionally coupled to an external field driving expansion. We start from the classical Vlasov-Poisson system on $I\times \R^d\times\R^d$ for $d\geq 1$ coupled to an external field $\tilde{\Phi}_b:I\times \R^d\longrightarrow\R$, using a standard Euclidean frame $(t,(r^i)_{i=1,\dots,d},(p^i)_{i=1,\dots,d})$.

\begin{subequations}
\begin{gather}
\label{eq:vlasov-canonical}\del_t\tilde{f}(t,r,p)+p^i(\del_{r^i}\tilde{f})(t,r,p)-\delta^{ij}\left(\del_{r^i}\tilde{\Phi}_b(t,r)+\del_{r^i}\tilde{U}(t,r)\right)(\del_{p^j}\tilde{f})(t,r,p)=0\,,\\
\label{eq:poisson-canonical}\delta^{ij}\del_{r^i}\del_{r^j}\tilde{U}(t,r)=4\pi\,\epsilon_F\,{\bm \varrho}(t,r),\quad {\bm \varrho}(t,r)=\int_{\R^d}\tilde{f}(t,r,p)\,dp\,.
\end{gather}
\end{subequations}

We are now interested in studying solutions in coordinates that are comoving with respect to an expanding universe, with expansion given by the scale factor $a:I\longrightarrow(0,\infty)$. Thus, we want to write this system with respect to spatial coordinates $(x^i)_{i=1,\dots,d}$ such that
\[r(t,x)=a(t)x\,.\]
More generally, for a particle moving along the trajectory $\gamma$ in $(t,x)$-coordinates, it moves along the trajectory $\tilde{\gamma}$ in $(t,r)$ coordinates given by
\[\tilde{\gamma}(t)=a(t)\gamma(t)\]
which has proper velocity
\[\del_t\tilde{\gamma}(t)=\del_t{a}(t)\,\gamma(t)+a(t)\,\del_t{\gamma}(t)\,.\]
Note that the velocity measured by an observer moving along $\gamma$ consists only of the second term, referred to as the \textit{proper peculiar velocity}, while the first term measures expansion due to the background. In accordance with this, we choose velocity coordinates
\[v(t,\gamma)=p(t,\tilde{\gamma})-{\del_ta(t)}\,\gamma(t)\,,\]
where $p(t,\tilde{\gamma})$ denotes the \textit{proper velocity} with respect to the associated curve $\tilde{\gamma}$ in standard coordinates. 
In summary, the natural comoving coordinates $(t,(x^i)_{i=1,\dots,d},(v^i)_{i=1,\dots,d})$ on $I\times \R^d\times\R^d$ are related to the Euclidean frame by
\[r^i(t,x)=a(t)x^i,\quad p^i(t,x,v)=v^i+\dot{a}(t)x^i\,.\]
We rewrite the solution variables $(\tilde{f},\tilde{U},{\bm \varrho})$ in terms of these coordinates:
\begin{align*}
f(t,x,v)=&\,\tilde{f}(t,r(t,x),p(t,x,v))\,,\\
{U}(t,x)=&\,\tilde{U}(t,r(t,x))\,,\\
\Phi_b(t,x)=&\,\tilde{\Phi}_b(t,r(t,x))\,,\\
{\underline{{\bm \rho}}}(t,x)=&\,{\bm \varrho}(t,r(t,x))\,.
\end{align*}
From here on out, we will suppress the dependency of $(t,r,p)$ on $(t,x,v)$ in notation, and denote the derivatives with respect to the slots in $t,\,r^i$ and $p^j$ by $\del_0,\del_i$ and $\del_{d+j}$ respectively, with terms of the form $p^j\del_{d+j}$ to be interpreted in terms of the standard Einstein summation convention with respect to the repeated index $j=1,\dots,d$.\\

Note that one has
\begin{gather*}
\del_{x^i}f=\frac{\del r^j}{\del x^i}\cdot\del_{j}\tilde{f}+\frac{\del p^j}{\del x^i}\cdot\del_{d+j}\tilde{f}=a\,\del_i\tilde{f}+\dot{a}\,\del_{d+i}\tilde{f}\\
\del_{v^i}f=\frac{\del p^j}{\del v^i}\cdot\del_{d+j}\tilde{f}=\del_{d+i}\tilde{f}
\end{gather*}
and consequently also
\begin{gather*}
\del_i\tilde{f}=a^{-1}\,\del_{x^i}f-\frac{\dot{a}}a\,\del_{v^i}f\,.
\end{gather*}
One computes:
\begin{subequations}
\begin{align*}
\numberthis\label{eq:landau-resc-with-bg}(\del_tf)(t,x,v)=&\,(\del_0\tilde{f})(t,r,p)+\dot{a}(t)\,x^i\,(\del_{i}\tilde{f})(t,r,p)+(\ddot{a}(t)\,x^i)\,(\del_{d+i}\tilde{f})(t,r,p)\\
=&\,-(v^i+\dot{a}(t)\,x^i)\,(\del_i\tilde{f})(t,r,p)+\left((\del_{i}\tilde{U})(t,r)+(\del_i\tilde{\Phi}_b)(t,r)\right)\,(\del_{d+i}\tilde{f})(t,r,p)\\
&\,+\dot{a}(t)\,x^i\,(\del_{i}\tilde{f})(t,r,p)+\ddot{a}(t)\,x^i\,(\del_{d+i}\tilde{f})(t,r,p)\\
=&\,-a(t)^{-1}\,v^i(\del_{x^i}f)(t,x,v)+\frac{\dot{a}(t)}{a(t)}\,v^i\,(\del_{v^i}f)(t,x,v)+a(t)^{-1}\,\delta^{ij}\,(\del_{r^i} U)(t,x)\,(\del_{v^j}f)(t,x,v)\\
&\,+\left(a(t)^{-1}\,(\del_{x^i}{\Phi_b})(t,x)+\ddot{a}(t)\,x^i\right)(\del_{v^i}{f})(t,x,v)\,.%\numberthis\label{eq:vlasov-eq-deriv}
\end{align*}
Similarly,
\begin{align}
\label{eq:poisson-with-bg}\delta^{ij}\,(\del_{x^i}\del_{x^j}U)(t,x)=&\,a(t)^{2}\,\delta^{ij}\,(\del_{i}\del_{j}\tilde{U})(t,r)=4\pi\,a(t)^2\,\epsilon_F\,{\bm \varrho}(t,r)=4\pi\,a(t)^2\,\epsilon_F\,\underline{{\bm \rho}}(t,x)\,,\\
\label{eq:density-with-bg}\underline{{\bm \rho}}(t,x)=&\,\int_{\R^d}\tilde{f}(t,r,p)\,dp=\int_{\R^d}\tilde{f}(t,r,p(t,x,v))\,dv=\int_{\R^d}f(t,x,v)\,dv\,.
\end{align}
\end{subequations}

What we have ignored so far is that we are interested in deviations from a matter distribution $f_0$ that is homogeneous and isotropic \textbf{from the perspective of a comoving observer} and how it influences the expansion rate via its mass density ${\bm \rho}_0(t)$, which must be spatially constant and which we want to preserve in the spatial average.\\
%Note that the Lagrangian of the particle motion is given by
%\[\mathcal{L}=\frac12 p(t,\gamma(t))^2-(\Phi_b+{U})(t,\gamma(t))\,,\]
%which is related via canonical transformation to the Lagrangian 
%\[\mathcal{L}^\prime=\frac12 a(t)^2\,\dot{\gamma}(t)^2-(\Phi_b+U-U_0)(t,\gamma(t)),\quad U_0(t,x)=\frac12\,a(t)\ddot{a}(t)\lvert x\rvert^2\,.\]
%Thus, particles are subject to the equations of motion
%\[{\bm p}(t)=a(t)^2\dot{\gamma}(t)=a(t)\,v(t,\gamma),\quad {\bm \dot{p}}(t)=-\nabla(\Phi_b+U-U_0)(t,x)\]
%In particular, a freely moving particle satisfies ${\bm p}=\text{const.}$, and its peculiar velocity decays as $a(t)^{-1}$. \\

We note that solutions to the Vlasov equation \eqref{eq:landau-resc-with-bg} must, at least locally, be of the form
\[f(t,x,v)=\mathring{f}(\gamma_{x,v}(t),{\bm v}_{x,v}(t))\]
where $(\gamma,{\bm v})_{x,v}$ describe the position $\gamma_{x,v}(t)$ and proper peculiar velocity ${\bm v}_{x,v}(t)$ of a particle at time $t$, emanating from position $x$ with initial proper peculiar velocity $v$ at $t=t_0$, and $\mathring{f}$ describes the particle distribution at $t=t_0$. The associated particle distribution measure is given by
\[dN=f(t,x,v)\bigwedge_{i=1,\dots,d}dx^i\wedge dv^i=\mathring{f}(\gamma_{x,v}(t),{\bm v}_{x,v}(t))\bigwedge_{i=1,\dots,d}dx^i\wedge dv^i\,.\]

Furthermore, note that homogeneous and locally isotropic solutions to \eqref{eq:landau-resc-with-bg} must satisfy
\begin{equation}\label{eq:vlasov-equil}
(\del_tf_0)(t,x,v)=\frac{\dot{a}(t)}{a(t)}\,v^i\,(\del_{v^i}f_0)(t,x,v)
\end{equation}
and, thus, the characteristics take the form
\[\gamma_{x,v}(t)=x,\quad {\bm v}_{x,v}(t)=a(t)\,v\,.\]
In other words, these solutions take the following form:
\[f_0(t,x,v)={\mu}(a(t)\,v)=\tilde{\mu}(\lvert a(t)\,v\rvert)=\tilde{\mu}\left(a(t)\left\lvert p(t,x,v)-\frac{\dot{a}(t)}{a(t)}\,r(t,x)\right\rvert\right)\]
In particular, its particle distribution measure is then given by
\begin{align*}
dN_0=\tilde{\mu}(a(t)v)\,\bigwedge_{i=1,\dots,d}dx^i\wedge dv^i=a(t)^{-d}{\mu}( w)\bigwedge_{i=1,\dots,d}dx^i\wedge dw^i
\end{align*}
and the particle density $\underline{{\bm \rho}}_0(t,x)$ hence satisfies
\[\underline{{\bm \rho}}_0(t,x)=\int_{\R^d} f_0(t,x,v)dv=a(t)^{-d}\left(\int_{\R^d}{\mu}(w)\,dw\right)\,.\]
From here on out, we choose the normalisation
\[\int_{\R^d}\mu(w)\,dw=1\,\]
and write $\underline{{\bm \rho}}_0(t)\equiv\underline{{\bm \rho}}_0(t,x)=a(t)^{-d}$. Solutions to the Poisson equation \eqref{eq:poisson-with-bg} on $\R^d$ with this density on the right-hand side take the form 
%\[\delta^{ij}\del_{x^i}\del_{x^j}{U}_0(t,r)=4\pi\epsilon_F\,a(t)^{2-d}\]
%take the form
\[U_0(t,x)=\frac{2\pi}{d}\epsilon_F\,a(t)^{2-d}\lvert x\vert^2 \leftrightharpoons \tilde{U}_0(t,r)=\frac{2\pi}{d}\epsilon_F\,a(t)^{-d}\lvert r\vert^2\,.\]
We now denote
\[g=f-f_0,\quad\underline{\rho}=\underline{{\bm \rho}}-\underline{{\bm \rho}}_0, \quad W=U-U_0\,\]
and, obtain the following system for the perturbed variables $(g,\underline{\rho},W)$ using \eqref{eq:vlasov-equil}.
\begin{align*}
(\del_tg)(t,x,v)=&\,-a(t)^{-1}\,v^i\,(\del_{x^i}g)(t,x,v)+\frac{\dot{a}(t)}{a(t)}\,(\del_{v^i}g)(t,x,v)+a(t)^{-1}\,\delta^{ij}\,(\del_{x^i} W)(t,x)\,(\del_{v^j}g)(t,x,v)\\
&\,+a(t)^{-1}\,\delta^{ij}\,(\del_{x^i} W)(t,x)\,(\del_{v^j}f_0)(t,v)\\
&\,+\,\left(a(t)^{-1}\delta^{ij}\del_{x^i}{\Phi_b}(t,x)+\ddot{a}(t)\,x^j+\frac{4\pi}d\,\epsilon_F\,a(t)^{1-d}\,x^j\right)\,\del_{v^j}{f}(t,x,v)\,.
\end{align*}
This yields a solution that is periodic in $x$ -- i.e., a solution on $\T^d\times\R^d$ in comoving coordinates -- if one chooses the background field to cancel the final line, namely
\[\Phi_b(t,x)=-\frac12\left(a(t)\ddot{a}(t)+\frac{4\pi}d\epsilon_F\,a(t)^{2-d}\right)\,\lvert x\rvert^2\,.\]
Conversely, if we view the background field to be fixed and of the form
\[\Phi_b(t,x)=\phi_b(t)\cdot\frac12\lvert x\rvert^2\,,\]
the scale factor must solve the \enquote{Friedman-like} equation
\begin{equation}\label{eq:friedman-esque-bg}
\ddot{a}(t)=-\frac{4\pi}d\epsilon_F\,a(t)^{1-d}-a(t)^{-1}\phi_b(t)\,.
\end{equation}
Assuming either is satisfied, we altogether obtain the following cosmological Vlasov-Poisson system near the isotropic background $f_0$.
\begin{subequations}\label{eq:vp-cosmo-general}
\begin{align*}
\numberthis(\del_tg)(t,x,v)&\,+a(t)^{-1}\,v^i\,(\del_{x^i}g)(t,x,v)-\frac{\dot{a}(t)}{a(t)}\,(\del_{v^i}g)(t,x,v)-a(t)^{-1}\,(\del_{x} W)(t,x)\cdot(\del_{v}g)(t,x,v)\\
=&\,a(t)^{-1}\,(\del_{x} W)(t,x)\cdot (\del_{v}f_0)(t,v)\,,\\
\numberthis(\Lap W)(t,x)=&\,4\pi\,\epsilon_F\,a(t)^2\,\underline{\rho}(t,x),\quad \underline{\rho}(t,x)=\int_{\R^d}g(t,x,v)\,dv\,.
\end{align*}
\end{subequations}
For $\epsilon_F=-1$ and $d=3$, this is precisely \eqref{eq:perturbation-evol} with the scale factor as in \eqref{eq:friedman-esque-bg} for some real-valued function $\phi_b$ that we can choose freely to determine the scale factor $a$. In terms of the original distribution $f:(0,\infty)\times\R^d\times\R^d\longrightarrow\R$, this reads
\begin{subequations}\label{eq:evol-full}
\begin{align}
\del_tf(t,x,v)+&\,a(t)^{-1}\,v^i(\del_{x^i}f)(t,x,v)-\frac{\dot{a}(t)}{a(t)}(\del_{v^i}f)(t,x,v)\\
&\,-a(t)^{-1}\left(\del_xU(t,x)-\frac{4\pi}3a(t)^{2-d}x\right)\cdot\del_vf(t,x,v)=0\,,\\
\Lap U(t,x)=&\,4\pi\,\epsilon_F\,a(t)^2\,\underline{{\bm \rho}}(t,x),\quad \underline{{\bm \rho}}(t,x)=\int_{\R^d}f(t,x,v)\,dv\,.
\end{align}
\end{subequations}

\begin{remark}[Scale factor asymptotics and \eqref{eq:friedman-esque-bg}]\label{rem:scale-factor-appendix}
For $\phi_b\equiv 0$ and $\epsilon_F=1$, the power law solutions of \eqref{eq:friedman-esque-bg} are given by $a(t)=a_0\,t^\frac2{d}$ for an appropriately chosen $a_0>0$, while for $\epsilon_F=-1$, any positive solution without a background field satisfies $a(t_0)^{1-d}\lesssim \ddot{a}(t)$ and thus $t^2\lesssim a(t)$ for all $t\geq t_0>0$.\\

Conversely, for an external field of the form
\begin{equation}\label{eq:bg-field}
\phi_{b}(t)=-\epsilon_F\,\frac{4\pi}{d}a(t)^{2-d}-q(q-1)a(t)^{-\frac{2}q}
\end{equation}
we obtain power law solutions
\[a(t)=t^q\,.\]
Of course, one can simply postulate the presence of such a background field from a purely abstract perspective for the purpose of our scaling analysis. For $d=3$ and $\epsilon_F=-1$, a more phenomenological view on this is that we are studying a universe with an equal density of positively and negatively charged particles while only perturbing one of these matter types to obtain \eqref{eq:vp-cosmo-general}, along with a homogeneous perfect fluid with linear equation of state $p=K\rho$ for $K\in(-\frac13,1]$. The background fluid would then lead to the second term in \eqref{eq:bg-field} and cause power law expansion with $q=\frac2{3(K+1)}\in[\frac13,1)$. Even slower expansion rates arise analogously for $K>1$, which are referred to as \textit{superluminal fluids} since the speed of sound in these mediums is, formally, larger than the speed of light.\\

We note that, in the non-expanding setting, a common way that the classical analogue of \eqref{eq:vp-cosmo-general} is derived involves the \enquote{Jeans swindle}, where no external field is introduced and field terms arising from the equilibrium are simply set to $0$ by force, since otherwise no equilibrium solutions would be possible. This formal argument is not common in the cosmological settings, since the Friedman-like equation \eqref{eq:friedman-esque-bg} above actually do allow for equilibria as studied in, say, \cite{RR94}. However, for the same reasons that the Jeans swindle is somewhat justifiable on small scales as discussed in \cite[p.~416f.]{BT08}, one could attempt to justify a similar swindle here when studying a smaller patch of the current universe, after which \eqref{eq:bg-field} would only contain the second term that determines the rate of expansion. 
\end{remark}

\begin{remark}[Comparison to relativistic Vlasov equations]\label{rem:nonrel-limits}
To illustrate the relationship between the cosmological Vlasov-Poisson system and its relativistic counterparts, we first look at the free transport equation. Consider the relativistic (massive) Vlasov equation on an FLRW background
\[\left(\M_{FLRW}=(0,\infty)\times\T^3,\ \g_{FLRW}=-c^2\,dt^2+\delta_{ij}\,dx^i\,dx^j\right)\]
for the particle distribution function $f^\flat$ on the co-mass shell, i.e., in canonical coordinates $(q^\mu,p_\mu)$ with $p_\mu\,p^\mu=-m^2\,c^2$ for $m>0$. Denoting $p^0=\sqrt{m^2+c^{-2}\lvert p\rvert^2}$, the Vlasov equation reads
%\[\g^{\mu\nu}\,p_\mu\del_{q^\mu}f+\g^{\mu\nu}\,p_\mu p_\nu\,\Gamma[\g]^{\mu}_{\nu i}\del_{p_i}f=0\]
%On an FLRW background with spatially flat geometry, this becomes
\begin{subequations}
\begin{equation}
\del_tf^\flat+a^{-2}\frac{p_i}{p^0}\,\del_{q^i}f^\flat=0\,.
\end{equation}
This is often written in terms of the distribution function $f^\sharp$ defined on the mass shell rather than the co-mass shell, i.e., in terms of
\[(t,q^i, p^j=(\g_{FLRW})^{\mu j}p_\mu=a^{-2}\delta^{ij}p_j)\,.\]
Then, this becomes
\begin{equation}\label{eq:rel-vlasov-mass}
\del_tf^\sharp+\frac{p^i}{p^0}\del_{q^i}f^\sharp-2\frac{\dot{a}}a p^i\del_{p^i}f^\sharp=0\,,
\end{equation}
\end{subequations}
which is the equation studied in \cite{TVR25} for $m=1$, $c=1$ and $a(t)=t^q$. Taking the nonrelativistic limit of the both by formally taking the limit $c\to\infty$ leads to the toy equations
\begin{subequations}
\begin{gather}
\label{eq:limit-comass}\del_tf^\flat+\frac{p_i}{m\,a^2}\del_{q^i}f^\flat=0\,\\
\label{eq:limit-mass}\del_tf^\sharp+\frac{p^i}{m}\del_{q^i}f-2\frac{\dot{a}}a p^i\del_{p^i}f^\sharp=0\,.
\end{gather}
\end{subequations}
At first glance, one might be confused that the free transport component for the (proper peculiar) velocity in \eqref{eq:perturbation-evol} enters with a factor of $-\sfrac{\dot{a}}a$, which matches neither of these equations. However, this is simply because we describe the system in terms of coordinates $(x,v)$ describing the position and velocity felt by a comoving observer. In particular, \cite[(9.3)]{Pee80} precisely recovers \eqref{eq:limit-comass} for vanishing external gravitational field: Therein, the author writes the system in terms of the particle position $x$ in comoving coordinates and the respective canonical momentum ${\bm p}$, which is related to the proper peculiar velocity ${\bm v}$ by ${\bm p}=m\,a\,{\bm v}$. Conversely, \eqref{eq:limit-mass} serves as the toy model \cite[(1.8)]{TVR25} for the relativistic Vlasov equation in coordinates on the mass shell, and for $m=1$, the coordinate change $v^i=a\,p^i=a^{-1}\,\delta^{ij}p_j$ recovers the free transport equation \eqref{eq:free-transport}. In short, while all three equations describe the particle distribution functions in terms of different variables adapted to their respective set-ups, they encode the same physical matter after taking a formal nonrelativistic limit.\\

Returning to the full cosmological Vlasov-Poisson system, \eqref{eq:perturbation-evol} also arises in a formal nonrelativistic limit of equations describing a charged relativistic gas on an expanding background: Recall that the one-species relativistic Vlasov-Maxwell system on a fixed FLRW background with spatial geometry $\R^3$ takes the following form in terms of spatial coordinates $(q^i)_{i=1,2,3}$ and canonical momentum coordinates $(p_\mu)_{\mu=0,1,2,3}$, where the cross product, divergence and curl operators are to be understood with respect to the Euclidean metric and its associated Levi-Civita connection.
\begin{gather*}
\del_tf+a^{-2}\,\frac{p}{p^0}\cdot\del_{q}f+e\left(E+\left(\frac{p}{p^0}\times B\right)\right)\cdot\del_{p}f=0\,,\\
\del_tE=-\frac{\dot{a}}a\,E+a^{-1}\,\curl B-\j,\quad \del_tB=-a^{-1}\,\curl E\,,\\
\div E=4\pi a^2\,(\rho-\overline{\rho}),\quad \div B=0\,,\\
\rho(t,q)=\int_{\R^3}f(t,q,p)\,dp\,,\quad \overline{\rho}(t)=\int_{\R^3}\rho(t,q)\,dq\,,\quad (\j(t,q))_i=\int_{\R^3} \frac{p_i}{p^0}\,f(t,q,p)\,dp\,.
\end{gather*}
We now write the system in terms of proper peculiar velocity $v^i=a^{-1}\,\delta^{ij}\,p_j$ from the perspective of a comoving observer with $v^0=\sqrt{1+c^{-2}\,a^2\,\lvert v\rvert^2}$, and consider the special case where the magnetic field $B$ vanishes. Then, this system becomes
\begin{align*}
&\,\del_tf+a^{-1}\frac{v^i}{v^0}\,\del_{q^i}f-\frac{\dot{a}}a\,v^i\del_{v^i}\,f+e\,a^{-1}\,E\cdot\del_{v}f=0\\
&\,\div E=4\pi a^2\,(\rho-\overline{\rho}),\,\quad \rho=\int_{\R^3}f(t,q,p)\,dp
\end{align*}
Note that the relativistic Vlasov-Poisson studied by Young in \cite{You15,You16} can be viewed as a special case of the relativistic Vlasov-Maxwell system in an entirely analogous fashion, up to using different velocity coordinates.\\
Finally, taking the formal relativistic limit then recovers \eqref{eq:evol-full} for $m=e=1$ and for
\[E_i=-\left(\del_{x^i}U-\frac{4\pi}3\,a^{2-d}\,\delta_{ij}x^j\right)\,.\]
In turn, this recovers \eqref{eq:perturbation-evol} for the near-equilibrium perturbation $g=f-f_0$. To summarize, one can formally view \eqref{eq:perturbation-evol} as the nonrelativistic limit of a family of special solutions to the relativistic Vlasov-Maxwell system, and it is in precisely this sense that our Newtonian description of a charged collisionless gas on an expanding background relates to a relativistic one on an FLRW background.
\end{remark}

\section{Gravitational particle interaction and Landau damping in high dimensions}\label{sec:4+}

In this section, before briefly discussing the case of gravitational particle interaction in $d=3$, we show that Landau damping occurs in \eqref{eq:vp-cosmo-general} for $d\geq 4$, even for gravitational particle interaction.

\subsection{Landau damping for $d\geq 4$}\label{subsec:4+}

In high dimensions, we no longer need to restrict ourselves to the Poisson background to show Landau damping, and the analysis is significantly closer to the classical one. To this end, we assume the background satisfies the following adapted Penrose condition.

\begin{ass}[Penrose condition for $d\geq 4$]\label{ass:penrose4+}
Let $f_0:(0,\infty)\times\T^d\times\R^d\longrightarrow(0,\infty)$ be of the form
\[f_0(t,x,v)=\mu(a(t)\,v)\]
where ${\mu}$ is analytic and nonnegative, normalised to $\int_{\R^d}\mu(w)dw=1$, as well as
\begin{equation}\label{eq:distr-ass-appendix}
\left\lvert \left(D_\xi^{\leq d}\hat{\mu}\right)(\xi)\right\rvert\lesssim e^{-\theta_0\lvert \xi\rvert}\,.
\end{equation}

Then, $f_0$ is said to satisfy the \textbf{adapted Penrose condition} for some $\kappa>0$ if one has, for $d=4$,
\begin{subequations}
\begin{equation}\label{eq:penrose4}
\inf_{k\in\Z^d\backslash\{0\}}\inf_{\Re(\lambda)\geq 0}\left\lvert1-4\pi\,\epsilon_F\int_0^\infty e^{-\lambda\,s}\,s\hat{\mu}(sk)\,ds\right\rvert\geq \kappa>0
\end{equation}
and, respectively, for $d\geq 5$ with initial data chosen at $t_0>0$,
\begin{equation}\label{eq:penrose5+}
\inf_{k\in\Z^d\backslash\{0\}}\inf_{\Re(\lambda)\geq 0}\left\lvert1-4\pi\,a(t_0)^{-(d-4)}\int_0^\infty e^{-\lambda\,s}\,s\lvert\hat{\mu}(sk)\rvert\,ds\right\rvert\geq \kappa>0
\end{equation}
\end{subequations}
\end{ass}

Note that \eqref{eq:penrose4} is simply the standard Penrose condition as it appears, for example, in \cite[p.~44,(L)]{MV10}. Thus, it remains satisfied for any radially symmetric background in the electrostatic case, including the Poisson equilibrium. The condition \eqref{eq:penrose5+} could potentially be improved for repulsive particle interaction with $d\geq 5$ given that, in absence of a scale factor, any radial background should satisfy the standard Penrose condition. Nevertheless, it is worth noting that \textbf{any} analytic background satisfies this condition for $d\geq 5$ if $t_0>0$ is chosen to be sufficiently large, irrespective of the type of particle interaction.

\begin{prop}[Landau damping on an expanding background for $d\geq 4$]\label{prop:trunc}
Let $f_0$ be a positive analytic equilibrium satisfying the adapted Penrose condition for some parameters $C_0,\theta_0,\kappa>0$, see Assumption \ref{ass:penrose4+}. Further, let $\gamma\in(\frac13,1],\,\sigma>d+1$ and $\lambda_0>0$. If $\gamma=1$, additionally assume $\lambda_0<\theta_0$. Additionally, let $a\in C^\infty(0,\infty)$ be positive and increasing. \\

We consider the cosmological Vlasov-Poisson system \eqref{eq:vp-cosmo-general}  for $(g,\underline{\rho},W)$ near the background $f_0$, with charge-neutral initial data $g_0$ imposed at $t_0>0$ that satisfies
\begin{equation}\label{eq:ass-init-d4}
\|g_0\|_{\mathcal{G}^{\lambda_0,\sigma,\gamma}(\T^d\times\R^d)}\leq\epsilon^2
\end{equation}
for some sufficiently small $\epsilon>0$. We further define
\[h(\tau(t),x,v)=g(\tau,x+\tau(t)v,a(t)^{-1}v)\quad\text{for}\quad\tau(t)=\int_{t_0}^ta(s)^{-2}\,ds\,.\]

Then, there exist $\lambda_1\in(0,4\lambda_0)$ and $h_\infty\in\mathcal{G}^{\lambda_1,\sigma,\gamma}(\T^d\times\R^d)$ such that, for any $\lambda^\prime\in(0,\lambda_1]$ and $t>t_0$, one has
\begin{align*}
\|h(\tau(t),\cdot,\cdot)-h_\infty\|_{\mathcal{G}^{\lambda^\prime,\gamma}(\T^d\times\R^d)}\lesssim&\,\epsilon e^{-(\lambda_1-\lambda^\prime)\langle\tau(t)\rangle^\gamma}\\
\|\underline{\rho}(t,\cdot)\|_{\mathcal{G}^{\lambda^\prime,\gamma}(\T^d)}\lesssim&\,\epsilon a(t)^{-d}e^{-(\lambda_1-\lambda^\prime)\langle\tau(t)\rangle^\gamma}\,.
\end{align*}
\end{prop}
If $\tau(t)\to \infty$ as $t\to\infty$, this implies improved decay, which is superpolynomial if $a$ grows slower than $t^\frac12$. Since we do not require Assumption \ref{ass:scale} on the scale factor, it also implies improved polynomial decay if $a(t)=t^\frac12$. \\
Recalling Remark \ref{rem:scale-factor-appendix} for the purely self-interacting case (i.e., $\phi_b\equiv 0$), this implies that solutions to the expanding Vlasov-Poisson system with gravitational self-interaction exhibit polynomial Landau damping for $d=4$, and exponential Landau damping for $d\geq 5$, if the adapted Penrose condition is satisfied.\\

For $d=4$, this result can be seen rather quickly after renormalising the equations as in Section \ref{subsec:transform}. First, defining
\[\rho(\tau,x)=a(T(\tau))^d\,\underline{\rho}(T(\tau),x)=\int_{\R^d} h(\tau,x-\tau v,v)\,dv\]
as before, the analogue of \eqref{eq:landau-fourier} reads
\begin{align*}
\numberthis\label{eq:landau-fourier-high-d}\del_\tau\hat{h}(\tau,k,\xi)-4\pi\, a(T(\tau))^{4-d}\,\frac{k\cdot(\xi-k\tau)}{\lvert k\rvert^2}\,\hat{\rho}_k(\tau)\,\hat{\mu}(\xi-k\tau)&\,\\
-4\pi\,a(T(\tau))^{4-d}\,\sum_{l\in\Z^d}\frac{l\cdot(\xi-k\tau)}{\lvert l\vert^2}\,\hat{\rho}_l(\tau)\,\hat{h}(\tau,k-l,\xi-l\tau)=&\,0\,.
\end{align*}

Crucially, for $d=4$, this equation is identical to the one arising in the classical Vlasov-Poisson system, and consequently, Proposition \ref{prop:trunc} now completely follows from the arguments in \cite{BMM16,GNR21}. \\

For $d\geq 5$, while the equations don't fully simplify, one still has that $(a\circ T)^{4-d}$ is uniformly bounded, so we need to proceed with a little more caution, which we sketch in the remainder of this subsection.\\
The Volterra integral equation for non-zero density modes reads% for the electrostatic setting, and the gravitational statement can be proven identically after replacing the Penrose condition \cite[(1.6)]{GNR21} with its gravitational analogue, i.e., \eqref{eq:penrose4} for $\epsilon_F=1$.\\
\begin{subequations}
\begin{align}
\label{eq:volterra-4+}\hat{\rho}_k(\tau)=&\,\hat{S}_k(\tau)+4\pi\,\epsilon_F\int_0^\tau(\tau-\tilde{\tau})\,\hat{\mu}(k(\tau-\tilde{\tau}))\,\hat{\rho}_k(\tilde{\tau})\cdot a(T(\tilde{\tau}))^{4-d}\,d\tilde{\tau}
\end{align}
with
%with, in the linearized setting,
%\[\hat{S}_k(\tau)=\hat{h}(0,k,\tau k)\]
%and, in the full setting
\begin{align}\label{eq:source-4+}
\hat{S}_k(\tau)=&\,\hat{h}(0,k,k\tau)-4\pi\,\epsilon_F\int_0^{\tau}\sum_{l\in\Z^3}\frac{(l\cdot k)(\tau-\tilde{\tau})}{\lvert l\vert^2}\,a(T(\tilde{\tau}))^{4-d}\,\hat{\rho}_l(\tilde{\tau})\,\hat{h}(T(\tilde{\tau}),k-l,k\tau-l\tilde{\tau})\,d\tilde{\tau}
\end{align}
\end{subequations}

While the equation \eqref{eq:volterra-4+} is not of convolution type and thus we cannot argue directly as in the classical setting, the boundedness of the expansion term allows us to relate the analysis back to classical settings. To this end, we import the following standard comparison theorem for Volterra inequalities.

\begin{lemma}[Comparison theorem for Volterra inequalities, see {\cite[Theorem 1.2.19]{Bru17}}]\label{lem:comparison}
Let $I$ be a bounded real interval and $y:I\longrightarrow\R$ and $K:I\times I\longrightarrow\R$ be continuous nonnegative functions. Furthermore, let $R$ be the resolvent kernel to the Volterra integral equation
\[\tilde{u}(t)=y(t)+\int_0^t K(t,s)\tilde{u}(s)\,ds\quad\text{for all}\quad t\in I.\]
Furthermore, assume $u\in C(I)$ satisfies
\[u(t)\leq y(t)+\int_0^t K(t,s) u(s)\,ds\,.\]
Then both $u$ and $R$ are nonnegative, and one has, for all $t\in I$,
\[u(t)\leq y(t)+\int_0^t R(t,s)y(s)\,ds\,.\]
\end{lemma}

With this in hand, one obtains the following.

\begin{corollary}[Pointwise damping for $d\geq 5$]\label{lem:pw-damp-trunc} Let $d\geq 5$. Then, for solutions to \eqref{eq:volterra-4+} with $k\in\Z^d\backslash\{0\}$ and under Assumption \ref{ass:penrose4+} on the background, there exists a continous resolvent kernel $R_k$ and an appropriate $\theta_1>0$ satisfying
\begin{subequations}
\begin{equation}
\lvert\hat{\rho}_k(\tau)\rvert\leq \lvert \hat{S}_k(\tau)\rvert+\int_0^{\tau}R_k(\tau,\tilde{\tau})\,\lvert\hat{S}_k(\tilde{\tau})\rvert\,d\tilde{\tau}\,.
\end{equation}
as well as
\begin{equation}\label{eq:4+-res-est}
\sup_{k\in\Z^d\backslash\{0\}}\lvert R_k(\tau,\tilde{\tau})\rvert\leq e^{-\theta_1\,\lvert k\rvert\,(\tau-\tilde{\tau})}\,.
\end{equation}
\end{subequations}
\end{corollary}
\begin{proof}
By applying Lemma \ref{lem:comparison} to
\[\lvert\hat{\rho}_k(\tau)\rvert\leq \lvert \hat{S}_k(\tau)\rvert+4\pi\int_0^{\tau}\left\lvert \hat{M}_k(\tau-\tilde{\tau})\,a(T(\tilde{\tau}))^{4-d}\right\rvert\,\lvert\hat{\rho}_k(\tilde{\tau})\rvert\,d\tilde{\tau}\,,\]
on $I=(0,\tau^\ast)$ for arbitrary $\tau^\ast>0$, the statement follows once it is shown that the resolvent kernel $R_k$ associated with
\[\tilde{u}_k(\tau)=\lvert \hat{S}_k(\tau)\rvert+4\pi\,a(t_0)^{4-d}\,\int_0^{\tau} \left\lvert\hat{M}_k(\tau-\tilde{\tau})\right\rvert\cdot \tilde{u}_k(\tilde{\tau})\,d\tilde{\tau}\]
satisfies \eqref{eq:4+-res-est}. Up to replacing $\hat{\mu}$ with $4\pi a(t_0)^{4-d}\lvert\hat{\mu}\rvert$, this is precisely the Volterra integral equation that arises in the classical case, and thus this follows as in \cite[Proposition 3.3]{GNR21} under the adapted Penrose condition \eqref{eq:penrose5+}.
\end{proof}
%For this to be useful, we need to apply it to the Gevrey norm level. Note that
%\begin{align}
%\langle k,\eta\rangle\leq&\,\langle k-k^\prime,\eta-\eta^\prime\rangle+\langle k^\prime,\eta^\prime\rangle\\
%\langle k,\eta\rangle\leq&\,2\langle k-k^\prime,\eta-\eta^\prime\rangle\,\langle k^\prime,\eta^\prime\rangle\,,
%\end{align}
%thus
%\begin{align*}
%e^{z\langle k,ks\rangle^\gamma}\leq&\, e^{z\left(\langle k(s-r)\rangle+\langle k,kr\rangle\right)^\gamma}\leq e^{z\langle k(s-r)\rangle^\gamma}e^{z\langle k,kr\rangle^\gamma}\\
%\langle k,ks\rangle^\sigma\leq&\,2^\sigma \langle k(s-r)\rangle^\sigma\langle kr\rangle^\sigma
%\end{align*}
%
%and thus, for $z\in \R$, $\gamma\leq 1$ and $\sigma\geq 1$,
%\[\frac{e^{z\langle k,ks\rangle^\gamma}\langle k,ks\rangle^\sigma}{e^{z\langle k,kr\rangle^\gamma}\langle k,kr\rangle^\sigma}\leq 2^\sigma e^{z\langle k(s-r)\rangle\gamma}\langle k(s-r)\rangle^\sigma\,.\]

From here, as in \cite[Lemma 3.6]{GNR21} and for Corollary \ref{cor:lin-damping-gevrey-concrete}, one applies the algebra properties in Lemma \ref{lem:algebraic} to deduce the following for $\tilde{F}$ defined as in Definition \ref{def:slide}.

\begin{corollary}\label{cor:damp-gevrey-4+}[Linearised damping in Gevrey norms] Under the assumptions of Lemma \ref{lem:pw-damp-trunc}, one has
\[\tilde{F}[\rho](\tau)\leq \tilde{F}[S](\tau)+\int_0^{\tau} e^{-\frac{\theta_1}4(\tau-\tilde{\tau})}\,\tilde{F}[S](\tilde{\tau})\,d\tilde{\tau}\]
for any $\tau>0$ and $\tilde{F}$ as in Definitions \ref{def:gevrey}--\ref{def:slide}.
\end{corollary}

Moving on to the nonlinear analysis, this can now largely proceed as in Section \ref{sec:nonlin} for $\beta=\beta^\prime=0$ and \cite[Section 4]{GNR21}, since the estimates therein already only depended on \eqref{eq:distr-ass}. We note that, to admit the weighted Sobolev embedding as in Lemma \ref{lem:prelim}, $G[h](\tau,z)$ must control up to $d$ derivatives in $\xi$ as in Definition \ref{def:gevrey}.\\
By analogous computations to the proof of Lemma \ref{lem:delt-G} and then applying\linebreak $a(T(\tau))^{4-d}\leq a(t_0)^{4-d}\lesssim 1$, one obtains the following differential bound in the parameter range of Proposition \ref{prop:trunc} for any $\tau>0,\ z\in[0,1]$.
\begin{align*}
\del_\tau G[h](\tau,z)\lesssim&\,a(T(\tau))^{4-d}\,F[\rho](\tau,z)\,\sqrt{G[h](\tau,z)}+a(T(\tau))^{4-d}\,\langle\tau\rangle\,F[\rho](\tau,z)\,\del_zG[h](\tau,z)\\
\lesssim&\,F[\rho](\tau,z)\,\sqrt{G[h](\tau,z)}+\langle\tau\rangle\,F[\rho](\tau,z)\,\del_zG[h](\tau,z)
\end{align*}
In particular, for sufficiently small $\epsilon>0$ and assuming
\begin{equation}\label{eq:F-cond-4+}
\tilde{F}[\rho](\tau)\lesssim \sqrt{\epsilon} \,\langle\tau\rangle^{-2-\delta}\,
\end{equation}
one has
\[\del_\tau\sqrt{\tilde{G}[h](\tau)}\lesssim \tilde{F}[\rho](\tau)\,.\]

Finally, since the scale factor also only appears in the nonlinear source \eqref{eq:source-4+} as $a(T(\tau))^{4-d}$, which is uniformly bounded, the pointwise estimates on $\tilde{F}[S]$ as in Lemma \ref{lem:pw-est-source} essentially extend after replacing $\beta$ and $\beta^\prime$ with $0$. To ensure summability over all case distinctions in the proof, we are required to choose $\sigma$ as in Proposition \ref{prop:trunc}. After this, proving the bootstrap improvement goes through entirely as in Section \ref{sec:nonlin}, respectively \cite[Section 4]{GNR21}, up to using Corollary \ref{lem:pw-damp-trunc} in the place of Corollary \ref{cor:lin-damping-gevrey-concrete}.

\subsection{On attractive particle interaction for $d=3$}\label{subsec:gravit}

When studying gravitational particle interaction for the classic Vlasov-Poisson system on $\T^3_{L}$, where $L>0$ is the length of the torus, the Penrose condition shows that whether damping occurs depends on the size of the torus. For example, for a Maxwellian distribution of the form
\[\mu(v)=\frac{\rho_0}{(2\pi\,T)^\frac32}\,e^{-\frac{\lvert v\rvert^2}{2T}}\]
for $T>0$, the Penrose condition becomes
\[L<L_J=\sqrt{\frac{4\,T}{\rho_0}}.\]
Here, $L_J$ is referred to as the Jeans length, see also \cite[Example 2.3]{MV10}. Thus, stability holds if the equilibrium is sufficiently spread out compared to the size of the torus, while for $L>L_J$, one expects Jeans instability, see \cite[Section 5.1]{BT08}.\\

In the expanding setting, the side length of the torus grows with $a(t)$ by construction, and consequently, we must expect that phase mixing only occurs on some finite interval, after which the background becomes too spread out. For example, in the Maxwellian case, gravitational collapse is thus expected to occur after $a(t_{cut})L=L_J$; see also \cite[Chapter II, Section 16]{Pee80}. Indeed, one can explicitly see that our argument breaks down in the gravitational setting even for the Poisson equilibrium when attempting to obtain an analogue of Proposition \ref{prop:res-bd}. Here, one needs to analyse the Volterra integral equation family
\[\phi_k(\tau)=s_k(\tau)+4\pi\int_0^\tau \hat{M}_k(\tau-\tilde{\tau})\,a(T(\tilde{\tau}))\,\phi_k(\tilde{\tau}))\,d\tilde{\tau}\]
and, thus, if we define $u_{\tilde{\tau}}$ analogously, it would need to satisfy the ODE
\[u^{\prime\prime}_{\tilde{\tau}}(\tau)-a(T(\tilde{\tau}))\,u(T(\tilde{\tau}))=16\pi^2\,a(T(\tau))\,a(T(\tilde{\tau}))\,(\tau-\tilde{\tau})\,.\]
By the Liouville-Green approximation, see \cite[Theorem 2.01]{Olw74} and Appendix \ref{sec:ode}, the fundamental solutions $u^{1,2}_{\tilde{\tau}}$ are of the form
\[u^{1,2}_{\tilde{\tau}}(\tau)=a(T(\tau))^{-\frac14}\,\left[\exp\left(\pm\int_{\tilde{\tau}}^\tau a(T(y))^\frac12\,dy\right)+\epsilon\left(\int_{\tilde{\tau}}^\tau a(T(y))^\frac12\,dy\right)\right]\]
For power law expansion with $q\in(0,\sfrac12)$, one has $a(T(\tau))\gtrsim \langle \tau\rangle^{\tilde{\beta}}$, and thus $u^{1,2}_{\tilde{\tau}}(\tau)$ must grow superexponentially in $\tilde{\tau}$, which then also applies to $u_{\tilde{\tau}}$ and $R_k(\tau,\tilde{\tau})$. This prohibits an estimate of the type of Corollary \ref{cor:lin-damping-gevrey-concrete}, and thus indicates that Landau damping does not occur. Conversely, for power law expansion with $q>\sfrac12$, the range of $\tau$ is finite and the solution cannot fully homogenize. 
%\begin{remark}[On the resolvent kernel for gravitational interactions]\label{rem:lin-concrete-grav}
%When considering the gravitational expanding Vlasov-Poisson system, the Volterra equation \eqref{eq:volterra} would be replaced by one with the kernel
%\[K_k(\tau,\tilde{\tau})=8\pi\,(\tau-\tilde{\tau})\exp(\theta_0\,\lvert k\rvert\,(\tau-\tilde{\tau}))\,a(T(\tilde{\tau}))\,.\]
%In particular, one can in principle proceed as in Proposition \ref{prop:res-bd}, which leads to the ODE
%\[u_{\tilde{\tau}}(\tau)-a(T(\tau))\,u_{\tilde{\tau}}=64\pi^2 a(T(\tau))\,a(T(\tilde{\tau}))(\tau-\tilde{\tau})\,.\]
%Applying the Liouville-Green approximation to this equation (see, for example, \cite[Chapter 6, Theorem 2.1]{Olw74}), then leads to one of the two fundamental solutions of the homogeneous equations growing as
%\[a(T(\tau))^{-\frac14}\,\exp\left(\int_{\tilde{\tau}}^\tau a(T(r))^\frac12\,dr\right)\,,\]
%and thus also the true particular solution satisfying the boundary conditions at $\tau=\tilde{\tau}$. Thus, both $u_{\tilde{\tau}}(\tau)$ and $R_k(\tau,\tilde{\tau})$ exhibit superexponential growth, prohibiting an estimate as in Corollary \ref{cor:lin-damping-gevrey-concrete}.
%\end{remark}
%

\section{The Liouville-Green approximation}\label{sec:ode}
In the following, we collect some standard results on the Liouville-Green approximation as they are relevant for the oscillatory ODE appearing in Section \ref{sec:lin}. In this section, $I=(b_-,b_+)$ is a real interval where $-\infty\leq b_-< b_+\leq \infty$.

\begin{lemma}[Liouville-Green approximation with error bounds]\label{lem:LG-hom}
Let $a\in C^2(I)$ be positive, and define
\[F(x):=\int_{b_-}^x a(y)^{-\frac14}\,\left(a^{-\frac14}\right)^{\prime\prime}(y)\,dy\,.\]
Further, let $\mathcal{V}_{b_-,x}[F]$ denote the total variation of $F$ on $(b_-,x)\subseteq I$, and define $\xi_a:I\longrightarrow (0,\infty)$ by
\[\xi_a(x)=\int_{b_-}^x a(y)^\frac12\,dy\,,\]
with its inverse denoted by $\xi_a^{-1}$. Then, solutions to the complex-valued differential equation
\begin{equation}\label{eq:osc-ode-hom}
w^{\prime\prime}(x)+a(x)\,w(x)=0
\end{equation}
are of the form
\begin{align*}
w_{C,\varpi}(x)=C\,a(x)^{-\frac14}\left[\sin\left(\xi_a(x)+\varpi\right)+(\epsilon\circ\xi_a)(x)\right]
\end{align*}
for $C\in\C,\,\varpi\in\R$, where the error function $\epsilon$ satisfies
\begin{equation}\label{eq:epsilon-no-xi}
\lvert \epsilon(\zeta)\rvert,\, \left\lvert \left((a\circ \xi_a^{-1})^{-\frac12}\,\epsilon^\prime\right)(\zeta)\right\rvert \leq \exp\left[\int_0^{\zeta} \left(a^{-\frac34}\,\lvert (a^\frac14)^{\prime\prime}\rvert\right)(\xi_a^{-1}(\tilde{\zeta}))\,d\tilde{\zeta}\right]-1
\end{equation}
and thus
\begin{equation}\label{eq:epsilon-xi}
\lvert (\epsilon\circ\xi_a)(x)\rvert,\, \lvert a(x)^{-\frac12}\,(\epsilon\circ \xi_a)^\prime(x)\rvert \leq \exp\left(\mathcal{V}_{b_-,x}[F]\right)-1\,.
\end{equation}
In particular, if $F$ has bounded variation on $I$, any solution to \eqref{eq:osc-ode-hom} is uniformly bounded by $a^{-\frac14}$ up to a uniform multiplicative constant.
\end{lemma}
\begin{proof}
This result is largely contained in \cite[Chapter 6, Theorem 2.02 and Exercise 2.5]{Olw74}, with the exception of \eqref{eq:epsilon-no-xi}. This follows from \eqref{eq:epsilon-xi} by variable transformation, but the proof therein even derives this to conclude that \eqref{eq:epsilon-xi} holds, see \cite[Chapter 6, (2.18)]{Olw74} for the analogous estimate when considering $w^{\prime\prime}(x)-a(x)w(x)=0$.
\end{proof}

If $F$ has bounded variation on $I$ and $a(x)\to a_-\neq 0$ as $x\to b_-$, then $(\epsilon\circ\xi_a)(x)$ and $(\epsilon\circ\xi_a)^\prime(x)$ converge to zero as $x\to b_-$. Thus, we then obtain the linearly independent fundamental solutions
\begin{align*}
w_1(x)=a(x)^{-\frac14}\left[\sin\left(\xi_a(x)\right)+(\epsilon\circ \xi_a)(x)\right]\,,\\
w_2(x)=a(x)^{-\frac14}\left[\cos\left(\xi_a(x)\right)+(\epsilon\circ\xi_a)(x)\right]\,.
\end{align*}
With these, we can describe solutions to the corresponding inhomogeneous problem as follows.
\begin{corollary}[Liouville-Green approximation for inhomogeneous equations]\label{cor:LG-inhom}
Let $a$ and $F$ be as above and consider
\begin{equation}\label{eq:osc-ode-inhom}
u^{\prime\prime}(x)+a(x)\,u(x)=q(x)
\end{equation}
for a source $f\in C(I)$, and let $w_{1,2}$ be the fundamental solutions above with Wronskian
\[W(x):=(w_1w_2^\prime-w_2w_1^\prime)(x)\,.\]
Then, one has $W\equiv -1$ and solutions to \eqref{eq:osc-ode-inhom} are given by
\begin{equation}\label{eq:ode-osc-inhom-sol}
u(x)=C_1\,w_1(x)+C_2\,w_2(x)-\int_{b_-}^x \left[w_1(y)w_2(x)-w_2(y)w_1(x)\right]\,q(y)\,dy
\end{equation}
for arbitrary $C_{1,2}\in\C$. In particular, if $\mathcal{V}_{b_-,\infty}[F]<\infty$, one has
\begin{equation}\label{eq:ode-osc-inhom-est}
\lvert u(x)\rvert\lesssim a(x)^{-\frac14}\left(1+\int_{b_-}^x a(y)^{-\frac14}\,\lvert q(y)\rvert\,dy\right)\,.
\end{equation}
\end{corollary}
\begin{proof}
When the Wronskian does not vanish, $w_1$ and $w_2$ span the solution space of \eqref{eq:osc-ode-hom}, see \cite[Chapter 5, Theorem 1.02]{Olw74}. Furthermore, the expression
\[\int_{b_-}^x W(y)^{-1}\left[w_1(y)w_2(x)-w_2(y)w_1(x)\right]\,f(y)\,dy\]
clearly is a solution to \eqref{eq:osc-ode-inhom} if $W$ does not vanish on $I$ since the first derivative of this expression vanishes, while the second reads
\[W(x)^{-1}\left[w_1(x)w_2^\prime(x)-w_2(x)w_1^\prime(x)\right]\,f(x)=W(x)^{-1}\,W(x)\,f(x)=f(x)\,.\]
Thus, we need to show $W\equiv -1$. To this end, observe that
\begin{align*}
w_1^\prime(x)=&\,-\frac14\,\frac{a^\prime(x)}{a(x)}w_1(x)+a(x)^{\frac14}\left[\cos(\xi_a(x))+a(x)^{-\frac12}\,(\epsilon\circ \xi_a)^\prime(x)\right]
\end{align*}
with an analogous expression for $w_2^\prime$. Consequently, one obtains
\begin{align*}
W(x)=&\,-\frac14\,\frac{a^\prime(x)}{a(x)}\left[w_1(x)w_2(x)-w_2(x)w_1(x)\right]\\
&\,+\left[-\sin\left(\xi_a(x)\right)^2\,+\sin\left(\xi_a(x)\right)\left(a(x)^{-\frac12}(\epsilon\circ\xi_a)^\prime(x)-\epsilon(\xi_a(x))\right)+a(x)^{-\frac12}\,\epsilon(\xi_a(x))\,(\epsilon\circ\xi_a)^\prime(x)\right]\\
&\,-\left[\cos\left(\xi_a(x)\right)^2+\cos\left(\xi_a(x)\right)\left(a(x)^{-\frac12}(\epsilon\circ\xi_a)^\prime(x)+\epsilon(\xi_a(x))\right)+a(x)^{-\frac12}\,\epsilon(\xi_a(x))\,(\epsilon\circ\xi_a)^\prime(x)\right]\\
=&\,-1+a(x)^{-\frac12}(\epsilon\circ\xi_a)^\prime(x)\big(\sin\left(\xi_a(x)\right)-\cos\left(\xi_a(x)\right)\big)-\epsilon(\xi_a(x))\,\big(\sin\left(\xi_a(x)\right)+\cos\left(\xi_a(x)\right)\big)\,.
\end{align*}
Since $\epsilon\circ\xi_a$ and $a(x)^{-\frac12}\,(\epsilon\circ \xi_a)^\prime(x)$ vanish as $x\downarrow b_-$, the Wronskian converges to $-1$ as $x\downarrow b_-$. On the other hand, the Wronskian is clearly constant since $w_1$ and $w_2$ solve \eqref{eq:osc-ode-hom} and hence 
\[W^\prime=w_1\,w_2^{\prime\prime}-w_2\,w_1^{\prime\prime}=w_1\cdot a\,w_2-w_2\cdot a\,w_1\equiv 0\,.\]
Consequently, it must be equal to $-1$ everywhere. This has shown \eqref{eq:ode-osc-inhom-sol}, and \eqref{eq:ode-osc-inhom-est} follows directly from it by applying the bound on $\epsilon\circ\xi_a$ in Lemma \ref{lem:LG-hom}.
\end{proof}

\printbibliography

\end{document}